\documentclass[letterpaper]{amsart}

\usepackage[stretch=10,shrink=10]{microtype}
\usepackage{amsfonts,amscd,amssymb,amsmath,amsthm,amsaddr}
\usepackage{mathtools}
\usepackage{graphicx,mathrsfs,tikz}
\usepackage{enumerate}
\usepackage{constants}
\usepackage{fullpage}

\usepackage{color}

\usetikzlibrary{matrix}
\usetikzlibrary{calc}
\usetikzlibrary{decorations.markings}
\usetikzlibrary{decorations.pathmorphing}

\usepackage{xparse}
\usepackage{stmaryrd}

\tikzset{vert/.style={circle,fill,inner sep=0,
              minimum size=0.15cm,draw,outer sep=0}}

\let\oldtocsubsection=\tocsubsection
\renewcommand{\tocsubsection}[2]{\hspace*{1.1cm}\oldtocsubsection{#1}{#2}}
\begin{document}

\newtheorem{theorem}{Theorem}[section]
\newtheorem{lemma}[theorem]{Lemma}
\newtheorem{corollary}[theorem]{Corollary}
\newtheorem{conjecture}[theorem]{Conjecture}
\newtheorem{cor}[theorem]{Corollary}
\newtheorem{proposition}[theorem]{Proposition}
\theoremstyle{definition}
\newtheorem{coupling}[theorem]{Coupling}
\newtheorem{definition}[theorem]{Definition}
\newtheorem{assumption}{Assumption}
\newtheorem{example}[theorem]{Example}
\newtheorem{claim}[theorem]{Claim}
\newtheorem{remark}[theorem]{Remark}

\Roman{assumption}


\newcommand{\Def}{:=}
\newcommand{\R}{\mathbb{R}}
\newcommand{\Z}{\mathbf{Z}}
\newcommand{\X}{\mathbf{X}}
\newcommand{\abs}[1]{\left| {#1} \right|}
\newcommand{\floor}[1]{{\left\lfloor #1 \right\rfloor}}
\newcommand{\eps}{\epsilon}
\newcommand{\norm}[1]{\left\| {#1} \right\|}
\newcommand{\divides}{\mid}
\newcommand{\tr}{\mathop{\mathrm{tr}}}
\newcommand{\omitted}{\bullet}


\newcommand{\expect}{\mathbb{E}}
\newcommand{\Exp}{\mathbb{E}}
\newcommand{\E}{\mathbb{E}}
\newcommand{\prob}{\Pr}
\newcommand{\Prob}{\prob}
\renewcommand{\P}{\prob}
\newcommand{\1}{\mathbf{1}}
\newcommand{\one}[1]{\1\{ {#1} \}}
\newcommand{\Ent}{\operatorname{Ent}}
\newcommand{\Var}{\operatorname{Var}}
\newcommand{\dtv}{d_{TV}}
\newcommand{\dw}{d_{W}}
\newcommand{\drawnfrom}{\sim}
\newcommand{\cov}{\mathop{\mathbf{Cov}}\nolimits}
\newcommand{\lawof}{\mathscr{L}}
\newcommand{\weakto}{\Rightarrow}
\newcommand{\probto}{\overset{\mathbb{P}}{\to}}

\newcommand{\Unif}{\operatorname{Unif}}
\newcommand{\Poisson}{\operatorname{Poisson}}
\newcommand{\Binom}{\operatorname{Binom}}
\newcommand{\Geom}{\operatorname{Geom}}
\newcommand{\Bernoulli}{\operatorname{Bernoulli}}


\newcommand{\ff}[2]{\left[ {#1} \right]_{#2}}
\newcommand{\dff}[2]{\left\llbracket {#1} \right\rrbracket_{#2}}
\newcommand{\edgesetof}{\mathcal{E}}


\newcommand{\umd}{\mathscr{G}_{n,d}}
\newcommand{\pmd}[1][d]{\mathscr{P}_{n,#1}}
\newcommand{\phammd}{\mathscr{T}_{n,d}}
\newcommand{\phamcondmd}{\mathscr{T}_{n,d}^*}
\newcommand{\phamrv}{\tilde{P}}
\newcommand{\hams}{ H_n }
\newcommand{\cycspace}[1][k]{\ensuremath{\mathcal{J}_{#1}}}
\newcommand{\cycprocess}[1][r]{\ensuremath{\mathbf{I}_{#1}}}
\newcommand{\rnd}[1][n]{  f_{#1} }
\DeclareDocumentCommand{\rndcyc}{ O{r} O{n} }{ \ensuremath{ {f_{#1,#2}}} }
\DeclareDocumentCommand{\rndpf}{ O{r} O{n} }{ \ensuremath{ {\mathcal{Y}_{#1,#2}} }}

\newcommand{\phamcycmean}[1][\alpha]{\mu_{ #1 } }
\newcommand{\pcycmean}[1][\alpha]{\lambda_{ #1 } }
\newcommand{\phamcycpf}{\tilde{\Z} }
\newcommand{\pcycpf}{\Z }

\newcommand{\limvar}[1][r]{ V^{({#1})}}

\newcommand{\cdp}{\varrho}
\newcommand{\cdpgf}{\ensuremath{\mathcal{P}}}
\newcommand{\cdproot}{\tau}

\newcommand{\fullcycspace}{\mathcal{J}}

\newcommand{\mixedcycspace}[1][k]{\mathcal{H}_{#1}}

\newcommand{\fullmixedcycspace}{\mathcal{H}}

\newcommand{\mixedcycprocess}{\mathbf{I}_r(\rHam,\rP)}

\newcommand{\rP}{Q}

\newcommand{\rHam}{\ensuremath{H}}

\newcommand{\rHammd}[1][n]{\mathscr{S}_{#1}}


\newcommand{\sepmixedmd}{\mathscr{S}_{n,d}}


\newcommand{\cvberror}{\ensuremath{\varepsilon}}
\newcommand{\cvbpf}{\mathbf{W}}
\newcommand{\cvbvar}[1][\alpha]{ W_{ #1 } }
\newcommand{\cvbeps}{\varepsilon}

\newcommand*\mystrut[1]{\vrule width0pt height0pt depth#1\relax}

\newcommand{\Hampoly}[1][d,n]{\Pi_{#1}}

\newcommand{\chainfield}{\mathcal{F}}
\newcommand{\ccoeff}{\theta}
\newcommand{\indcoeff}{\beta}



\newcommand{\nvec}{\phi}
\newcommand{\nvspf}{Z_\ell}
\newcommand{\nvpf}{Z_\phi}
\newcommand{\nvserror}{\xi_n}

\newcommand{\nvmc}[1]{X_{ {#1} }}
\newcommand{\nvmcrnd}{\rho}
\newcommand{\nvmcpi}{\pi}

\newcommand{\nvstein}{f_h}
\newcommand{\nvsteinsigma}{\mathcal{F}}

\newcommand{\lplus}{\cdproot_{+}}


\newcommand{\cycvar}[2]{X_{#1,#2}}
\DeclareDocumentCommand{\cycsig}{ O{r} O{n} }{ \ensuremath{ {\mathscr{F}_{#1,#2}}} }

\newcommand{\tverror}[3]{\varepsilon^{#1}({#2,#3})}
\newcommand{\rndfun}[3]{{f^{#1}_{#2,#3}}}

\title{Quantitative Small Subgraph Conditioning}
\author{Tobias Johnson}
\address{Department of Mathematics, University of Southern California}
\email{tobias.johnson@usc.edu}
\author{Elliot Paquette}
\address{Faculty of Mathematics, Weizmann Institute of Science}
\email{elliot-andrew.paquette@weizmann.ac.il}
\thanks{%
The authors acknowledge partial support from 
NSF grant
DMS-0847661.
}
\date{May 22, 2015}

\begin{abstract} 
We revisit the method of small subgraph conditioning, used to
establish that random regular graphs are Hamiltonian a.a.s. We refine
this method using new technical machinery 
for random
$d$-regular graphs on $n$ vertices that hold not just asymptotically,
but for any values of $d$ and $n$. This lets us estimate how quickly
the probability of containing a Hamiltonian cycle converges to $1$,
and it produces quantitative contiguity results between different
models of random regular graphs. These results hold with $d$ held
fixed or growing to infinity with $n$. As additional applications, we establish
the distributional convergence of the number of Hamiltonian cycles
when $d$ grows slowly to infinity, and we prove that the number of
Hamiltonian cycles can be approximately computed from the graph's
eigenvalues for almost all regular graphs.\end{abstract}

\maketitle

\tableofcontents

\section{Introduction}

The \emph{uniform model} $\umd$ of random regular graph of degree $d$ on $n$ vertices is the setting for many celebrated theorems concerning discrete random structures, and much is known about it.  
For an excellent survey of random regular graphs, consider~\cite{WormaldSurvey}.  In a line of work going back to Fenner and Frieze~\cite{FennerFrieze84}, Bollob\'as~\cite{Bollobas83}, and Frieze~\cite{Frieze88}, it was settled finally by Robinson and Wormald~\cite{RobinsonWormald92, RobinsonWormald94} that a uniformly chosen $d$-regular graph was 
a.a.s.\ Hamiltonian as $n \to \infty$ for any fixed $d\geq 3$.  The techniques of Bollob{\'a}s, Fenner and Frieze are algorithmic, while the work of Robinson and Wormald
follows
a second-moment method approach together with what is known as \emph{small subgraph conditioning}.  The combined efforts of~\cite{FJMRN} show that there are many Hamiltonian cycles a.a.s.\ and produce an algorithm that finds them a.a.s.  

These results are concerned with holding $d$ fixed and letting $n$ tend to infinity.  An alternative is to allow $d=d(n)$ to vary with $n,$ possibly growing to 
infinity at some rate.  Along this line of reasoning, it is shown in~\cite{CFR} that there is a constant $c>0$ so that if $d_0 \leq d(n) \leq cn,$ the graph is Hamiltonian a.a.s.  By a different approach, it is shown in~\cite{KSVW} that if $d(n) \geq \sqrt{n} \log n$ then the graph is Hamiltonian a.a.s.  

All of these techniques are ultimately asymptotic, in the sense that they show a graph feature holds with some probability tending to 
$1$.  In this paper, we will show how the \emph{small subgraph conditioning method} can be used to produce estimates that do not just hold in the limit as $n \to \infty$ but hold for all $n$ and $d$ simultaneously.  In particular, we extend the method of small subgraph conditioning to the regime where $d=d(n)$ may grow to infinity, and we are principally interested in the regime in which $\log d / \log n \to 0.$

As with much work on the uniform model, we actually work with the \emph{configuration model} (or \emph{pairing model}) $\pmd$.  In this model, $nd$ balls are partitioned into $n$ bins of equal size, noting this 
requires 
$nd$ to be even.  A matching of all the balls is chosen uniformly at random, and then the balls in each bin are identified to form vertices.  All the edges are preserved in the identification to produce a $d$-regular pseudograph, 
which we call the \emph{projection} of the pairing.
For definiteness, we will refer to the balls as \emph{prevertices}, which are partitioned into $n$ vertex bins of size $d,$ and we will reserve typical graph nomenclature for the projected pseudograph.

Our central object of study is the number of Hamiltonian cycles in a random regular graph.
If $P$ is some pairing of $nd$ prevertices, we denote by $H_n(P)$ 
the number of Hamiltonian cycles in the projection of $P$ to a pseudograph. As in the theorem below,
we will often write simply $H_n$, indicating the distribution of its parameter underneath
a $\P$ or $\E$ symbol.

Our first theorem gives a bound on the probability that there are no Hamiltonian
cycles in the configuration and uniform models: 
\begin{theorem}
\label{thm:ham}
Suppose that $d=d(n)\geq 4$ satisfies
$\log d / \log n \to 0$. For every $\epsilon > 0,$ 
\[
\Pr_{\pmd}\left[ \hams = 0 \right] = O(n^{-1/3 + \epsilon}).
\]
If $d \geq 3$ and $d^2/ \log n \to 0,$ then for every $\epsilon >0,$
\[
\Pr_{\umd}\left[ \hams = 0 \right] = O(n^{-1/3 + \epsilon}).
\]
\end{theorem}
Note that when $d=3$, the theorem is not true for $\pmd.$  A self-loop anywhere in the graph obstructs the existence of Hamiltonian cycles, and the number of self-loops is asymptotically $\Poisson(1)$ (see Corollary~\ref{cor:tvbound}).
Whereas previous results show that these probabilities are $o(1),$ the novelty here is the establishment of a rate.   Previous work of~\cite{CFR} shows that for $d \geq d_0$ large, this probability is at most $O(n^{-2})$,
but the approach taken in that paper is unlikely to extend to all $d \geq 3.$
It remains an open question to determine the true rate, 
or even to determine if the rate decays as a power of $n$.  

\subsection{Contiguity}

After the initial developments by Robinson and Wormald, Janson further developed small subgraph conditioning~\cite{Jan95}
by using it to prove a property
known as \emph{contiguity}.  Two sequences of laws $P_n$ and $Q_n$ on a common measurable space are contiguous if 
for any sequence $A_n$ of measurable events,  $P_n(A_n) \to 0 \iff Q_n(A_n) \to 0,$ which is a sort of qualitative asymptotic equivalence between the two models.  Contiguity has proven useful in that it allows difficult estimates, such as Friedman's second eigenvalue 
bounds~\cite{Friedman}, to be made for a regular graph model of choice and then transferred to other models.  Contiguity alone, however, gives 
only an asymptotic estimate for the probabilities in one model based on the probabilities in the other.

Beyond generalizing small subgraph conditioning to growing $d,$ we seek to understand how precisely estimates for a probability in one random regular graph model transfer to another.  We will initiate this study for $\pmd$ and a second graph model pertinent to studying $\hams$ in $\pmd$. We define the model $\phammd$ that, as a pseudograph model, can be considered as a degree $d-2$ regular pseudograph induced from $\pmd[d-2]$ with a superimposed, independent and uniformly chosen Hamiltonian cycle.  At the pairing level, we define it by adding two prevertices to each bin of $\pmd[d-2],$ sampling a uniform matching of these new prevertices conditioned to project to a Hamiltonian cycle, and then randomizing the ordering of the prevertices in each bin so that they remain exchangeable. 
Formally, we consider both $\phammd$ and $\pmd$ as laws of pairings on $nd$ prevertices, and we refer to a \emph{pairing event} as any set of these pairings.  Further, the law of $\phammd$ is absolutely continuous with respect to $\pmd,$ and the Radon-Nikodym derivative is precisely $\hams / \Exp_{\pmd} \hams$,
\label{page:radonnikodym}
meaning that for any pairing $P_0$,
\begin{align*}
  \frac{\P_{\phammd}[\left\{ P_0 \right\}]}{\P_{\pmd}[\left\{P_0\right\}]} &= \frac{\hams(P_0)}{\E_{\pmd}\hams},
\end{align*}
where the measure on pairings is given under the $\P$ symbol. Equivalently,
\begin{align}
  \E_{\phammd}[f] = \E_{\pmd}\biggl[\frac{\hams}{\E_{\pmd}\hams}f\biggr]\label{eq:radonnikodymstatement}
\end{align}
for any function $f$ on pairings.

Small subgraph conditioning
actually shows that $\phammd$ and $\pmd$ are contiguous.  As $\phammd$ always has a Hamiltonian cycle, the consequence that 
$\pmd$ is Hamiltonian a.a.s.\ follows
directly from the contiguity of the models, which we extend to the case of growing $d.$

\begin{theorem}
\label{thm:contig}
Suppose that $d=d(n)$ satisfies $4\leq d \leq n^{\alpha_0-\epsilon}$ where $\alpha_0 = \frac{8}{3(8+\sqrt{2})} \approx 0.283$,
for some $\epsilon > 0$.
Then for any sequence of pairing events $A_n,$ 
\[
\Pr_{\phammd}\left[ A_n \right] \to 0 \iff
\Pr_{\pmd}\left[ A_n \right] \to 0. 
\]
If in addition $d \to \infty,$ then 
\(
\dtv(\phammd,\pmd) \to 0.
\)
\end{theorem}

By conditioning the pairings to project to simple graphs, these results can be transferred to the uniform model.  This requires that we introduce $\phamcondmd,$ which is $\phammd$ conditioned to project to a simple graph.  Note that on conditioning, $\hams$ is still the Radon-Nikodym derivative between $\phamcondmd$ and $\umd$ up to renormalization. We are not able to show this same sort of general contiguity statement for $\phamcondmd$ and $\umd$ when $d \to \infty.$  However, we do show that a certain type of quantitative contiguity does transfer.
\begin{theorem}
\label{thm:qcontig}
Suppose that $d \geq 4$ and $\log d / \log n \to 0,$ and suppose that $A_n$ is some sequence of pairing events.  Let $\alpha>0$ be fixed.  Then,
\[
\Pr_{\phammd}\left[ A_n \right] = O(n^{-\alpha})
\implies
\Pr_{\pmd}\left[ A_n \right] = O(n^{-\beta+\epsilon})~\forall \epsilon>0,
\]
where $\beta = \alpha \wedge \frac13.$
Likewise, 
\[
\Pr_{\pmd}\left[ A_n \right] = O(n^{-\alpha})
\implies
\Pr_{\phammd}\left[ A_n \right] = O(n^{-\beta+\epsilon})~\forall \epsilon>0,
\]
where $\beta = \alpha \wedge \frac13.$  In particular 
\[
\Pr_{\phammd}\left[ A_n \right] = O(n^{-1/3 + \epsilon})~\forall \epsilon>0
\iff
\Pr_{\pmd}\left[ A_n \right] = O(n^{-1/3 + \epsilon})~\forall \epsilon>0.
\]
If we additionally assume that $A_n$ consists only of pairings that project to simple graphs, then we may assume $d \geq 3.$
The same results hold with $\phammd$ replaced by $\phamcondmd$ and $\pmd$ replaced by $\umd$ for $3 \leq d = o(\sqrt{\log n}).$
\end{theorem}
\noindent Note that Theorem~\ref{thm:ham} is an immediate consequence of Theorem~\ref{thm:qcontig}. 

\subsection{Other applications}

The machinery developed here has further applications beyond the contiguity results.  In \cite{Jan95}, the limiting distribution of $\hams$ is derived for $d$ fixed and $n\to \infty.$  We can derive the distributional convergence of $\hams$ in the $d \to \infty$ regime.  As expected, its logarithm is asymptotically normal.
\begin{theorem}
\label{thm:lognorm}
If $d \to \infty$ slowly enough that $\log d/ \log n \to 0,$ then with $P \sim \pmd,$
\[
\frac{\log \hams(P) - \log \Exp \hams(P)}{\sqrt{2/d}} \weakto N(0,1).
\]
\end{theorem}
In fact, it can be seen that $\hams/\Exp \hams$ is well-approximated by a multiple of the number of self-loops in the $d \to \infty$ regime, by virtue of which the normality follows. 
Better approximations for $\hams/\Exp \hams$ can be obtained by using more cycle information.  Also, in a sufficiently sparse regime, cycle counts can be computed from the graph's eigenvalues with high probability.
This allows the Hamiltonian cycle count to be well approximated 
by an explicitly computable trace, for almost all regular graphs:
\begin{theorem}\label{thm:evalhams}
  Suppose that $3\leq d\leq n^{1/12}$. There is a polynomial $\Hampoly(x)$,
  given in \eqref{eq:Hampoly}, such that
  \begin{align*}
    \Pr_{\pmd}\left[\abs{\frac{\hams}{\E\hams} - \exp(\tr\Hampoly(P))} > n^{-1/12} \right]
      &= O\big((\log n)^2n^{-1/6}\big),
  \end{align*}
  where $\tr\Hampoly(P) = \sum_{i=1}^n\Hampoly (\lambda_i / \sqrt{d-1})$, with $\lambda_1,\ldots,\lambda_n$
  the eigenvalues of the adjacency matrix of $P$.
\end{theorem}
%

\subsection{Small subgraph conditioning}

To introduce the method, we will sketch how small subgraph conditioning can be used to estimate $\Pr\left[ \hams = 0 \right]$ in the configuration model. All probabilities
and expectations in this section are taken with respect to $\pmd$. 
One possible first instinct is to apply the second moment method, but there is the unfortunate difficulty that $\Var \hams$ is the same order as $\left(\Exp \hams\right)^2$.  The miracle is that most of the variance can be understood as arising from short cycles.

Let $X_k$ be the number of cycles of length $k$ in the configuration model, so that 
\[
f_{r,n} \Def \Exp\left[
\frac{\hams}{\Exp \hams}
\middle\vert
X_1,X_2,\ldots,X_r
\right]
\]
is the Radon-Nikodym derivative of the cycle-count vector of $\phammd$ with respect to the cycle-count vector of $\pmd$ (we actually condition with respect to slightly more information in \eqref{eq:rndcyc} to prove our theorems).  It can be shown that the cycle counts in both models 
are asymptotically vectors of independent Poissons.  Thus from Fatou's lemma, one can calculate
\[
V_r \Def \liminf_{n\to\infty} \Var\left(
\Exp\left[
\frac{\hams}{\Exp \hams}
\middle\vert
X_1,X_2,\ldots,X_r
\right]
\right)
\]
solely using the limiting Poisson structure.  On the other hand, an explicit variance calculation shows that
\[
V_\infty \Def \lim_{n \to \infty} \frac{\Var \hams}{(\Exp \hams)^2}
 = \lim_{r \to \infty} V_r,
\]
which in a sense says that the two graph models $\phammd$ and $\pmd$ conditioned to have the same short cycle counts are asymptotically indistinguishable. Note that these limits being equal is not simply a question of reversing the order of the $r$ and $n$ limits; it asserts, moreover, that the cycle count $\sigma$-algebra asymptotically determines the Radon-Nikodym derivative.  

Then, for any $r$ and $\epsilon>0$ one has the bound 
\begin{align*}
\Pr\left[ \hams = 0 \right]
&\leq 
\Pr\left[ f_{r,n} \leq \epsilon \right]
+
\Pr\left[ \left| \frac{\hams}{\Exp \hams} - f_{r,n}\right| \geq \epsilon \right] \\
&\leq 
\Pr\left[ f_{r,n} \leq \epsilon \right]
+ \frac{ (\Var \hams) / (\Exp \hams)^2 - \Var f_{r,n}}{\epsilon^2},
\end{align*}
where the bound follows from Chebyshev's inequality and 
the following property of conditional expectations:
\begin{align*}
  \Var\bigl(X-\E[X\mid Y]\bigr) = \Var(X) - \Var\bigl( \E[X\mid Y] \bigr).
\end{align*} 
Taking the limit supremum,
\[
\limsup_{n\to \infty}
\Pr\left[ \hams = 0 \right]
\leq 
\lim_{n\to \infty}
\Pr\left[ f_{r,n} \leq \epsilon \right]
+ \frac{V_\infty - V_r}{\epsilon^2}.
\]
From the limiting Poisson structure of the cycle counts, this limiting probability 
 exists and has an explicit form, and it is now a calculation to choose $r$ and $\epsilon$ appropriately to make this bound as small as desired.

\subsection{Quantitative estimate}\label{subsec:quant_est}

We essentially follow the approach outlined above in the classic small subgraph conditioning method.  The two innovations necessary to produce a rate in this argument are a variance estimate of $\hams$ that holds for a large range of $n$ and $d$ simultaneously and an estimate on \emph{how} nearly Poisson are the cycle counts.  The remainder of the work is to make estimates of the conditioned Radon-Nikodym derivative using the Poisson approximations.

Because of the nature of our Poisson approximation, we will also change the $\sigma$-algebra used in the conditioning.  Roughly speaking, we will keep track of not only how many cycles appear but where they appear as well.  As always, we work in the pairing model.  By a cycle in a pairing $P$, we mean a collection of pairs that projects down to a cycle in the pseudograph.

Let $\cycspace[k]$ be the set of all possible cycles of length $k$ that
could appear in an instance of $\pmd$. 
We note that
$\abs{\cycspace}=\ff{n}{k}(d(d-1))^k/2k$, and that this holds even for
$k=1$ and $k=2$. 
 Let $\fullcycspace=\bigcup_{k=1}^r\cycspace$ for $r$ to be specified. 
For any $\alpha \in \fullcycspace$,
define $I_\alpha(P)=\one{\text{$P$ contains $\alpha$}}$. Further, let
$\cycprocess[r](P)=\bigl(I_\alpha(P),\,\alpha\in\fullcycspace\bigr)$.   
Our main approximation theorem says that $\cycprocess[r](P)$ is well-approximated
by a vector of independent Poissons for $P$ drawn from either $\pmd$ or $\phammd$.  These Poisson vectors have slightly different means, and this difference will ultimately account for the dominant term in the variance of $\hams$.  We will use $\pcycmean$ and $\phamcycmean$ to denote the approximate means of $I_\alpha(P)$ with $P$ drawn from $\pmd$ or $\phammd$ respectively.  So, we define, for $\alpha \in \cycspace[k],$
\begin{align}
\pcycmean &\Def \frac{1}{(nd)^k} \label{eq:pcycmean}\\
\phamcycmean &\Def \frac{1}{(nd)^k} + \frac{(-1)^k - 1}{(nd(d-1))^k}\label{eq:phamcycmean}.
\end{align}
Let $\pcycpf=\bigl(Z_\alpha,\,\alpha\in\fullcycspace\bigr)$  
be a vector whose coordinates are independent Poisson random variables with $\E Z_\alpha=\pcycmean$ for $\alpha\in\cycspace[k]$.  
Likewise, let 
 $\phamcycpf=\bigl(\tilde{Z}_\alpha,\,\alpha\in\fullcycspace\bigr)$ 
be a vector whose coordinates are independent Poisson random variables with 
\(
\E \tilde{Z}_\alpha=\phamcycmean
\)
for $\alpha\in\cycspace[k]$.

%
%
A typical distributional approximation theorem between
$\cycprocess[r](P)$ and $\pcycpf$ might be given as a bound between their laws in some probability metric, such as the total variation distance.  This is not quite enough for all of our purposes.  
Especially when it comes to estimating the variance of the conditional Radon-Nikodym derivative (see Lemma~\ref{lem:cvb}), we need better control over the point probabilities of the law of $\cycprocess[r](P)$ for a pairing $P$ from either $\pmd$ or $\phammd$.  Thus by modifying 
standard Stein's method machinery, we seek to show that for a fixed $\{0,1\}$-vector $x$  encoding
a configuration of cycles, 
\begin{align*}
\frac{
\Pr\left[ 
\cycprocess[r](P) = x
\right]
}{
\Pr \left[ 
\pcycpf = x
\right]
}\approx 1 \text{ for $P\sim\pmd$}, \qquad\text{ and }\qquad
\frac{
\Pr\left[ 
\cycprocess[r](P) = x
\right]
}{
\Pr \left[ 
\phamcycpf = x
\right]
}\approx 1 \text{ for $P\sim\phammd$}.\end{align*}
We are not able to do this for all $x$: indeed, $\Pr [ 
\cycprocess[r](P) = x
]$ is zero for some choices of~$x$. We restrict ourselves to vectors 
$x$ representing configurations of not too many cycles, none of which overlap. 
Specifically, we estimate 
the ratio for cycle configurations that are \emph{strictly $\lambda$-neat}, as defined below.
  \begin{definition}\label{def:strictlneat}
    For some $\lambda\geq 1$, 
    a $\{0,1\}$-vector $x=(x_\alpha,\,\alpha\in\fullcycspace)$ is strictly $\lambda$-neat
    if the following hold:
    \begin{enumerate}[i)]
      \item If $x_\alpha=x_\beta=1$ for any $\alpha,\beta\in\fullcycspace$, 
        then $\alpha$ and $\beta$ do not share a vertex in the graph projection.
      \item The total number of prevertices contained in $x$, given by
        \begin{align*}
          \sum_{k=1}^r\sum_{\alpha\in\cycspace[k]}2kx_{\alpha},
        \end{align*}
        is at most
         $\lambda(d-1)^r$.
    \end{enumerate}
  \end{definition}
  
  We now present our Poisson approximation theorem.
  \begin{proposition}\label{prop:multpoiapprox}
    There is an absolute constant $\Cr{C:mpa}$ such that
    for all strictly $(\log n)$-neat $x$ and all
    $d\geq 3$, $r\geq 4$, and $n$ satisfying
    $\Cr{C:mpa}(\log n)^2(d-1)^{2r-1}<n/2$,
    \begin{align}
      \abs{\frac{\P[\cycprocess(P) = x]}{\P[\pcycpf=x]} -1} &\leq
        \frac{\Cl{C:mpa}(\log n)^2(d-1)^{2r-1}}{n}
         &  \text{for } P&\drawnfrom\pmd,
        \label{eq:pmdapprox}\\
        \intertext{and}
      \abs{\frac{\P[\cycprocess(\phamrv) = x]}{\P[\phamcycpf=x]}-1} &\leq 
        \frac{\Cr{C:mpa}(\log n)^2(d-1)^{2r-1}}{n}
         & \text{for }\phamrv&\drawnfrom\phammd
        \label{eq:phammdapprox}.
    \end{align}
  \end{proposition}
  To prove this proposition, we develop a variation on Stein's method.
  The approach is similar to the method of size-bias couplings
  for Poisson approximation
  expounded in \cite{BHJ}, and it also has much in common with
  the method of switchings used to derive point probability estimates in~\cite{MWW}
(see \cite[Section~2.4]{WormaldSurvey} for a good, general introduction to the method of 
switchings).
  We discuss our technique more in Section~\ref{subsec:multbounds}.

We also bound the probability that $\cycprocess(P)$ or $\cycprocess(\tilde P)$ is not strictly $(\log n)$-neat to be of order $(d-1)^{2r}/n$ (see Proposition~\ref{prop:strictlneat}).  As a consequence, we can derive total variation bounds.
\begin{corollary} \label{cor:tvbound}
    There is an absolute constant $\Cr{C:mptv}$ such that
    for all
    $d\geq 3$ and $r\geq 4$
    \begin{align}
      \dtv(\cycprocess(P), \pcycpf) &\leq
        \frac{\Cl{C:mptv}(d+(\log n)^2)(d-1)^{2r-1}}{n}
         & \text{for } P&\drawnfrom\pmd,
        \label{eq:pmdtv}\\
        \intertext{and}
      \dtv(\cycprocess(\phamrv), \phamcycpf) &\leq
        \frac{\Cr{C:mptv}(d+(\log n)^2)(d-1)^{2r-1}}{n}
         & \text{for }\phamrv&\drawnfrom\phammd.
        \label{eq:phammdtv}
    \end{align}
\end{corollary}  
To go with our quantitative Poisson approximations, we will need a quantitative estimate
of the second moment of $\hams$.
Let $\rnd \Def \tfrac{\hams}{\Exp_{\pmd} \hams}$, which
 as explained on p.~\pageref{page:radonnikodym}  
is the Radon-Nikodym derivative of $\phammd$ with respect to $\pmd$.  
We show with an error bound that the second moment of $\rnd$ is approximately $d/(d-2)$
when $d$ is not too large:
\begin{proposition}
\label{prop:Variance}
For any $\alpha$ with $1 \leq \alpha < 8/\sqrt{2},$ there is a constant $M_\alpha$ so that for any $d \leq n^{1/2}/\log n$ 
\[
\Exp_{\pmd}\bigl[\rnd^2\bigr] \leq \frac{d}{d-2}+ M_\alpha\frac{d^{\frac32(1+1/\alpha)}}{\sqrt{n}}.
\]
\end{proposition}

This differs from previous work such as \cite{FJMRN} in that we develop a bound that works for a range of $d$ and $n$ simultaneously.  Our methodology differs significantly from their work, and we develop a semi-probabilistic technique for making the comparison.  We show that there is a law $\phi$ of two-colorings of the edges of a cycle so that
\[
\Exp_{\pmd}\bigl[\rnd^2\bigr] \approx \Exp_\phi \sqrt{ \frac{d}{d-2} } \exp( Z_n^2/d),
\]
where $Z_n$ is a statistic of the two-coloring that is approximately standard normal.  
This gives an interpretation of why
a Gaussian integral appears in the purely combinatorial variance calculation of \cite{FJMRN}, as well
as in other applications of the small subgraph conditioning method. 
The measure $\phi$ has the form of a \emph{factor model} or \emph{graphical model} (see~\cite{DemboMontanari} for an overview of the general theory of these objects).  We then show that $\phi$ is very nearly that joint law on $\{0,1\}^n$ that would come from a $2$-state Markov chain $\pi,$ so that $\Exp_\phi \exp( Z_n^2/d) \approx \Exp_\pi \exp(Z_n^2/d)$.  Under this law, $Z_n$ can be understood as a centered, scaled additive functional on the Markov chain.  Thus, we are able to approximate it 
very precisely by a Gaussian using size-bias coupling (see Appendix~\ref{sec:markov}).  We then compare these expectations by a modification of Stein's method suitable to comparing expectations of test functions of the form $h(x) = e^{a x^2}$ for positive $a$.   

\subsection{Organization}
This paper is organized into four sections and one appendix.  We begin in Section~\ref{sec:supporting} with some preliminary calculations and lemmas that are useful throughout the 
paper.  
In Section~\ref{sec:poisson} we prove the multiplicative Poisson bound Proposition~\ref{prop:multpoiapprox} and estimates for the number of
$(\log n)$-neat graphs.
In Section~\ref{sec:variance} we prove the variance bound Proposition~\ref{prop:Variance}; we do not include the Markov chain estimates here.    In Section~\ref{sec:proof} we prove the main theorems using the tools developed.  Finally, we include Appendix~\ref{sec:markov} in which we prove precise estimates for a $2$-state Markov chain.

\subsection{Notation}
Here and throughout the paper, we use the $O(\cdot), \Omega(\cdot), \Theta(\cdot)$ to mean something stronger than their usual meaning.  We always mean the implied constants are independent of $d,n,$ and $r,$ and that these bounds hold for all $d,n,$ and $r$ in the ranges considered. 

We will also make use of the following falling factorial notation.  We let $\ff{a}{b}$ be the usual falling factorial $\ff{a}{b} = a(a-1)\cdots (a-b+1)$, with $\ff{a}{0}=1$.  We also use the double falling factorial $\dff{a}{b}$ (in analogy with double factorial), which is useful for describing combinatorial quantities arising from matchings.  This is given by $\dff{a}{b} = (a-1)(a-3)\cdots(a-2b+1)$, 
with the caveat that instead of $\dff{2n-1}{n}=0$, we let 
$\dff{2n-1}{n} = \dff{2n-1}{n-1} = (2n-2)!!$  
(this exact expression comes up in Section~\ref{sec:variance}).  
We additionally use the notation $[n]$ to mean the integers $\{1,2,\ldots, n\},$ noting that the falling factorial always has a subscript.

\section{Supporting tools}
\label{sec:supporting}

Here we collect some important technical tools we will use throughout the paper.  
We frequently need to make calculations of statistics
computable in terms of $2$-colorings of a cycle.  Thus, we find some explicit 
expressions for polynomials that can be used to do these calculations.  
Consider edge coloring a cycle $\mathcal{C}_k$ on $k$ vertices by two colors $\{R,B\}$.  Choose an orientation for $\mathcal{C}_k,$ and let $r_1$ denote the number of vertices with an incoming $R$ edge and an outgoing $B$ edge.  Let $r_2$ denote the number of vertices with two incident $R$ edges, and let $b_2$ denote the number of vertices with two incident $B$ edges.  Note that all of these statistics are independent of the orientation chosen.  
Let $\edgesetof(\mathcal{C}_k)$ denote the edge set of $\mathcal{C}_k$, and
define
\begin{align}
\cdp_{k}(a,b,c) = \sum_{f \colon \edgesetof(\mathcal{C}_k) \to \{R,B\} } a^{r_2} b^{r_1} c^{b_2}.\label{eq:cdpdef}
\end{align}
Here and in the following sums, the statistics $r_1$, $r_2$, and $b_2$ refer to the coloring $f$.
Further, define
\begin{align}
p_k(a,b) = \cdp_{k}(a,b,1) = \sum_{f \colon \edgesetof(\mathcal{C}_k) \to \{R,B\} } a^{r_2}b^{r_1}.\label{eq:pkdef}
\end{align}
Note that $k = r_2 + b_2 + 2r_1,$ and hence $\cdp_{k} = c^kp_k(a/c,b/c^2)$.  Thus, it suffices to compute $p_k$.

To compute $p_k,$ we will break the cyclic structure and instead work with analogous polynomials with respect to colorings of the directed path $\mathcal{P}_{k+2}$ on $k+2$ vertices.  We will identify $\edgesetof(\mathcal{P}_{k+2})$ with $[k+1]$ in the natural way and define
\begin{align*}
p_k^{R}(a,b) &= \sum_{\substack{f \colon \edgesetof(\mathcal{P}_{k+2}) \to\{R,B\} \\ f(1) = f(k+1) = R}} a^{r_2}b^{r_1},
  \\
p_k^{B}(a,b) &= \sum_{\substack{f \colon \edgesetof(\mathcal{P}_{k+2}) \to\{R,B\} \\ f(1) = f(k+1) = B}} a^{r_2}b^{r_1}.
\end{align*}
By identifying the first and last edges of this path, we have that
\[
p_k(a,b)= p_k^{R}(a,b) + p_k^{B}(a,b).
\]
Beyond their use for computing $p_k(a,b),$ these polynomials are also needed in Section~\ref{sec:poisson}.

    We use the method of the transfer matrix to find generating
    functions for these expressions.
    Consider a color pattern $f\in \{B,R\}^{k+1}$ as the walk
    on the following directed graph
    of length $k$ whose vertices spell out~$f$:
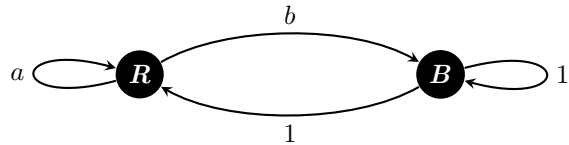
    \begin{center}
      \begin{tikzpicture}[>=stealth,scale=2]
         \path (0,0) node[fill,shape=circle,text=white,inner sep=0.25em] (R) 
                     {\textit{\textbf{R}}}
               (2,0) node[fill,shape=circle,text=white,inner sep=0.25em] (B) 
                     {\textit{\textbf{B}}};
         \begin{scope}[->, thick]
           \draw (R) to[bend left,looseness=0.75] node[auto] {$b$} (B);
           \draw (B) to[bend left,looseness=0.75] node[auto] {$1$} (R);
           \draw (R) to[loop left,min distance=0.75cm] node[auto] {$a$} (R);
           \draw (B) to[loop right,min distance=0.75cm] node[auto] {$1$} (B);
         \end{scope}
      \end{tikzpicture}
    \end{center}
    The product of the edge weights is $a^{r_2}b^{r_1}$.
    Let $A=\left[\begin{smallmatrix}a&b\\1&1
      \end{smallmatrix}\right]$.  Then
    \begin{align*}
      A^k = \begin{bmatrix}
        p_k^{R}(a,b) & p_k^{RB}(a,b)\\[0.4em]
        p_k^{BR}(a,b) & p_k^{B}(a,b)
      \end{bmatrix}
    \end{align*}
    with $p_k^{BR}$ and $p_k^{RB}$ defined analogously to
    $p_k^{R}$ and $p_k^{B}$.
    By Theorem~4.7.2 in \cite{EC},
    we find the following generating functions:
    \begin{align*}
      \sum_{k\geq 0} p_k^{R}(a,b)t^k &= \frac{1-t}{(a-b)t^2-(a+1)t+1},\\
      \sum_{k\geq 0} p_k^{B}(a,b)t^k &= \frac{1-at}{(a-b)t^2-(a+1)t+1}.
    \end{align*}
    Using partial fraction expansions, we arrive at
    \begin{align*}
      p_k^{R}(a,b) &= \frac{t_+ t_-}{t_+-t_-} \left(
        \frac{1-t_-}{t_-^{k+1}}-
          \frac{1-t_+}{t_+^{k+1}} \right),\\
      p_k^{B}(a,b) &= \frac{t_+ t_-}{t_+ - t_-}
        \left( \frac{1-at_-}{t_-^{k+1}} - \frac{1-at_+}{t_+^{k+1}}\right)
    \end{align*}
    where
    \begin{align*}
      t_\pm &= \frac{a+1 \pm \sqrt{(a-1)^2+4b}}{2(a-b)}.
    \end{align*}
    It is now a simple exercise to produce an expression for $\cdp_k$.  Letting
    \[
      \cdproot_{\pm} = \frac{c+a \pm \sqrt{(c-a)^2 + 4b}}{2},
    \]
    we have that
    \begin{align}
      \label{eq:cdp}
      \cdp_{k}(a,b,c) = \cdproot_+^k + \cdproot_-^k.
    \end{align}

\vskip 10pt
\noindent \emph{Poisson tails}
\vskip 10pt

We will 
frequently require tail estimates of functions of Poisson fields.  For this purpose, 
we use a bound that can be derived from modified log-Sobolev inequalities.
  \begin{lemma}
    \label{lem:PoissonTail}
      Let $\pi$ be a product measure of $q$ Poisson laws, with means $m_i$ for $1 \leq i \leq q$.  
      Let $F$ be a function from $\mathbb{N}^q \to \R,$ and define 
      $\nabla_i F( \mathbf{x} ) = F( \mathbf{x} + e_i) - F(\mathbf{x})$ with $e_i$ a 
      standard basis vector.  Further, let 
      $\| \nabla_i F \| = \sup_{ \mathbf{x} \in \mathbb{N}^q} |\nabla_i F( \mathbf{x})|$.  
      If there are positive reals $M_1$ and $M_2$ so that
      \[
        \sum_{i=1}^q \left\| \nabla_i F \right\|^2 m_i \leq M_1 \text{~~~and~~~}
        \max_{i=1,\ldots,q} \left\| \nabla_i F \right\| \leq M_2,
      \]
      then for every $r \geq 0,$
      \[
        \pi \left( F \geq \expect F + r \right) \leq
          \exp\left(
            -\frac{r}{2M_2}\log\left(1+\frac{M_2 r}{M_1}\right)
          \right).
      \]
  \end{lemma}
  \begin{proof}
    This is a special case of the stronger theorem of Wu~\cite{Wu00}, Proposition 3.1.
  \end{proof}

\section{Poisson approximations}
\label{sec:poisson}
\newcommand{\Pairings}{\mathscr{P}}

\newcommand{\Hams}{\mathscr{H}}

\newcommand{\gis}{\mathcal{I}}

\newcommand{\gind}{F}

\newcommand{\I}{\mathbf{I}}

\newcommand{\Ww}{\mathcal{W}}

\newcommand{\err}{\mathcal{E}}

\newcommand{\Poilimit}{\mathbf{Y}}
\newcommand{\poilimit}[1]{Y_{#1}}
\newcommand{\poilimitE}[1]{Z_{#1}}

\newcommand{\conditionalprocess}{J_{\omitted\alpha}}

In this section, we will prove Proposition~\ref{prop:multpoiapprox}, establishing
Poisson approximations for the cycle process $\cycprocess(P)$ when
$P\drawnfrom \pmd$ or $P\drawnfrom \phammd$.
Let $\rHam$ be a uniformly sampled Hamiltonian cycle on vertices
$\{1,\ldots,n\}$, and let $\rP\drawnfrom \pmd[d-2]$.
Recall that by representing $\rHam$ as a pairing, combining this pairing
with $P$, and randomly reordering the prevertices in each
vertex bin, we obtain the model $\phammd$. Our strategy
will be to avoid this complication for as long as possible,
and rather to work directly with $(\rHam, \rP)$. Let $\sepmixedmd$ be
the distribution 
of $(\rHam, \rP)$, which we call the unscrambled mixed model.
 
  We will represent 
  $(\rHam, \rP)$ as a pseudograph formed by superimposing their 
  projections.
  Further, we color the edges from
  the configuration model \emph{red} (denoted simply $R$), and edges from the 
  Hamiltonian cycle \emph{blue} (denoted simply $B$), which agrees with our 
  terminology in other sections.  We will
  label each endpoint of a red edge by the
  prevertex in $\rP$ from which it comes (see Figure~\ref{fig:labeledgraph}).  
  This provides the same information
  as $(\rHam, \rP)$, and we will go back and forth between the two views
  as the situation demands.
  
  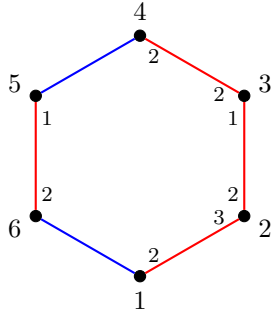
\begin{figure}
    \begin{center}
      \begin{tikzpicture}[scale=1.6,vert/.style={circle,fill,inner sep=0,
              minimum size=0.15cm,draw},>=stealth]
          \path  (270:1) node[vert] (s0)  {}   +(270:.2) node {$1$}
                 (330:1) node [vert] (s1)  {} +(330:0.2) node {$2$}
                 (30:1) node[vert] (s2) {} +(30:0.2) node {$3$}
                 (90:1) node[vert] (s3) {} +(90:0.2) node {$4$}
                 (150:1) node[vert] (s4) {} +(150:0.2) node {$5$}
                 (210:1) node[vert] (s5) {} +(210:0.2) node {$6$};
            \draw[thick, blue] (s3) -- (s4);
            \draw[thick,blue] (s5) -- (s0);
            \draw[thick,red] (s0) -- (s1) node[black,pos=0.15,font=\footnotesize,
                                               shift=(300:-0.15)] {$2$}
                                          node[black,pos=0.85,font=\footnotesize,
                                               shift=(300:-0.15)] {$3$};
            \draw[thick,red] (s1) -- (s2) node[black,pos=0.15,font=\footnotesize,
                                               shift=(0:-0.15)] {$2$}
                                          node[black,pos=0.85,font=\footnotesize,
                                               shift=(0:-0.15)] {$1$};
            \draw[thick,red] (s2) -- (s3) node[black,pos=0.15,font=\footnotesize,
                                               shift=(60:-0.15)] {$2$}
                                          node[black,pos=0.85,font=\footnotesize,
                                               shift=(60:-0.15)] {$2$};

              \draw[thick,red] (s4) -- (s5) node[black,pos=0.15,font=\footnotesize,
                                               shift=(180:-0.15)] {$1$}
                                          node[black,pos=0.85,font=\footnotesize,
                                               shift=(180:-0.15)] {$2$};
      \end{tikzpicture}
    \end{center}
    \caption{A representation of a cycle $\alpha\in\fullmixedcycspace$
    on vertices $1,\ldots,6$
    as a graph with labeled edges. Edges in $\rHam$ are colored blue, and edges in
    $\rP$ are colored red. The ends of each red edge are labeled with a number
    from $1,\ldots,d-2$ to indicate which prevertices the edge comes from.}
    \label{fig:labeledgraph}
  \end{figure}

We define
$\mixedcycspace[k]$ to be the set of all possible cycles of length $k$
  in $(\rHam, \rP)$, in analogy with $\cycspace[k]$.  
  As with the graph representation of $(\rHam, \rP)$, we represent these by
  $\{R,B\}$--edge-colored cycles, with prevertex labels on each red edge.
  We will refer to the \emph{color pattern} of an element
  of $\mixedcycspace$ as a
  sequence of $R$s and $B$s of length $k$ identified up to rotation and reversal,
  corresponding to the order in which the edge colors appear on the cycle.
  Define $\fullmixedcycspace = \bigcup_{k=1}^r \mixedcycspace$.
  Let $I_\alpha$ be the indicator that $(\rHam, \rP)$ contains the
  cycle $\alpha$, for any $\alpha\in\fullmixedcycspace$, and let
  $\mixedcycprocess = (I_\alpha,\,\alpha\in\fullmixedcycspace)$.
  \label{page:mixedcycprocessdef}
  Our goal is to show Poisson approximations for $\mixedcycprocess$,
  and then to transfer these results over to Poisson approximations
  for $\cycprocess(P)$ with $P\drawnfrom\phammd$.
  
  All in all, this section is quite technical and delicate. 
  For the reader who wants to skip to the chase,
  we recommend focusing on the arguments for $\cycprocess(P)$, which typically
  use the same ideas as those for
  $\mixedcycprocess$ but have fewer technical details.
  The most important part of our argument is in Section~\ref{subsec:multbounds},
  from Lemma~\ref{lem:perturb} to Corollary~\ref{cor:relativepointprobs}.
  
  In Section~\ref{subsec:Ecycles}, we use the polynomials from
  Section~\ref{sec:supporting} to compute the expected number of cycles
  of each size in $(\rHam,\rP)$, as well as a few related quantities.
  Section~\ref{subsec:exceptional} is devoted to a bound on the probability
  that the cycles in $P$ or $(\rHam,\rP)$ are exceptional, in that they
  overlap each other or
  there are an unusually large number of them.
  In Section~\ref{sec:couplings}, we give a coupling of the model
  $\pmd$ with a conditioned version of itself, and we do the same thing
  for $\sepmixedmd$. Finally, in Section~\ref{subsec:multbounds}, we give the 
  main argument and prove Proposition~\ref{prop:multpoiapprox}.

  \subsection{Expectations of cycle counts}\label{subsec:Ecycles}
    Our first job is to use the polynomials from Section~\ref{sec:supporting}
    to compute the expected number of cycles in $(\rHam, \rP)$.
    We will need the following facts about the asymptotics of
    $\ff{n}{k}$ and $\dff{n}{k}$, which are elementary to check.
    \begin{lemma}\label{lem:ffstirling}
      For all $k < n/2$,
      \begin{align*}
        n^ke^{-O(k^2/n)}
          \leq \ff{n}{k}  &\leq n^k,\\
        \intertext{and}
        n^ke^{-O(k^2/n)}
          \leq \dff{n}{k}  &\leq n^k.
      \end{align*}
    \end{lemma}
    \begin{lemma}\label{lem:cycleprob}  
    Suppose $\alpha$ is an $\{R,B\}$--edge-colored $k$-cycle.  
    Let $r_1$ and $r_2$ 
    refer to the color pattern of $\alpha$,
    as in Section~\ref{sec:supporting},
    If $k<n/2$ and $r_1\geq 1$,
    \begin{align*}
      p_\alpha := \E I_\alpha &= 
        \frac{2^{r_1}}{\ff{n-1}{k-r_2-r_1}\dff{n(d-2)}{r_2+r_1}}\\
        &= \frac{2^{r_1}}{n^k(d-2)^{r_2+r_1}}
            \exp\left(O\left( \frac{(k-r_1-r_2)^2}{n}+\frac{(r_1+r_2)^2}{nd}
            \right)\right)\\
        &= \frac{2^{r_1}}{n^k(d-2)^{r_2+r_1}}
        e^{O(k^2/n)}.
    \end{align*}
  \end{lemma}
  \begin{proof}
    The cycle $\alpha$ contains $r_1+r_2$ red edges and $k-r_1-r_2$
    blue edges. The probability that $\rP$ contains all of these
    red edges is $1/\dff{n(d-2)}{r_1+r_2}$. The blue edges form
    $r_1$ disjoint paths, and the probability that
    $\rHam$ contains these paths is $2^{r_1}/\ff{n-1}{k-r_1-r_2}$.
    The approximations follow from 
    Lemma~\ref{lem:ffstirling}.
  \end{proof}
  The restriction $r_1\geq 1$ in the previous theorem is because an all-blue
  edge-colored cycle cannot appear in $(H,Q)$. 
  Now, we compute the expected number of cycles of each length
  in $(\rHam,\rP)$, as well as the expected number
  of red and blue edges in these cycles:
  \begin{corollary}\label{cor:expectations}
    Recall that 
    $I_\alpha$ is the indicator that $(\rHam, \rP)$ contains the cycle
    $\alpha$.  Let $s_\alpha$ be the number of red edges 
    and $t_\alpha$ be the number of blue edges in $\alpha$.
    Then
    \begin{align}
      \E \sum_{\alpha\in\mixedcycspace} s_\alpha I_\alpha &=
      \frac{(d-2)(d-1)^k + 2(-1)^k}{2d}
      e^{O(k^2/n)},        \label{eq:Escycles}\\
      \E \sum_{\alpha\in\mixedcycspace} t_\alpha I_\alpha &=
      \frac{2(d-1)^k + (d-2)(-1)^k-d}{2d}
        e^{O(k^2/n)},\label{eq:Etcycles}\\
      \E \sum_{\alpha\in\mixedcycspace} I_\alpha &=
        \frac{(d-1)^k + (-1)^k - 1}{2k}e^{O(k^2/n)}. \label{eq:Ecycles}
    \end{align}
  \end{corollary}
  \begin{proof}
    We begin by counting the expected number of cycles of a given length
    and color pattern.
    Fix an $\{R,B\}$--edge-colored, rooted oriented cycle $\mathcal{C}_k.$
    There are $\ff{n}{k}$ different ways to choose the vertices of
    such a cycle, and there are
    $(d-2)^{2r_1+r_2}(d-3)^{r_2}$ ways to assign 
    labels to the prevertices on the red edges.
    So long as the cycle is not all blue,
    it has the probability given in Lemma~\ref{lem:cycleprob},
    so the expected number of cycles with this color pattern
    is 
    \begin{align*}
      \frac{2^{r_1}\ff{n}{k}(d-2)^{2r_1+r_2}(d-3)^{r_2}}
        {\ff{n-1}{k-r_2-r_1}\dff{n(d-2)}{r_2+r_1}}
        = (2(d-2))^{r_1}(d-3)^{r_2} e^{O(k^2/n)}.
    \end{align*}
    Summing this over all possible edge colorings
    besides $B^k$, the expected number of rooted, oriented
    cycles of length $k$ is
    \begin{align*}
      2k\E\sum_{\alpha\in\mixedcycspace} I_\alpha
        &= \big(p_k(d-3,2(d-2))-1\big)e^{O(k^2/n)},
    \end{align*}
    referring to the polynomials $p_k(a,b)$ from \eqref{eq:pkdef}.
    Applying \eqref{eq:cdp}
    proves \eqref{eq:Ecycles}.
    If we repeat the same counting procedure, but only sum over edge colorings
    of $\mathcal{C}_k$ that color the first edge red,
    we count each cycle $\alpha$ a total of $2s_\alpha$ times. Using the polynomials
    $p_k^R$ from Section~\ref{sec:supporting}, 
    \begin{align*}
      \E\sum_{\alpha\in\mixedcycspace} 2s_\alpha I_\alpha
        &= p_k^R(d-3,2(d-2))e^{O(k^2/n)},
    \end{align*}
    which proves \eqref{eq:Escycles}.  To show \eqref{eq:Etcycles},
    we do the same thing and subtract off the term given by the all blue
    pattern:
    \begin{align*}
      \E\sum_{\alpha\in\mixedcycspace}2t_\alpha I_\alpha &= \left(p_k^B(d-3,2(d-2))
         -1\right)
         e^{O(k^2/n)}.\qedhere
    \end{align*}
  \end{proof}

  \subsection{Exceptional cycle counts}\label{subsec:exceptional}
  
  If $x=(x_\alpha,\,\alpha\in\fullmixedcycspace)$ 
  or $x=(x_\alpha,\,\alpha\in\fullcycspace)$ with $x_\alpha$ equal to
  zero or one for each $\alpha$, then we interpret $x$
  as a collection of cycles, and we will say that $x$ contains 
  $\alpha$ to mean $x_\alpha=1$.  Our estimates will fail
  for states $x$ that contain too many cycles or overlapping cycles.
  The following definitions describe which states in
  the unscrambled mixed model
  $\sepmixedmd$ and in the pairing model $\pmd$ we will be able to
  analyze:
  \begin{definition}\label{def:lneat}
    For some $\lambda\geq 1$, 
    a vector $x=(x_\alpha,\,\alpha\in\fullmixedcycspace)$ is $\lambda$-neat
    if the following hold:
    \begin{enumerate}[i)]
      \item The vector $x$ does not contain any overlapping cycles;
        that is, if $x_\alpha=x_\beta=1$, then $\alpha$ and $\beta$ share
        no prevertices or Hamiltonian vertices.  
      \item Let $x$ contain a total of $\Phi$ prevertices
        and $\Psi$ Hamiltonian cycle vertices.  These
        two counts satisfy
        \begin{align*}
          \Phi \leq \lambda(d-1)^r,\qquad\qquad \Psi\leq \lambda(d-1)^{r-1}.
        \end{align*}
    \end{enumerate}
  \end{definition}
  \begin{definition}\label{def:lneatpmd}
    For some $\lambda\geq 1$, 
    a vector $x=(x_\alpha,\,\alpha\in\fullcycspace)$ is $\lambda$-neat
    if the following hold:
    \begin{enumerate}[i)]
      \item The vector $x$ does not contain any overlapping cycles;
        that is, if $x_\alpha=x_\beta=1$, then $\alpha$ and $\beta$ share
        no prevertices.
      \item The total number of prevertices contained in $x$ is at most
         $\lambda(d-1)^r$.
    \end{enumerate}
  \end{definition}
  
  The point of this section is to show that when $\lambda$
  grows logarithmically, nearly all graphs have cycle
  counts satisfying these criteria.
  \begin{proposition}\label{prop:pmdneat}
    For $d\geq 3$, and all $r$ and $n$,
    \begin{align*}
      \P[\text{$\cycprocess(P)$ is not $(\log n)$-neat}]
        &\leq \frac{\Cl{C:lneatprob}(d-1)^{2r-1}}{n},\qquad P\drawnfrom{\pmd}.
    \end{align*}
  \end{proposition}
  \begin{proposition}\label{prop:mixedneat}
    For $d\geq 3$, and all $r$ and $n$,
    \begin{align*}
      \P[\text{$\mixedcycprocess$ is not $(\log n)$-neat}]
        &\leq \frac{\Cr{C:lneatprob}(d-1)^{2r-1}}{n}.
    \end{align*}
  \end{proposition}
  These two propositions have essentially the same proof, except that
  the details of the second one are somewhat trickier.
  \begin{proof}[Proof of Proposition~\ref{prop:pmdneat}]
    \newcommand{\setofbadcycles}{\mathcal{B}}
    \newcommand{\Ii}{\fullcycspace}
    \newcommand{\Overlap}{\textsc{Overlap}}
    \newcommand{\Many}{\textsc{Many}}
    Let $\lambda=\log n$.
    We define two events whose probability we wish to bound:
    \begin{align*}
      \Overlap &= \{\text{$P$ contains two cycles of length $r$ or less
        sharing an edge}\},\\
      \Many &= \{\text{$\cycprocess(P)$ contains more than $\lambda(d-1)^r$
        prevertices}\} \cap \Overlap^C.
    \end{align*}
    To bound the probability of $\Overlap$,
    we  bound
    $\sum_{\alpha,\beta}\E I_\alpha I_\beta$, where $\alpha$ and $\beta$
    range over all pairs of overlapping cycles.
    For some $\alpha\in\fullcycspace$, let $\Ii^l_\alpha\subseteq\fullcycspace$
    denote the set of cycles that share exactly $l$ pairs with $\alpha$, 
    but otherwise do not share any prevertices.
    For any $\beta\in\Ii_\alpha^l$,
    \begin{align*}
      \E[I_\alpha I_\beta]=\P[\text{$P$ contains $\alpha$ and $\beta$}]
      = \frac{1}{\dff{nd}{\abs{\alpha}+\abs{\beta}-l}}.
    \end{align*}
    
    \begin{figure}
         \begin{center}
        \begin{tikzpicture}[scale=1.75,vert/.style={circle,fill,inner sep=0,
              minimum size=0.15cm,draw}, H/.style={dashed},>=stealth]
          \begin{scope}
          \path  (270:1) node[vert] (s1)  {}   +(270:.2) node {$1$}
                 (237:1) node [vert] (s2)  {} +(237:0.2) node {$2$}
                 (205:1) node[vert] (s3) {} +(205:0.2) node {$3$}
                 (171:1) node[vert] (s4) {} +(171:0.2) node {$4$}
                 (139:1) node[vert] (s5) {} +(139:0.2) node {$5$}
                 (106:1) node[vert] (s6) {} +(106:0.2) node {$6$}
                 (74:1) node[vert] (s7) {} +(74:0.2) node {$7$}
                 (41:1) node[vert] (s8) {} +(41:0.2) node {$8$}
                 (8:1) node[vert] (s9) {} +(8:0.2) node {$9$}
                 (335:1) node[vert] (s10) {} +(335:0.2) node {$10$}
                 (302:1) node[vert] (s11) {} +(302:0.2) node {$11$};
            \draw[thick] (s1) -- (s2) node[pos=0.15,font=\footnotesize,
                                           shift=(253:-0.15)] {$1$}
                                      node[pos=0.85,font=\footnotesize,
                                           shift=(253:-0.15)] {$1$}
                         (s2) -- (s3) node[pos=0.15,font=\footnotesize,
                                           shift=(220:-0.15)] {$2$}
                                      node[pos=0.85,font=\footnotesize,
                                           shift=(220:-0.15)] {$1$};
            \draw[densely dotted, thick] (s3) -- (s4) node[pos=0.15,font=\footnotesize,
                                           shift=(187:-0.15)] {$3$}
                                      node[pos=0.85,font=\footnotesize,
                                           shift=(187:-0.15)] {$2$}
                         (s4) -- (s5) node[pos=0.15,font=\footnotesize,
                                           shift=(155:-0.15)] {$1$}
                                      node[pos=0.85,font=\footnotesize,
                                           shift=(155:-0.15)] {$1$};
            \draw[thick] (s5) -- (s6) node[pos=0.15,font=\footnotesize,
                                           shift=(122:-0.15)] {$2$}
                                      node[pos=0.85,font=\footnotesize,
                                           shift=(122:-0.15)] {$1$}
                         (s6) -- (s7) node[pos=0.15,font=\footnotesize,
                                           shift=(89:-0.15)] {$2$}
                                      node[pos=0.85,font=\footnotesize,
                                           shift=(89:-0.15)] {$3$};
            \draw[thick, densely dotted](s7) -- (s8) node[pos=0.15,font=\footnotesize,
                                           shift=(57:-0.15)] {$1$}
                                      node[pos=0.85,font=\footnotesize,
                                           shift=(57:-0.15)] {$2$};
            \draw[thick] (s8) -- (s9) node[pos=0.15,font=\footnotesize,
                                           shift=(24:-0.15)] {$3$}
                                      node[pos=0.85,font=\footnotesize,
                                           shift=(24:-0.15)] {$1$};
            \draw[thick, densely dotted] (s9) -- (s10) node[pos=0.15,font=\footnotesize,
                                           shift=(351:-0.15)] {$2$}
                                      node[pos=0.85,font=\footnotesize,
                                           shift=(351:-0.15)] {$1$};
            \draw[thick] (s10) -- (s11) node[pos=0.15,font=\footnotesize,
                                           shift=(318:-0.15)] {$2$}
                                      node[pos=0.85,font=\footnotesize,
                                           shift=(318:-0.15)] {$1$}
                         (s11) -- (s1) node[pos=0.15,font=\footnotesize,
                                           shift=(285:-0.15)] {$2$}
                                      node[pos=0.85,font=\footnotesize,
                                           shift=(285:-0.15)] {$2$};

            \draw (0,-1.5) node[text width=5.8cm,anchor=base]
            {
              The cycle $\alpha$, with $\gamma$ dotted.  The subgraph
              $\gamma$ has components $A_1,\ldots,A_p$.  In this example,
              the number of components of $\gamma$ is $p=3$, 
              the size of $\alpha$ is $k=11$, and the number
              of edges in $\gamma$ is $l=4$.\\[0.1cm]
              In this example, we will construct a cycle $\beta$ of length
              $j=10$
              that overlaps with $\alpha$ at $H$.
            };
          \end{scope}
          \begin{scope}[xshift=3.7cm]
            \path (-1.476,0) node[vert,label=below:{$3$}] (s3) {}
                  -- ++(0.563,0) node[vert,label=below:$4$] (s4) {}
                  -- ++(0.563,0) node[vert,label=below:$5$] (s5) {}
                  -- ++(0.35,0) node[vert,label=below:$10$] (s10) {}
                  -- ++(0.563,0) node[vert,label=below:$9$] (s9){}
                  -- ++(0.35,0) node[vert,label=below:$7$] (s7) {}
                  -- ++(0.563,0) node[vert,label=below:$8$] (s8){};
            \draw[thick] (s3) -- (s4) node[pos=0.15,font=\footnotesize,
                                           shift={(0,0.15)}] {$3$}
                                      node[pos=0.85,font=\footnotesize,
                                           shift={(0,0.15)}] {$2$}
                         (s4) -- (s5) node[pos=0.15,font=\footnotesize,
                                           shift={(0,0.15)}] {$1$}
                                      node[pos=0.85,font=\footnotesize,
                                           shift={(0,0.15)}] {$1$}
                         (s10) -- (s9) node[pos=0.15,font=\footnotesize,
                                           shift={(0,0.15)}] {$1$}
                                      node[pos=0.85,font=\footnotesize,
                                           shift={(0,0.15)}] {$2$}
                         (s7) -- (s8) node[pos=0.15,font=\footnotesize,
                                           shift={(0,0.15)}] {$1$}
                                      node[pos=0.85,font=\footnotesize,
                                           shift={(0,0.15)}] {$2$};
            \draw (0,-1.5) node[text width=5.8cm,anchor=base]
            {
              \textbf{Step 1.} We lay out the components $A_1,\ldots,A_p$.
              We can order and orient $A_2,\ldots,A_p$ however we would like,
              for a total of $(p-1)!2^{p-1}$ choices.
              Here, we have ordered the components $A_1, A_3, A_2$, 
              and we have reversed the orientation of $A_3$.
            };
          \end{scope}
          \begin{scope}[yshift=-4.1cm,new/.style={}]
            \path (-1.476,0) node[vert,label=below:{$3$}] (s3) {}
                  -- ++(0.563,0) node[vert,label=below:$4$] (s4) {}
                  -- ++(0.563,0) node[vert,label=below:$5$] (s5) {}
                  -- ++(0.35,0) node[vert,label=below:$10$] (s10) {}
                  -- ++(0.563,0) node[vert,label=below:$9$] (s9){}
                  -- +(-0.1065,0.553) node[vert,new] (a1){}
                  -- +(0.4565,0.553) node[vert,new] (a2) {}
                  -- ++(0.35,0) node[vert,label=below:$7$] (s7) {}
                  -- ++(0.563,0) node[vert,label=below:$8$] (s8){}
                  -- (0,-0.5) node[vert,new] (a3) {};
            \draw[thick] (s3) -- (s4) node[pos=0.15,font=\footnotesize,
                                           shift={(0,0.15)}] {$3$}
                                      node[pos=0.85,font=\footnotesize,
                                           shift={(0,0.15)}] {$2$}
                         (s4) -- (s5) node[pos=0.15,font=\footnotesize,
                                           shift={(0,0.15)}] {$1$}
                                      node[pos=0.85,font=\footnotesize,
                                           shift={(0,0.15)}] {$1$}
                         (s10) -- (s9) node[pos=0.15,font=\footnotesize,
                                           shift={(0,0.15)}] {$1$}
                                      node[pos=0.85,font=\footnotesize,
                                           shift={(0,0.15)}] {$2$}
                         (s7) -- (s8) node[pos=0.15,font=\footnotesize,
                                           shift={(0,0.15)}] {$1$}
                                      node[pos=0.85,font=\footnotesize,
                                           shift={(0,0.15)}] {$2$};
            \draw[thick,new] (s5) to [out=90,in=90] (s10);
            \draw[thick,new] (s9) to (a1);
            \draw[thick,new] (a1) to (a2);
            \draw[thick,new] (a2) to (s7);
            \draw[thick,new] (s8) to[out=225,in=0] (a3);
            \draw[thick,new] (a3) to[out=180,in=305] (s3);
            
            \draw (0,-0.95) node[anchor=base,text width=5.8cm]
            {
              \textbf{Step 2.} Next, we choose how many edges will go in each
              gap between components.  Each gap must contain at least
              one edge, and we must add a total of $j-l$ edges, giving
              us $\binom{j-l-1}{p-1}$ choices.
              In this example, we have added one edge after $A_1$,
              three after $A_3$, and two after $A_2$.
            };
          \end{scope}
          \begin{scope}[yshift=-4.1cm,xshift=3.7cm,new/.style={}]
            \path (-1.476,0) node[vert,label=left:{$3$}] (s3) {}
                  -- ++(0.563,0) node[vert,label=below:$4$] (s4) {}
                  -- ++(0.563,0) node[vert,label=below:$5$] (s5) {}
                  -- ++(0.35,0) node[vert,label=below:$10$] (s10) {}
                  -- ++(0.563,0) node[vert,label=below:$9$] (s9){}
                  -- +(-0.1065,0.553) node[vert,label={[new]above left:$23$}] (a1){}
                  -- +(0.4565,0.553) node[vert,label={[new]above right:$1$}] (a2) {}
                  -- ++(0.35,0) node[vert,label=below:$7$] (s7) {}
                  -- ++(0.563,0) node[vert,label=right:$8$] (s8){}
                  -- (0,-0.5) node[vert,label={[new]below:$15$}] (a3) {};
            \draw[thick] (s3) -- (s4) node[pos=0.15,font=\footnotesize,
                                           shift={(0,0.15)}] {$3$}
                                      node[pos=0.85,font=\footnotesize,
                                           shift={(0,0.15)}] {$2$}
                         (s4) -- (s5) node[pos=0.15,font=\footnotesize,
                                           shift={(0,0.15)}] {$1$}
                                      node[pos=0.85,font=\footnotesize,
                                           shift={(0,0.15)}] {$1$}
                         (s10) -- (s9) node[pos=0.15,font=\footnotesize,
                                           shift={(0,0.15)}] {$1$}
                                      node[pos=0.75,font=\footnotesize,
                                           shift={(0,0.15)}] {$2$}
                         (s7) -- (s8) node[pos=0.25,font=\footnotesize,
                                           shift={(0,0.15)}] {$1$}
                                      node[pos=0.85,font=\footnotesize,
                                           shift={(0,0.15)}] {$2$}
                         (a1) -- (a2) node[pos=0.15,font=\footnotesize,
                                           shift={(0,0.15)}] {$3$}
                                      node[pos=0.85,font=\footnotesize,
                                           shift={(0,0.15)}] {$2$}
                         (a2) -- (s7) node[pos=0.15,font=\footnotesize,
                                           shift={(0.14,0)}] {$1$}
                                      node[pos=0.65,font=\footnotesize,
                                           shift={(0.1,0)}] {$2$}
                         (s8) to[out=225,in=0] (a3)
                         (a3) to[out=180,in=305] (s3)
                         (s8) node[font=\footnotesize,shift={(-0.06,-0.25)}] {$1$}
                         (a3) node[font=\footnotesize,shift={(0.2,0.15)}] {$1$}
                           node[font=\footnotesize,shift={(-0.23,0.15)}] {$2$}
                         (s3) node[font=\footnotesize,shift={(0.06,-0.25)}] {$1$}
                         (s9) -- (a1) node[pos=0.35,font=\footnotesize,
                                           shift={(-0.16,0)}] {$1$}
                                      node[pos=0.85,font=\footnotesize,
                                           shift={(-0.16,0)}] {$2$};
            \draw        (s5) to [out=90,in=90] (s10);
            \draw (s5) node[font=\footnotesize,
                         shift={(0.025,0.35)}] {$2$}
                  (s10) node[font=\footnotesize,
                                           shift={(-0.025,0.35)}] {$3$};

            \draw (0,-0.95) node[anchor=base, text width=5.8cm]
            {
              \textbf{Step 3.} We can choose the new vertices in
              $[n-p-l]_{j-p-l}$ ways, and we can direct and give labels
              to the new edges in $(d-1)^{j-l+p}d^{j-l-p}$ ways.
            };
          \end{scope}
        \end{tikzpicture}          
         \end{center}
      \caption{Counting the cycles that overlap $\alpha$ at $\gamma$.}
    \end{figure}
    Our plan is to bound the size of $\Ii_\alpha^l$.
   Fix some $\alpha\in\cycspace$ and let $\gamma$ be some set of its edges of
   size $\abs{\gamma}=l$ and with $p$ connected components (and thus $p+l$
   vertices).
   We will show that the number of $j$-cycles that
   overlap with $\alpha$ at $\gamma$ is bounded by
   \begin{align}
     2^{p-1}(p-1)!\binom{j-l-1}{p-1}[n-p-l]_{j-p-l}(d-1)^{j-l+p}
     d^{j-l-p}\label{eq:Hoverlap}
   \end{align}
   Call the components of $\gamma$ $A_1,\ldots,A_p$.  We can construct any
   $\beta\in\Ii_j$ that overlaps with $\alpha$ at $\gamma$ by stringing together
   these components with other edges in between them.
   The components can appear in $\beta$ in any order, and each can appear with
   one of two orientations.  Since the vertices in $\beta$ are only
   given up to cyclic rotation, we can assume without loss of generality
   that component $A_1$ appears first, with some fixed orientation, followed
   by $A_2,\ldots,A_p$ in any order with any orientation, for a total
   of $2^{p-1}(p-1)!$ choices.
   
   Now, imagine the components laid in a line, with gaps between them,
   and count the number of ways to fill the gaps.  Each of the $p$ gaps
   must contain at least one edge, and the total number of edges in
   the gaps is $j-l$.  Thus the total number of possible gap sizes is the
   number of compositions of $j-l$ into $p$ parts, $\binom{j-l-1}{p-1}$.
   
   Now that we have chosen the number of edges to appear in each gap, we choose
   the edges themselves.  We can do this by giving an ordered list
   of $j-p-l$ vertices to go in the gaps, along with a label and an orientation
   for each of the $j-l$ new edges.  There are $[n-p-l]_{j-p-l}$ ways to 
   choose the vertices, and $(d-1)^{j-l+p}d^{j-l-p}$ ways to choose
   the labels.  This establishes \eqref{eq:Hoverlap}.
   It is a bound rather than an equality because some cycles constructed in
   this way might have additional overlap with $\alpha$.
   
   Next, we count the number of ways to choose a subgraph $\gamma$ of $\alpha$
   with $l$ edges and $p$ components.  Let $s_1,\ldots,s_k$ be the vertices
   of $\alpha$, in order.  Suppose that we have sequences of positive
   integers satisfying $a_1+\cdots+a_p=l$ and $b_1+\cdots+b_p=k-l$.
   Then we can obtain a subgraph of $\alpha$ with $l$ edges and
   $p$ components by starting at some vertex $s_i$ and including
   the next $a_1$ edges of $\alpha$ in $\gamma$, then excluding the next
   $b_1$, then including the next $a_2$, and so on.  Every subgraph with
   $p$ components and $l$ edges is given in exactly $p$ ways, since
   $s_i$ can be at the beginning of any of the $p$ components.  The number
   of ways to choose $i$ and the two sequences is
   $k\binom{l-1}{p-1}\binom{k-l-1}{p-1}$, and so the total number
   of such subgraphs is this, divided by $p$.
   
   All together, we have
   \begin{align*}
     \abs{\Ii_\alpha^l\cap\cycspace[j]}&\leq \sum_{p=1}^{l\wedge j-l}
     \frac{k}{p}\binom{l-1}{p-1}\binom{k-l-1}{p-1}
     2^{p-1}(p-1)!\binom{j-l-1}{p-1}\times\phantom{}\\
       &\qquad\qquad[n-p-l]_{j-p-l}(d-1)^{j-l+p}
     d^{j-l-p}.
   \end{align*}
   We apply the bounds
       \begin{align*}
         \binom{l-1}{p-1}&\leq \frac{r^{p-1}}{(p-1)!},\\
         \binom{k-l-1}{p-1},\,\binom{j-l-1}{p-1}&\leq (er/(p-1))^{p-1},
       \end{align*}
  to get
  \begin{align*}
    \abs{\Ii_\alpha^l\cap\Ii_j}&\leq \sum_{p=1}^{l\wedge j-l}
    \frac{k}{p}\left(\frac{2e^2r^3}{(p-1)^2}\right)^{p-1}
      [n-p-l]_{j-p-l}(d-1)^{j-l+p}d^{j-l-p}\\
      &=k(d-1)^{j-l+1}d^{j-l-1}[n-l-1]_{j-l-1}\times\phantom{} \\
      &\qquad\left(1+\sum_{p=2}^{l\wedge j-l}
    \frac{1}{p[n-l-1]_{p-1}}\left(\frac{2(d-1)e^2r^3}{d(p-1)^2}\right)^{p-1}
    \right).
  \end{align*}
  We can assume without loss of generality that
  $r\leq n^{1/10}$, since the proposition holds
  for all $r > n^{1/10}$ just by choosing $\Cr{C:lneatprob}$
  large enough to make $\Cr{C:lneatprob}(d-1)^{2r-1}/n\geq 1$ in this
  case. Thus the sum in the above equation
  is bounded by a universal constant, and
  we can compute
  \begin{align*}
    \sum_{\beta\in\Ii_\alpha^l}\E I_\alpha I_\beta
      &\leq\sum_{j=l+1}^r\sum_{\beta\in\Ii_\alpha^l\cap\cycspace[j]}\frac{1}
        {\dff{nd}{k+j-l}}\\
      &=\sum_{j=l+1}^r 
        O\left(\frac{k(d-1)^{j-l+1}}
        {(nd)^{k+1}}\right)=O\left(\frac{k(d-1)^{r-l+1}}{(nd)^{k+1}}
        \right),
  \end{align*}
  and
  \begin{align}
    \P[\Overlap]\leq \sum_{\alpha\in\fullcycspace}\sum_{l\geq 1}
    \sum_{\beta\in\Ii_\alpha^l}\E I_\alpha I_\beta
    &\leq \sum_{k=1}^r\frac{[n]_kd^k(d-1)^k}{2k}\sum_{l=1}^{k-1}
    O\left(\frac{k(d-1)^{r-l+1}}{(nd)^{k+1}}
        \right)\nonumber\\
      &=\sum_{k=1}^r\frac{[n]_kd^k(d-1)^k}{2k}
      O\left(\frac{k(2d-1)^{r-1}}{(nd)^{k+1}}\right)\nonumber\\
      &=O\left(\frac{(d-1)^{2r-1}}{n}\right).\label{eq:pmdoverlap}
  \end{align}
    
    Now, we bound $\P[\Many]$, using another union bound, which will
    reduce
    the problem to computing tail probabilities of a Poisson process.
    Let $S\subseteq\fullcycspace$, and let $\abs{S}$ denote the total number
    of edges in all cycles in $S$. We call $S$ a bad set if
    $\abs{S}>\lambda(d-1)^r$ and
    its cycles do not overlap at any prevertices. If no proper subset
    of $S$ contains more than $\lambda(d-1)^r$ edges, we call $S$ a minimal
    bad set.
    For any choice of $m$ pairs of distinct prevertices out of
    $nd$, the probability that $P$ contains all of them is exactly
    $1/\dff{nd}{m}$. Thus by a union bound,
    \begin{align*}
      \P[\Many]
       &\leq \sum_{S} 
        \frac{1}{\dff{nd}{\abs{S}}},
    \end{align*}
    where $S$ ranges over all minimal bad subsets of $\fullcycspace$.
    
    If $S$ is a minimal bad set, then it contains at most $\lambda(d-1)^r/2+r$
    edges.
    By Lemma~\ref{lem:ffstirling}, 
    \begin{align*}
      \frac{1}{\ff{nd}{\abs{S}}} &=
      (nd)^{-\abs{S}}\exp\left(O\bigg(\frac{\abs{S}^2}{nd}\bigg)\right)
        \leq (nd)^{-\abs{S}}\exp\left(\frac{\Cl{C:experr}\lambda^2(d-1)^{2r-1}}
          {n}\right)
    \end{align*}
    for some absolute constant $\Cr{C:experr}$.
    Now, we estimate
    \begin{align}
      \sum_{S} 
        \frac{1}{\dff{nd}{\abs{S}}}
        &\leq\exp\left(\frac{\Cr{C:experr}\lambda^2(d-1)^{2r-1}}
          {n}\right)
          \sum_S \prod_{\alpha\in S} (nd)^{-\abs{\alpha}}\nonumber\\
        &\leq \exp\left(\frac{\Cr{C:experr}\lambda^2(d-1)^{2r-1}}
          {n}\right)
          e^\mu \sum_S \P[\pcycpf = \one{S}],\label{eq:pmdub2}
    \end{align}
    where we recall that $\pcycpf=(\poilimitE{\alpha},\,\alpha\in\fullcycspace)$
    has as its
    entries independent Poisson random variables with $\E\poilimitE{\alpha}
    = (nd)^{-\abs{\alpha}}$, and $\mu=\sum_\alpha \E \poilimitE{\alpha}$.
    Define $F(x)$ for $x=(x_\alpha,\,\alpha\in\fullcycspace)$
    by $F(x)=\sum_\alpha 2\abs{\alpha}x_\alpha$, so that
    $F(\one{S})$ is the number of prevertices
    in all cycles in $S$. Now,
    \begin{align*}
      \sum_S \P[\pcycpf = \one{S}] &\leq
        \P[F(\pcycpf) > \lambda(d-1)^r],
    \end{align*}
    and we can bound this probability with the modified log-Sobolev
    inequalities.
    First, we compute
    \begin{align*}
      \E F(\pcycpf) = \sum_{k=1}^r\frac{2k\abs{\cycspace[k]}}
        {(nd)^k} \leq \Cl{C:logsob1}(d-1)^r
    \end{align*}
    for a constant $\Cr{C:logsob1}$.
    In the notation of Lemma~\ref{lem:PoissonTail},
    \begin{align*}
      \sum_{\alpha\in\fullcycspace}\norm{\nabla_\alpha F}^2 \E \poilimitE{\alpha}
      =
      \sum_{\alpha\in\fullcycspace}(2\abs{\alpha})^2\E \poilimitE{\alpha}
        &\leq  2r\sum_{\alpha\in\fullcycspace} 2\abs{\alpha} \E \poilimitE{\alpha}\\
        &= 2r\E F(\pcycpf)\leq 2\Cr{C:logsob1}r(d-1)^r,
    \end{align*}     
    and
    \begin{align*}
      \max_{\alpha\in\fullcycspace}\norm{\nabla_\alpha F} \leq 2r.
    \end{align*}
    By Lemma~\ref{lem:PoissonTail},
    \begin{align*}
      \P\left[F(\pcycpf)>\lambda(d-1)^r\right]
        \leq\exp\left( -\frac{(\lambda-\Cr{C:logsob1})(d-1)^r}{4r}
          \log\left(\frac{\lambda}{\Cr{C:logsob1}}
        \right)\right).
    \end{align*}    
    
    Now, we substitute this back into \eqref{eq:pmdub2}.
    Making sure that we have chosen
    $\Cr{C:logsob1}$ large enough, we have
    $\mu\leq \Cr{C:logsob1}(d-1)^r/r$, and we obtain
    \begin{equation}\label{eq:pmdmanybound}
      \begin{split}
      \P[\Many] &\leq
      \exp\Bigg( -(d-1)^r \bigg(\frac{\log n -
        \Cr{C:logsob1}}{4r}\log\bigg(\frac{
      \log n}{\Cr{C:logsob1}}\bigg) - \frac{\Cr{C:logsob1}}{r} \\
      &\qquad\qquad\qquad\qquad
        \qquad\qquad-\frac{\Cr{C:experr}(\log n)^2 (d-1)^{r-1}}{n} \bigg)\Bigg).
      \end{split}
    \end{equation}
    For all $n$, $d$, $r$ such that $\Cr{C:experr}(\log n)^2 (d-1)^{r-1}>n$,
    the proposition holds trivially by choosing $\Cr{C:lneatprob}$
    sufficiently large. For the remaining values of $n$, $d$, and $r$,
    \eqref{eq:pmdmanybound} shows that $\P[\Many]= O(1/n)$,
    which with \eqref{eq:pmdoverlap} completes the proof.
  \end{proof}
  \begin{proof}[Proof of Proposition~\ref{prop:mixedneat}]
    \newcommand{\Overlap}{\textsc{Overlap}}
    \newcommand{\Many}{\textsc{Many}}
    Define events $\Overlap$ and $\Many$ as in the previous proposition.
    Let $\fullmixedcycspace_\alpha$ consist of all cycles in 
    $\fullmixedcycspace$ that share an entire edge with $\alpha$
    (and possibly other edges and prevertices as well), and let
    $\fullmixedcycspace_\alpha'$ consist of all cycles that share no edges
    with $\alpha$ but do share Hamiltonian vertices.
    
    First, we show that
    \begin{align}
      \sum_{\alpha\in\fullmixedcycspace}\sum_{\beta\in\fullmixedcycspace_\alpha}
        \E I_\alpha I_\beta = O\left(\frac{(d-1)^{2r-1}}{n}\right).
      \label{eq:varoverlap}
    \end{align}
    
    \newcommand{\salpha}{\tilde{\alpha}}
    \newcommand{\sbeta}{\tilde{\beta}}
    \newcommand{\sgamma}{\tilde{\gamma}}
    Let $\salpha$ and $\sbeta$ denote graph cycles without edge colors 
    or labels.
    Let $\sgamma$ be the subgraph made up of edges common to $\salpha$
    and $\sbeta$.
    Suppose that $\sgamma$ consists of $p$ paths, with a total of $l$
    edges.  First, we count how many possible $\salpha$ and $\sbeta$
    can give rise to $\sgamma$ with these properties.
    
    Fix a choice of $\salpha$ and $\sgamma\subseteq\salpha$,
    and we will determine how many possible $\sbeta$ there are.
    Let the components of $\sgamma$ be $A_1,\ldots,A_p$, in the order
    they appear in $\salpha$.  To construct $\sbeta$, imagine laying
    out these components, with $A_2,\ldots,A_p$ ordered and oriented
    any way, a total of $2^{p-1}(p-1)!$ choices.  Then, we will create
    $\sbeta$ by filling in the gaps between these components.
    Each gap between components must contain at least one edge, and
    there are a total of $j-l$ edges to add.  So, the number of possible
    gap sizes is $\binom{j-l-1}{p-1}$, the number of compositions
    of $j-l$ into $p$ parts.
    This creates $j-p-l$ new vertices, and we have less than
    $n^{j-p-l}$ choices for these.  Thus, for this fixed $\salpha$
    and $\sgamma$, there are at most
    \begin{align*}
      2^{p-1}(p-1)!\binom{j-l-1}{p-1}n^{j-p-l}
    \end{align*}
    choices of $\sbeta$.
    
    We choose $\salpha$ from the $\ff{n}{k}/2k< n^k/2k$ possible
    $k$-cycles (without edge labels).
    To count how many $\sgamma\subseteq\salpha$ we can form with $p$ components
    and $l$ total edges, fix a vertex in $\salpha$.
     Then, we can
       specify which edges to include in $\sgamma$ by giving a sequence
       $a_1,b_1,\ldots,a_p,b_p$ instructing us
       to include in $\sgamma$ the first $a_1$ edges after the vertex,
       then to exclude the next $b_1$, then to include the next $a_2$,
       and so on.  Any sequence for which $a_i$ and $b_i$
       are positive integers, $a_1+\cdots+ a_p=l$,
       and $b_1+\cdots+b_p=k-l$ gives us a valid choice of $l$ edges of 
       $\salpha$
       making up $p$ components.  This counts each subgraph $\sgamma$ a total of
       $p$ times, since we could begin with any component of $\sgamma$.
       Hence the number of subgraphs $\sgamma$ with $l$ edges
       and $p$ components is $(k/p)\binom{l-1}{p-1}\binom{k-l-1}{p-1}$.
       In all, there are at most
    \begin{equation}\label{eq:numoverlap}
      \begin{split}
      2^{p-1}(p-1)!\binom{j-l-1}{p-1}&n^{j-p-l}
      \frac{n^k}{2k}
      (k/p)\binom{l-1}{p-1}\binom{k-l-1}{p-1}\\
      &\leq \frac{1}{2p}\left(\frac{2e^2r^3}{(p-1)^2}\right)^{p-1}
      n^{k+j-p-l}
      \end{split}
    \end{equation}
    pairs of cycles $\alpha$ and $\beta$ such that their edges intersect
    to form $p$ paths with a total of $l$ edges.

    Now, we bound all the ways to add edge colors and labels to $\salpha$
    and $\sbeta$ to form $\alpha$ and $\beta$, and the probability
    that $\alpha\cup\beta$ appears in $(\rHam,\rP)$.
    Let $B_1,\ldots,B_p$ be the components of $\sbeta$ that do not
    overlap with $\salpha$, with $B_i$ appearing immediately after $A_i$.
    Let $B_i$ have $m_i$ edges, with $\sum m_i = j-l$.

    Every edge coloring of $\salpha \cup \sbeta$ gives rise to an
    $f \in \{B,R\}^{k},$ the edge coloring of $\salpha$ and 
    $\tilde f_i \in \{B, R\}^{m_i},$ the edge coloring of $B_i$ 
    for $1 \leq i \leq p.$  Define $f_i$ from $\tilde f_i$ by prepending and
    appending a $B$ to each side of $f_i,$ so that $f_i \in B\{B, R\}^{m_i}B.$
    Any edge coloring of $\alpha \cup \beta$ is determined by the
    colorings $f,f_1,\ldots, f_p.$ 

    Let $r_1$ and $r_2$ be the number of times the $RB$ and $RR$ patterns
    respectively
    occur in $f$.  
    Let $r_1^{(i)}$ and $r_2^{(i)}$ be the same quantities for $f_i$.  
    Let $R_1=r_1+\sum r_1^{(i)}$ and $R_2=r_2+\sum r_2^{(i)}$.
    
    We will now turn this procedure around and edge--color $\alpha\cup\beta$ 
    according to the colorings $f, f_1,\ldots,f_p$ as described above.  
    This may not yield a possible coloring of $(\rHam, \rP).$ 
    For example, it may be that $\alpha \cup \beta$ has three blue edges
    incident to a vertex.  For any such nonsense
    coloring, we may simply take $\E I_\alpha I_\beta = 0.$

    Regardless, the subgraph $\alpha\cup\beta$ has a total
    of $R_1+R_2$ red edges and $j+k-l-R_1-R_2$ blue edges.
    It contains at most $R_1+p$ disjoint blue paths.
    Thus for \emph{every} coloring of $\alpha \cup \beta$ attained in this way,
    we have,
    \begin{align}
      \E I_\alpha I_\beta &\leq \frac{2^{R_1+p}}{\ff{n-1}{j+k-l-R_1-R_2}
        \dff{n(d-2)}{R_1+R_2}}\nonumber\\
        &\leq \frac{2^{R_1+p}}{n^{j+k-l}(d-2)^{R_1+R_2}}e^{O(r^2/n)}\label{eq:pp}
    \end{align}
    by Lemma~\ref{lem:ffstirling}.

    Next, we determine how many different ways we can add edge labels to $\alpha$
    and $\beta$, given the edge colors.
    There will be a total of $2R_1+2R_2$ edge labels to assign, two for each
    red edge. Imagine walking 
    around $\alpha\cup\beta$ assigning edge labels to its red edges. 
    At each end of a red edges, we have $d-2$ choices of edge labels if no adjacent
    red edge has had an edge label assigned yet, and we have $d-3$
    choices or fewer otherwise.
    For each $RR$ in $f, f_1,\ldots,f_p$,  there is a vertex label with
    $d-3$ or fewer choices. There are $R_2$ of these, so for at most
    $2R_1+R_2$ of the edge labels do we have $d-2$ choices. So, the number
    of ways to assign edge labels is at most $(d-3)^{R_2}(d-2)^{2R_1+R_2}$.
    (Note that this analysis works even in the $d=3$ case, when containing the pattern
    $RR$ implies the $\alpha\cup\beta$ cannot occur.)
    
    Multiplying this by the bound given in \eqref{eq:pp}, the sum of $\E I_\alpha
    I_\beta$ as $\alpha$ and $\beta$ range over all cycles formed by adding
    edge labels to $\salpha$ and $\sbeta$ is at most
    \begin{align*}
      &\frac{2^pe^{O(r^2/n)}}{n^{j+k-l}}\sum_{f,f_1,\ldots,f_p}\big(2(d-2)\big)^{R_1}(d-3)^{R_2}\\
        &\qquad\qquad= \frac{2^pe^{O(r^2/n)}}{n^{j+k-l}}
           p_k\big(d-3,2(d-2)\big)\prod_{i=1}^pp^B_{m_i+1}\big(d-3,2(d-2)\big)\\
        &\qquad\qquad= \frac{2^pe^{O(r^2/n)}}{n^{j+k-l}}
           O\big((d-1)^k\big)\prod_{i=1}^p \frac52(d-1)^{m_i}\\
        &\qquad\qquad= O(1)\left(\frac{d-1}{n}\right)^{j+k-l}
                   5^p
                   e^{O(r^2/n)}.
    \end{align*}
    We can assume without loss of generality that 
    $r\leq n^{1/10}$, since the proposition holds
    for all $r>n^{1/10}$ by choosing
    $\Cr{C:lneatprob}$ sufficiently large.
    This means that the $e^{O(r^2/n)}$ factor can be absorbed into the $O(1)$ factor.
    Applying \eqref{eq:numoverlap} and summing over all $k$, $j$,
    $p$, and $l$, we have
    \begin{align*}
      \sum_{\alpha\in\fullmixedcycspace}&\sum_{\beta\in\fullmixedcycspace_\alpha}
        \E I_\alpha I_\beta \\
        &\leq \sum_{k,j=1}^r\sum_{l\geq 1}\sum_{p\geq 1}
          \frac{1}{2p}\left(\frac{2e^2r^3}{(p-1)^2}\right)^{p-1}
      n^{k+j-l-p} O\left(5^p\left(\frac{d-1}{n}\right)^{j+k-l}\right)
                 \\
        &\leq \sum_{k,j=1}^r\sum_{l\geq 1}O\left(\frac{(d-1)^{j+k-l}}{n}\right)\sum_{p\geq 1} 
          \frac{1}{2p}\left(\frac{10e^2r^3}{(p-1)^2n}\right)^{p-1}.
    \end{align*}
    Our assumption that $r\leq n^{1/10}$ implies that the innermost sum is bounded
    by an absolute constant, proving \eqref{eq:varoverlap}.
    
    Now, it only remains to bound $\sum_{\alpha\in\fullmixedcycspace}
    \sum_{\beta\in\fullmixedcycspace_\alpha'}\E I_\alpha I_\beta$.
    If $\beta\in\fullmixedcycspace_\alpha'$, then $\E I_\alpha I_\beta
    \leq p_\alpha p_\beta$, so we will use this as a summand.
    We enumerate all $\beta\in\fullmixedcycspace_\alpha'$.
    Choose some Hamiltonian vertex~$v$ of $\alpha$ and an orientation 
    for $\alpha$.
    We will count all possible ways of constructing a rooted, oriented cycle 
    $\beta$ that also contains $v$.
    Fix an edge coloring $f\in B\{B,R\}^{j-1}$ and let 
    $r_1$ and $r_2$ be the values from
    Lemma~\ref{lem:cycleprob} corresponding to $f$.
    There are at most
    $n^{j-1}(d-2)^{2r_1+r_2}(d-3)^{r_2}$ ways to fill in the remaining vertices
    and edge labels of $\beta$.
    Suppose that $\alpha$ contains $\phi$ prevertices.  Then,
    \begin{align*}
      \sum_{\substack{\beta\in\fullmixedcycspace_\alpha',\\
                      \abs{\beta}=j}} p_\beta 
        &\leq (2k-\phi)\smashoperator{\sum_{f\in B\{B,R\}^{j-1}}}
          n^{j-1}(d-2)^{2r_1+r_2}(d-3)^{r_2}
          \frac{2^{r_1}}{\ff{n-1}{j-r_2-r_1}\dff{n(d-2)}{r_2+r_1}}\\
          &\leq  (2k-\phi)p_j^{B}
           (d-3,2(d-2)) O(n^{-1})\\
          &=O\left( \frac{k(d-1)^{j-1}}{n}\right).
    \end{align*}
    Now we have
    \begin{align}
      \sum_{\alpha\in\fullmixedcycspace}\sum_{\beta\in\fullmixedcycspace_\alpha'}
      p_\alpha p_\beta
        &\leq \sum_{\alpha\in\fullmixedcycspace} p_\alpha
         \sum_{j=1}^r O\left( \frac{\abs{\alpha}(d-1)^{j-1}}{n}\right)\nonumber\\
        &=\sum_{\alpha\in\fullmixedcycspace} p_\alpha
          O\left( \frac{\abs{\alpha}(d-1)^{r-1}}{n}\right).\nonumber\\
        &= O\left(\frac{(d-1)^{2r-1}}{n}\right).\label{eq:varpapb}
    \end{align}
    This and \eqref{eq:varoverlap} combine to show
    that $\P[\Overlap] = O((d-1)^{2r-1}/n)$.
    
    Now, we bound $\P[\Many]$.
    We say that $S\subseteq\fullmixedcycspace$ is a minimal bad set if it 
    consists of non-overlapping cycles with a total of either
    more than $\lambda(d-1)^r$ prevertices or more than
    $\lambda(d-1)^{r-1}$ Hamiltonian vertices, and if no proper subset of
    $S$ has this property.

    \newcommand{\ro}[1]{r_1^{(#1)}}
    \newcommand{\rt}[1]{r_2^{(#1)}}
    Let $\ro{\alpha}$ and $\rt{\alpha}$ be the number of $RB$s
    and $RR$s in the color pattern of $\alpha$, as in
    Lemma~\ref{lem:cycleprob}, and let
    \begin{align*}
      \ro{S} = \sum_{\alpha\in S}\ro{\alpha},\qquad\qquad
      \rt{S} = \sum_{\alpha\in S}\rt{\alpha}.
    \end{align*}
    Let $\abs{S}$ denote the total number of edges in $S$.
    Let $p_S$ denote the probability that 
    $(\rHam, \rP)$ contains every cycle in $S$.
    The total number of prevertices in
    a minimally bad set is at most $\lambda(d-1)^r+2r$, and the number
    of Hamiltonian vertices is at most $\lambda(d-1)^{r-1}+r$.
    Thus for a minimal bad set $S$, the total number of red edges,
    $\ro{S}+\rt{S}$, and the total number of blue edges, $\abs{S}-
    \ro{S}+\rt{S}$, satisfy
    \begin{align*}
      \ro{S}+\rt{S} &\leq \frac{\lambda}{2}(d-1)^r +r,\\
      \abs{S}-\ro{S}-\rt{S} &\leq \lambda(d-1)^{r-1}+r-1.
    \end{align*}
    By Lemma~\ref{lem:ffstirling},
    \begin{align*}
      p_S &= \frac{2^{\ro{S}}}{\ff{n-1}{\abs{S}-\ro{S}-\rt{S}}
      \dff{n(d-2)}{\ro{S}+\rt{S}}}\\
        &= \frac{2^{\ro{S}}}{n^{\abs{S}}(d-2)^{\ro{S}+\rt{S}}}
        \exp\left(O\left(\frac{\lambda^2(d-1)^{2r-2}}{n}\right)
         +O\left(\frac{\lambda^2(d-1)^{2r}}{n(d-2)}\right)\right)\\
       &\leq \frac{2^{\ro{S}}}{n^{\abs{S}}(d-2)^{\ro{S}+\rt{S}}}
         \exp\left( \frac{\Cl{C:experr2}\lambda^2(d-1)^{2r-1}}
           {n}
         \right)
    \end{align*}
    for some absolute constant $\Cr{C:experr2}$.
    
    Our goal is to bound $\sum_S p_S$ as $S$ ranges over all minimal bad 
    subsets of $\fullmixedcycspace$.
    Let $\Poilimit = (\poilimit{\alpha},\alpha\in\fullmixedcycspace)$
    be a vector of independent Poisson random variables with
    \begin{align*}
      \E\poilimit{\alpha} = \frac{2^{\ro{\alpha}}}{n^{\abs{\alpha}}
        (d-2)^{\ro{\alpha}+\rt{\alpha}}},
    \end{align*}
    and let $\mu=\sum_{\alpha\in\fullmixedcycspace}\E\poilimit{\alpha}$.
    \begin{align}
      \P[\Many]\leq 
         \sum_S p_S &\leq  \exp\left( \frac{\Cr{C:experr2}\lambda^2(d-1)^{2r-1}}
           {n}\right)
                  \sum_S \prod_{\alpha\in S} \frac{2^{\ro{\alpha}}}
                  {n^{\abs{\alpha}}(d-2)^{\ro{\alpha}+\rt{\alpha}}}\nonumber\\
                &=          \exp\left( \frac{\Cr{C:experr2}\lambda^2(d-1)^{2r-1}}
           {n}
         \right)
         e^{\mu}\sum_S\P[\Poilimit=\one{S}].\label{eq:mixedtoomanybound}
    \end{align}
    For any cycle $\alpha\in\fullmixedcycspace$, let $\phi_\alpha$
    and $\psi_\alpha$ be the number of prevertices and Hamiltonian vertices,
    respectively, in $\alpha$. 
    Let $s_\alpha$ and $t_\alpha$ be the number of red and blue edges,
    respectively,
    in the color pattern of $\alpha$, as in Section~\ref{subsec:Ecycles}.
    Note that $\phi_\alpha=2s_\alpha$ and $\psi_\alpha \leq 2t_\alpha$.
    So,
    \begin{align*}
      \sum_S \P[\Poilimit=\one{S}] &\leq
       \P\left[\text{$\sum_{\alpha\in\fullmixedcycspace} \phi_\alpha 
         \poilimit{\alpha}>\lambda(d-1)^r$
      or $\sum_{\alpha\in\fullmixedcycspace} \psi_\alpha \poilimit{\alpha}>\lambda(d-1)^{r-1}$}\right]\\
      &\leq \P\left[\text{$\sum_{\alpha\in\fullmixedcycspace} 2s_\alpha \poilimit{\alpha}>\lambda(d-1)^r$
      or $\sum_{\alpha\in\fullmixedcycspace} 2t_\alpha \poilimit{\alpha}>\lambda(d-1)^{r-1}$}
        \right].
    \end{align*}
    Now, we use the modified log-Sobolev inequalities to get a tail
    estimate for each of these sums.
    Let $F(x) = \sum_\alpha 2s_\alpha x_\alpha$ for $x=(x_\alpha,\,\alpha\in
    \fullmixedcycspace)$.
    By the same proof as for~\eqref{eq:Escycles},
    \begin{align*}
      \E F(\Poilimit)=\sum_{\alpha\in\fullmixedcycspace}
       2s_\alpha \E\poilimit{\alpha} \leq \Cl{C:logsob2}(d-1)^r
    \end{align*}
    for some absolute constant $\Cr{C:logsob2}$.
    In the notation of Lemma~\ref{lem:PoissonTail},
    \begin{align*}
      \sum_{\alpha\in\fullmixedcycspace}\norm{\nabla_\alpha F}^2 \E \poilimit{\alpha}
      =
      \sum_{\alpha\in\fullmixedcycspace}(2s_\alpha)^2\E \poilimit{\alpha}
        &\leq  2r\sum_{\alpha\in\fullmixedcycspace} 2s_\alpha \E \poilimit{\alpha}\\
        &= 2r\E F(\Poilimit)\leq 2\Cr{C:logsob2}r(d-1)^r,
    \end{align*}     
    and
    \begin{align*}
      \max_{\alpha\in\fullmixedcycspace}\norm{\nabla_\alpha F} \leq 2r.
    \end{align*}
    By Lemma~\ref{lem:PoissonTail},
    \begin{align}\label{eq:red-logsob}
      \P\left[\text{$\sum_\alpha 2s_\alpha \poilimit{\alpha}>\lambda(d-1)^r$}\right]
        \leq\exp\left( -\frac{(\lambda-\Cr{C:logsob2})(d-1)^r}{4r}
          \log\left(\frac{\lambda}{\Cr{C:logsob2}}
        \right)\right).
    \end{align}
    In the same way, if $G(x)=\sum_\alpha 2t_\alpha x_\alpha$, then
    \begin{align*}
      \E G(\Poilimit)\leq \Cr{C:logsob2}(d-1)^{r-1},
    \end{align*}
    and
    \begin{align*}
      \sum_{\alpha\in\fullmixedcycspace}\norm{\nabla_\alpha G}^2 \E \poilimit{\alpha}=
       \sum_{\alpha\in\fullmixedcycspace}(2t_\alpha)^2\E \poilimit{\alpha} &\leq
      2r\sum_{\alpha\in\fullmixedcycspace}2t_\alpha \E \poilimit{\alpha}\\
        &\leq 2r\E G(\Poilimit)\leq 2\Cr{C:logsob2}r(d-1)^{r-1},
    \end{align*}
    and
    \begin{align*}
      \max_{\alpha\in\fullmixedcycspace}\norm{\nabla_\alpha G}\leq 2r.
    \end{align*}
    Thus by Lemma~\ref{lem:PoissonTail},
    \begin{align}\label{eq:blue-logsob}
      \P\left[\text{$\sum_{\alpha\in\fullmixedcycspace} 2t_\alpha \poilimit{\alpha}>\lambda(d-1)^{r-1}$}\right]
        \leq\exp\left( -\frac{(\lambda-\Cr{C:logsob2})(d-1)^{r-1}}{4r}
          \log\left(\frac{\lambda}{\Cr{C:logsob2}}
        \right)\right).
    \end{align}
    Making sure that we have chosen $\Cr{C:logsob2}$ to be large enough,
    by the proof of \eqref{eq:Ecycles}, we have
    $\mu\leq \Cr{C:logsob2}(d-1)^r/r$.
    We now sum \eqref{eq:red-logsob} and \eqref{eq:blue-logsob} to show
    that
    \begin{equation}\label{eq:logsobmany}
      \begin{split}
      &\P\left[\text{$\sum_{\alpha\in\fullmixedcycspace} \phi_\alpha 
         \poilimit{\alpha}>\lambda(d-1)^r$
      or $\sum_{\alpha\in\fullmixedcycspace} \psi_\alpha \poilimit{\alpha}>\lambda(d-1)^{r-1}$}\right]\\  &\qquad\qquad\qquad\qquad\qquad\leq \exp\left( -\frac{(\lambda-\Cr{C:logsob2})d(d-1)^{r-1}}{4r}
          \log\left(\frac{\lambda}{\Cr{C:logsob2}}
        \right)\right),
      \end{split}
    \end{equation}
    and then substitute this into \eqref{eq:mixedtoomanybound}
    to get
    \begin{align*}
      \P[\Many] &\leq 
        \exp\left[
          -(d-1)^r\left(
            \frac{d(\lambda-\Cr{C:logsob2})}{4(d-1)r}\log\Big(\frac{\lambda}
            {\Cr{C:logsob2}}\Big)
             - \frac{\Cr{C:logsob2}}{r}
             - \frac{\Cr{C:experr2}\lambda^2(d-1)^{r-1}}{n}
        \right)\right].
    \end{align*}
    As with $\eqref{eq:pmdmanybound}$, this is $O(1/n)$
    so long as $\Cr{C:experr2}\lambda^2(d-1)^{r-1}>n$, which
    finishes the proof, together with the bound
    on $\P[\Many]$.
  \end{proof}
  
  \begin{proposition}\label{prop:strictlneat}
    For $d\geq 3$ and all $r$ and $n$,
    \begin{align}
      \P[\text{$\cycprocess(P)$ is not strictly $(\log n)$-neat}]
        &\leq \frac{\Cl{C:strictlneat}(d-1)^{2r}}{n},\label{eq:slneatpmd}\\
        \begin{split}
      \P[\text{$\cycprocess(\phamrv)$ is not strictly $(\log n)$-neat}]
        &=   \P[\text{$\mixedcycprocess$ is not strictly $(\log n)$-neat}]
          \\
        &\leq \frac{\Cr{C:strictlneat}(d-1)^{2r}}{n}.
        \end{split}   \label{eq:slneatpham}
    \end{align}
  \end{proposition}
  
    \begin{proof}
      \newcommand{\Overlap}{\textsc{Overlap}}
      \newcommand{\Many}{\textsc{Many}}
      We only need to make minor changes to the previous proofs.
      To prove \eqref{eq:slneatpmd}, define
      \begin{align*}
        \Overlap &= \{\text{$P$ contains two cycles of length $r$ or less
        sharing a vertex}\}.
      \end{align*}
      Again, we will bound $\sum_{\alpha,\beta}\E I_\alpha I_\beta$
      where $\alpha$ and $\beta$ range over all overlapping pairs
      of cycles.
      We have already shown in \eqref{eq:pmdoverlap}
      that the sum over the cycles that overlap at an entire edge
      is $O\big((d-1)^{2r-1}/n\big)$. For any $k$-cycle
      $\alpha$, the number of $j$-cycles  with a vertex in common
      with $\alpha$ is at most $kn^{j-1}\big(d(d-1)\big)^j$.
      For any pair $\alpha$ and $\beta$ with a vertex in common but
      no edge in common,
      \begin{align*}
        \E[I_\alpha I_\beta] = \frac{1}{\dff{nd}{j+k}}.
      \end{align*}
      Thus the probability that $P$ contains some a cycle $\alpha$
      and a cycle
      $\beta$ overlapping $\alpha$ at a vertex but at no edges is at most
      \begin{align}
        \sum_{k=1}^r\abs{\cycspace[k]}
          \sum_{j=1}^r \frac{kn^{j-1}\big(d(d-1)\big)^j}{\dff{nd}{j+k}}
            &=\sum_{k=1}^r \frac{\ff{n}{k}\big(d(d-1)\big)^k}
               {2k}\sum_{j=1}^r \frac{kn^{j-1}\big(d(d-1)\big)^j}
               {\dff{nd}{j+k}}\nonumber\\
            &= O\left(\frac{(d-1)^{2r}}{n}\right),\label{eq:pmdvertoverlap}
      \end{align}
      proving that $\P[\Overlap] = O\big((d-1)^{2r}/n)$.
      Combined with the bound on $\P[\Many]$ from
      Proposition~\ref{prop:pmdneat}, this proves
      \eqref{eq:slneatpmd}.
        
      The equality in \eqref{eq:slneatpham} holds because
      if $\phamrv$ is given by 
      scrambling the prevertices in each bin of
      $(\rHam, \rP)$, then $\phamrv$ is strictly $\lambda$-neat
      if and only if $(\rHam, \rP)$ is.
      To adjust the proof of Proposition~\ref{prop:mixedneat},
      we just need to change the definition
      of $\fullmixedcycspace'_\alpha$ to be all cycles that
      share no edges with $\alpha$ but do have a vertex in common,
      Hamiltonian or otherwise, and then do the computations
      leading up to \eqref{eq:varpapb} again.
      To enumerate all cycles in $\fullmixedcycspace_\alpha'$,
      first choose any vertex in $\alpha$.
      Let $\beta$ have color pattern $f$, and let $r_1$ and $r_2$ have their
      usual definitions of the number of $RB$s and $RR$s in $f$.
      The number of ways to fill in the remaining vertices and edge labels
      of $\beta$ is at most $n^{j-1}(d-2)^{2r_1+r_2}(d-3)^{r_2}$.
      Thus
      \begin{align*}
      \sum_{\substack{\beta\in\fullmixedcycspace_\alpha',\\
                      \abs{\beta}=j}} p_\beta 
        &\leq k\smashoperator{\sum_{f\in \{B,R\}^{j}}}
          n^{j-1}(d-2)^{2r_1+r_2}(d-3)^{r_2}
          \frac{2^{r_1}}{\ff{n-1}{j-r_2-r_1}\dff{n(d-2)}{r_2+r_1}}\\
          &\leq  kp_j
           (d-3,2(d-2)) O(n^{-1})\\
          &=O\left( \frac{k(d-1)^{j}}{n}\right).
    \end{align*}
    Now,
    \begin{align*}
      \sum_{\alpha\in\fullmixedcycspace}\sum_{\beta\in\fullmixedcycspace_\alpha'}
      p_\alpha p_\beta
        &\leq \sum_{\alpha\in\fullmixedcycspace} p_\alpha
         \sum_{j=1}^r O\left( \frac{\abs{\alpha}(d-1)^{j}}{n}\right)\\
        &=\sum_{\alpha\in\fullmixedcycspace} p_\alpha
          O\left( \frac{\abs{\alpha}(d-1)^{r}}{n}\right).\\
        &= O\left(\frac{(d-1)^{2r}}{n}\right).
    \end{align*}
    The rest of Proposition~\ref{prop:mixedneat} goes through as before.
    \end{proof}

  \subsection{Couplings}\label{sec:couplings}
    We will employ some variations of Stein's method
    that use coupling techniques, so we will need to define
    couplings between conditioned pairings and
    Hamiltonian cycles and their unconditioned counterparts.

      Suppose that a pairing contains
    the edges $a\sim A$ and $b\sim B$.  We can delete these edges
    and replace them with $a\sim b$ and $A\sim B$ to get a new pairing.  
    We call this \emph{switching} the edges $a\sim A$ and $b\sim B$.
    This only makes sense if $a \neq b$, but it is not a problem
    if $A=b$ and $B=a$, in which case the switching has no effect.
    We will use this operation to define couplings of pairings
    $P$ and $P'$:
    \begin{coupling}\label{coup:Pforward}
      Let $P \drawnfrom \pmd$, and fix distinct prevertices $a_1, \ldots, a_k$
      and $b_1, \ldots, b_k$.
      Let $A_1$ and $B_1$ be the (random) prevertices paired with
      $a_1$ and $b_1$, respectively.
      Define $P'$ to be the pairing obtained by the following procedure:
      Switch $a_1\sim A_1$
      and $b_1\sim B_1$.  Let $A_2$ and $B_2$ be the prevertices
      now paired with $a_2$ and $b_2$, and switch 
      $a_2\sim A_2$ and $b_2\sim B_2$.  Repeat for the remaining $a_i$
      and $b_i$.
    \end{coupling}
    \begin{coupling}\label{coup:Pbackward}
      Fix distinct prevertices $a_1,\ldots,a_k$ and $b_1,\ldots,b_k$,
      and let
      $P'$ be distributed as $\pmd$ conditioned to contain the pairs
      $a_i\sim b_i$
      for $1\leq i\leq k$.  Define $P$ as follows:
      Sample $A_k$ uniformly from all prevertices except
      $a_1,\ldots,a_k$ and $b_1,\ldots, b_{k-1}$, and let $B_k$ denote the
      prevertex for which $B_k\sim A_k$.  
      Switch $a_k\sim b_k$ and $A_k\sim B_k$.
      Then sample $A_{k-1}$ uniformly from all prevertices
      except $a_1,\ldots,a_{k-1}$ and $b_1,\ldots, b_{k-2}$, let
      $B_{k-1}\sim A_{k-1}$, and switch
      $a_{k-1}\sim b_{k-1}$ and $A_{k-1}\sim B_{k-1}$. Repeat another
      $k-2$ times.
    \end{coupling}
    \begin{proposition}\label{prop:Pcouple}
      In both couplings, $P\drawnfrom \pmd$ and $P'$ is distributed
      as $P$ conditioned to contain $a_i\sim b_i$ for $1\leq i\leq k$.
      (In fact, these couplings are the same, though this is not
      important to us.)
    \end{proposition}
    \begin{proof}
      Let $\Pairings_i$ be the set of all pairings on $nd$ prevertices
      such that $a_j\sim b_j$ for $j\leq i$, with $\Pairings_0$
      the set of all pairings.  
      Let $p\in\Pairings_{i-1}$, and let $a'\sim a_i$ and
      $b'\sim b_i$.  Define $p'$ from $p$ 
      by switching these two edges, as in Coupling~\ref{coup:Pforward}.
      Let $\varphi_i\colon\Pairings_{i-1}\to
      \Pairings_i$ be given by $\varphi_i(p)=p'$.
            The elements of $\varphi_i^{-1}(p')$ are all given
      by switching the edge $a_i\sim b_i$ in $p'$ with
      some edge other than $a_j\sim b_j$ for $j\leq i$, and other than
      $b_j\sim a_j$ for $j<i$.  (Switching $a_i\sim b_i$ with $b_i\sim a_i$
      does give an element of $\varphi_i^{-1}(p')$, the pairing $p'$ itself.)
      This demonstrates that $\varphi_i^{-1}(p')$ has the same size
      regardless of $p'$, namely $nd-2i+1$.

      Step~$i$ of Coupling~\ref{coup:Pforward} can be interpreted
      as plugging the current random pairing into $\varphi_i$.
      Similarly, step~$i$ (counting backward)
      of Coupling~\ref{coup:Pbackward} can be interpreted
      as randomly choosing one of the preimages under
      $\varphi_i$ of the current pairing.
      Because each preimage of $\varphi_i$ is the same size,
      at step~$i$ of either coupling, the pairing
      is distributed uniformly on $\Pairings_i$.
    \end{proof}
     We now consider similar couplings for random Hamiltonian cycles.
  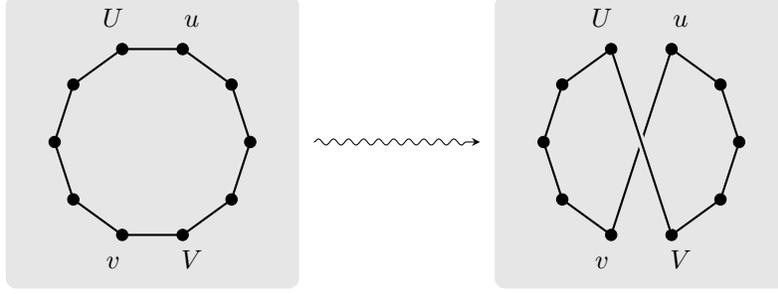
\begin{figure}
    \begin{center}
      \begin{tikzpicture}[scale=1.3,>=stealth]
        \begin{scope}
          \fill[black!10,rounded corners] (-1.5,-1.5) rectangle (1.5,1.5);

          \foreach \deg [count=\i] in {0, 36, ..., 324}
            \node[vert] (v\i) at (\deg: 1) {};
          \path (72:1.3) node[anchor=mid] {$u$}
                (108:1.3) node[anchor=mid] {$U$}
                (-72:1.3) node[anchor=mid] {$V$}
                (-108:1.3) node[anchor=mid] {$v$};
          \draw[thick] (v1) -- (v2) -- (v3) -- (v4) -- (v5) -- (v6)
                     -- (v7) -- (v8) -- (v9) -- (v10) -- (v1);
        \end{scope}
        \draw[->, decorate, decoration={snake,amplitude=.4mm,
          segment length=2mm,post length=1.5mm}] 
          (1.65, 0)--+(1.7,0);
        \begin{scope}[shift={(5,0)}]
          \fill[black!10,rounded corners] (-1.5,-1.5) rectangle (1.5,1.5);

          \foreach \deg [count=\i] in {0, 36, ..., 324}
            \node[vert] (v\i) at (\deg: 1) {};
          \path (72:1.3) node[anchor=mid] {$u$}
                (108:1.3) node[anchor=mid] {$U$}
                (-72:1.3) node[anchor=mid] {$V$}
                (-108:1.3) node[anchor=mid] {$v$};
          \draw[thick] (v1) -- (v2) -- (v3) -- (v8) -- (v7) -- (v6)
                     -- (v5) -- (v4)  (v9) -- (v10)  -- (v1);
          \draw[preaction={draw, line width=3pt, black!10}, thick] (v4) -- (v9);
        \end{scope}
      \end{tikzpicture}
    \end{center}
    \caption{Switching the edges $uU$ and $vV$. These cycles are oriented
    counter-clockwise.}
    \label{fig:hamcond}
  \end{figure}
  Suppose that we have a Hamiltonian cycle $v_1\cdots v_n$.
  We define a switching of the edges $v_iv_{i+1}$ and $v_jv_{j+1}$ as the action of
    deleting these edges and replacing them with $v_iv_j$ and $v_{i+1}v_{j+1}$, 
    as in
    Figure~\ref{fig:hamcond}.  For $i< j$, this leaves us with the cycle
    $v_1\cdots v_iv_jv_{j-1}\cdots v_{i+1}v_{j+1}\cdots v_n$.
    We now give a procedure to successively condition a random Hamiltonian cycle
    to contain given paths.
    \begin{coupling}\label{coup:Hforward}
      Let
      $P_1, \ldots, P_p$ be a set of disjoint paths of vertices
      from $\{1,\ldots,n\}$.
      Let $\rHam$ be a random Hamiltonian cycle conditioned to contain
      the paths $P_1,\ldots,P_{p-1}$. Let $P_p=v_1\cdots v_l$.
      Assign an orientation to $\rHam$, choosing each with probability $1/2$.
      Suppose that $v_1x_1$ and $v_2y_1$ are edges in $\rHam$ consistent with this orientation
      and switch $v_{1}V_1$ and $v_{2}V'_1$, performing the transformation
      \begin{align}
        v_1x_1\cdots x_k v_2 y_1\cdots y_j\mapsto
        v_1v_2x_k\cdots x_1 y_1\cdots y_j. \label{eq:hamtrans}
      \end{align}
      This resulting cycle is oriented according to the right hand side in the above equation. 
      Hence it contains the edges $v_1v_2$ and $v_2x_k.$ 
      It also contains $v_3z$ for some $z.$
      Now, switch the edges $v_2x_k$ and $v_3z$, and then proceed in this way
      for the remaining edges of $P_p$. Let $H'$ be the resulting Hamiltonian cycle.
    \end{coupling}
    \begin{coupling}\label{coup:Hbackward}
      Let
      $P_1, \ldots, P_p$ be a set of disjoint paths of vertices
      from $\{1,\ldots,n\}$.
      Let $\rHam'$ be distributed as a uniformly random Hamiltonian
      cycle conditioned to contain these paths.
      Let $P_p=v_1\cdots v_l$.
      Define $\rHam$ by the following algorithm:
      With probability $1/2$, replace $P_p$ by $v_l\cdots v_1$.
      Orient $\rHam$ so that $v_1\cdots v_l$ is found in that order, and let
      $\rHam$ be $x_1\cdots x_n$ with this orientation. 
      Considering indices modulo $n$,
      choose $I$ uniformly
      from the indices $i\in\{1,\ldots,n\}$ such that $x_ix_{i+1}$ is not contained 
      in any of $P_1,\ldots,P_p$.
      Switch $x_ix_{i+1}$ and $v_{l-1}v_l$.
      Relabel the resulting cycle $x_1\cdots x_n$.
      Repeat the procedure by sampling a new index $I$ such that $x_Ix_{I+1}$
      is not contained in $P_1,\ldots,P_{p-1}$
      nor $v_1\cdots v_{l-1}$, and then switching $x_Ix_{I+1}$ with $v_{l-2}v_{l-1}$.
      Repeat until all the edges of $P_p$ have been switched.
      Let $H$ be the resulting Hamiltonian cycle.
    \end{coupling}
    \begin{proposition}\label{prop:Hcouple}
      In both couplings, $\rHam$ is a uniformly random Hamiltonian cycle 
      conditioned to contain
      paths $P_1\cdots P_{p-1}$,
      and $\rHam'$ is distributed as a uniformly random Hamiltonian cycle
      conditioned to contain
      $P_1,\ldots,P_{p}$.
    \end{proposition}
    \begin{proof}
      Let $\Hams_0$ be the set of oriented Hamiltonian cycles containing
      paths $P_1,\ldots,P_{p-1}$, and let $\Hams_{i}$ be the set of oriented Hamiltonian cycles
      containing these paths as well as
      the first $i$ edges of $P_p=v_1\cdots v_l$, ordered in this direction.
      Let $\varphi_0$ be the map on $\Hams_0$ given by performing the switching \eqref{eq:hamtrans} to create
      the edge $v_1v_2$ (this produces a Hamiltonian cycle with a well-defined ordering).
      Observe that $\varphi_0$ maps $\Hams_0$ into $\Hams_1$.
      Let $\varphi_1$ be the map given by performing the next switching, and so on, giving rise
      to the sequence of maps
      \begin{center}
        \begin{tikzpicture}[xscale=1.5]
          \node (H0) at (0,0) {$\Hams_0$};
          \node (H1) at (1,0) {$\Hams_1$};
          \node (dots) at (2,0) {$\cdots$};
          \node (Hl) at (3,0) {$\Hams_l.$};
          \draw[->] (H0) to node[auto] {$\varphi_0$} (H1);
          \draw[->] (H1) to node[auto] {$\varphi_1$} (dots);
          \draw[->] (dots) to node[auto] {$\varphi_{l-1}$} (Hl);
        \end{tikzpicture}
      \end{center}
      We claim that for each $i$, the fibers $\varphi^{-1}_i(h)$ are the same size
      for all $h\in\Hams_{i+1}$.
      Indeed, as in Proposition~\ref{prop:Pcouple}, 
      the elements of $\varphi_{i,j}^{-1}(h)$ are given
      by switching $v_{i}v_{i+1}$ with any edge except those in
      the paths $P_1,\ldots,P_{p-1}$ or in the first $i$ edges of $P_p$.
      
      Coupling~\ref{coup:Hforward} can be seen as choosing a random orientation for $H$
      and then plugging it into $\varphi_0, \ldots,\varphi_{l-1}$ in succession.
      Coupling~\ref{coup:Hbackward} first chooses a random orientation for $H'$. Then, chooses
      uniformly at random from the preimage of $H'$ under $\varphi_{l-1}$, and then from
      the preimage of this under $\varphi_{l-2}$, and so on. It follows from all of the fibers
      having the same size that in each coupling, the uniform measure on $\Hams_i$ is maintained
      at each step.
    \end{proof}
  
  We will apply this collection of couplings to condition a random graph $(H,Q)$ from the unscrambled
  mixed model to contain
  some given cycle. 
  We will refer to the configuration model part and the Hamiltonian
  part of an element $\alpha\in\fullcycspace$, meaning the edges
  of $\alpha$ that come from each respective part of the model.
  We say that two cycles in $\fullcycspace$ \emph{overlap}
  if their configuration
  model parts contain any prevertices in common, or if their Hamiltonian
  cycle parts contain any vertex in common.
  
  \begin{figure}
    \begin{center}
      \begin{tikzpicture}[>=stealth]
         \node[fill=black!10, rounded corners, outer sep=5pt] (before) at (0,4.25)
         {
           \begin{tikzpicture}
             \foreach \i in {1, ..., 7}
               \node[vert, label={[xshift=5pt,yshift=3pt]left:$v_\i$}] (v\i) at (\i, 0) {};
                 
             \node[vert] (w1) at (1,  1) {};
             \node[vert] (w3) at (3,  1) {};
             \node[vert] (w4) at (4,  1) {};
             \node[vert] (w5) at (5,  1) {};
             \path (1, -1) +(-0.175, 0) node[vert] (h1) {}
                           +( 0.175, 0) node[vert] (hh1) {}
                   (2,  1) +(-0.175, 0) node[vert] (w2) {}
                           +( 0.175, 0) node[vert] (ww2) {}
                   (3, -1) +(-0.175, 0) node[vert] (h3) {}
                           +( 0.175, 0) node[vert] (hh3) {}
                   (4, -1) +(-0.175, 0) node[vert] (h4) {}
                           +( 0.175, 0) node[vert] (hh4) {}
                   (5, -1) +(-0.175, 0) node[vert] (h5) {}
                           +( 0.175, 0) node[vert] (hh5) {}
                   (6, -1) +(-0.175, 0) node[vert] (h6) {}
                           +( 0.175, 0) node[vert] (hh6) {}
                   (7, -1) +(-0.175, 0) node[vert] (h7) {}
                           +( 0.175, 0) node[vert] (hh7) {};
             \foreach \i in {1, 3, 4, 5, 6, 7}
               \draw[thick, blue] (h\i) to[bend left] (v\i) 
                                  (v\i) to[bend left] (hh\i);
             \foreach \i in {1, 2, 3, 4, 5}
               \draw[thick, red] (v\i) -- (w\i);
             \draw[thick, red] (v2) -- (ww2);
             
           \end{tikzpicture}
         };
         
         \node[fill=black!10, rounded corners, outer sep=5pt] (after) at (0,0)
         {
           \begin{tikzpicture}
             \foreach \i in {1, ..., 7}
               \node[vert, label={[yshift=-4pt]above:$v_\i$}] (v\i) at (\i, 0) {};
                 
             \node[vert] (w1) at (1,  1) {};
             \node[vert] (w3) at (3,  1) {};
             \node[vert] (w4) at (4,  1) {};
             \node[vert] (w5) at (5,  1) {};
             \path (1, -1) +(-0.175, 0) node[vert] (h1) {}
                           +( 0.175, 0) node[vert] (hh1) {}
                   (2,  1) +(-0.175, 0) node[vert] (w2) {}
                           +( 0.175, 0) node[vert] (ww2) {}
                   (3, -1) +(-0.175, 0) node[vert] (h3) {}
                           +( 0.175, 0) node[vert] (hh3) {}
                   (4, -1) +(-0.175, 0) node[vert] (h4) {}
                           +( 0.175, 0) node[vert] (hh4) {}
                   (5, -1) +(-0.175, 0) node[vert] (h5) {}
                           +( 0.175, 0) node[vert] (hh5) {}
                   (6, -1) +(-0.175, 0) node[vert] (h6) {}
                           +( 0.175, 0) node[vert] (hh6) {}
                   (7, -1) +(-0.175, 0) node[vert] (h7) {}
                           +( 0.175, 0) node[vert] (hh7) {};
             \draw[thick, red] (v1) -- (v2) -- (v3) (v4) -- (v5);
             \draw[thick, blue] (v3) -- (v4) (v5) -- (v6) -- (v7);
             \draw[thick, blue] (v1) to[bend right, looseness=0.5] (v7);
             
             \draw[thick, red] (w1) -- (w2)  
                               (ww2) -- (w3)  
                               (w4) -- (w5);
             \draw[preaction={draw, line width=3pt, black!10},
                   thick, blue, bend left] 
                   (h3) to (v3)
                   (h4) to (v4)
                   (h5) to (v5);
             \draw[thick, blue, bend left]
                   (hh4) to (hh3)
                   (hh6) to (hh5)
                   (h6) to (h7)
                   (hh7) to[looseness=0.4] (h1)
                   (v1) to (hh1);
           \end{tikzpicture}
         };
         \draw[->, decorate, decoration={snake,amplitude=.4mm,
          segment length=2mm,post length=1.5mm}] (before) -- (after);
      \end{tikzpicture}
    \end{center}
    
    \caption{Conditioning a graph formed by a random Hamiltonian cycle
    superimposed on the configuration model to contain a cycle
    $v_1\cdots v_7$.
    Edges from the configuration model are colored red, and edges
    from the Hamiltonian cycle model are colored blue.
    The color pattern
    of the cycle to be created is red, red, blue, red, blue, blue, blue.}
    \label{fig:cycleconditioning}
  \end{figure}
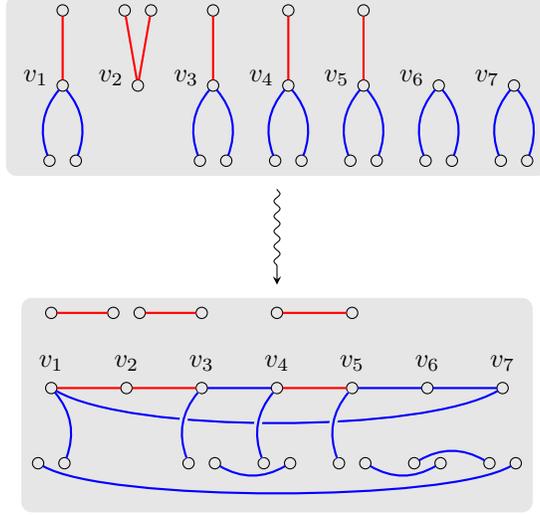

    For any cycle $\alpha\in\fullmixedcycspace$, we define a
  pairing $\rP_\alpha$  conditioned to contain the configuration
  model part of $\alpha$, coupled with $\rP$ 
  according to Coupling~\ref{coup:Pforward}.
  We define a Hamiltonian cycle $\rHam_\alpha$
  conditioned to contain
  the Hamiltonian part of $\alpha$, coupled with $\rHam$
  by successively applying Coupling~\ref{coup:Hforward} 
  to $\rHam$ for each path in the Hamiltonian part of $\alpha$.
  Let $G$ and $G_\alpha$ be the 
  $d$-regular pseudograph given by the projections of
  $(\rHam,\rP)$ and $(\rHam_\alpha,\rP_\alpha)$, respectively.
  See Figure~\ref{fig:cycleconditioning} for an illustration.
  We state some properties of these couplings that
  are apparent from their construction. First, we define a collection
  of subgraphs $K_i$ of $G$, which will be the only parts of $G$ not found in $G_\alpha$.
  \begin{definition}\label{def:K}
    Fix some cycle $\alpha\in\fullmixedcycspace$. Projected onto
    a pseudograph, let its vertices be $v_1,\ldots,v_k$.
    If $v_i$ lies between two red edges in $\alpha$, then there are two prevertices
    $a$ and $b$ used by $\alpha$ at $v_i.$  Each is the endpoint of an edge in
    $\rP.$  Let $K_i$ be the projection of these two red edges and their
    endpoints.
    If $v_i$ lies between two blue edges in $\alpha$, then let $K_i$ be the two 
    blue edges incident to $v_i$ and their endpoints.
    Finally, if $v_i$ lies between one blue edge and one red edge in $\alpha,$
    then let $K_i$ be the two blue edges incident to $v_i,$ the
    red edge labeled by the prevertex used by $\alpha$ and their endpoints.
    These graphs are illustrated by the connected components of the top graph
    in Figure~\ref{fig:cycleconditioning}.
%
%
    
  \end{definition}
  
  \begin{proposition}\label{prop:couplingproperties}\hfill
    \begin{enumerate}[i)]
      \item Suppose that $(\rHam,\rP)$ contains the cycle 
        \label{item:removedcycle}
        $\beta$, but $(\rHam_\alpha,\rP_\alpha)$ does not.
        Then $\beta$ and $\alpha$ overlap.
      \item 
        Suppose that some edge is
        present in $(\rHam_\alpha,\rP_\alpha)$ \label{item:clusters}
        but not in $(\rHam,\rP)$. Then this edge is either
        contained in $\alpha$, or its projection to $G$ is an
        edge 
        between a vertex in $K_i$ and a vertex in $K_{i+1}$ 
        for some $i$ (considering indices 
        modulo~$\abs{\alpha}$).
    \end{enumerate}
  \end{proposition}
  \begin{proof}
  For the first claim, note that the only edges that are destroyed by
  the coupling have some $v_i$ as an endpoint.  For the second, the only edges
  that are created by the coupling appear between some $K_i$ and $K_{i+1}$ at 
  the corresponding step in the coupling algorithm.
  \end{proof}

  \subsection{Poisson approximation with multiplicative bounds}
  \label{subsec:multbounds}
  The usual goal in Poisson approximation is to bound the total
  variation distance between some distribution $\mu$
  and the Poisson distribution. This gives an estimate
  of the point probabilities $\mu(k)$ with a uniform, additive error.
  We, on the other hand, want an approximation of these point
  probabilities in which the error term is relative to the size of
  $\mu(k)$.
  
  First, we present a framework for this form of approximation,
  echoing the one given in \cite{BHJ}.
  Let $F=(F_\alpha,\,\alpha\in\gis)$ be a vector of Bernoulli random
  variables, and let $p_\alpha= \E F_\alpha$.
  Suppose that
  we have a family of random vectors
  $\conditionalprocess = (J_{\beta\alpha},\,\beta\in\gis)$, each coupled
  with $F$, such that
  $\conditionalprocess$ is distributed as $F$ conditioned on $F_\alpha=1$.
  Let $e_\alpha$ denote the standard basis vector equal to one at
  position $\alpha$ and zero elsewhere.
  Define $E_\alpha$ to be the event that $\conditionalprocess = F+e_\alpha,$
  i.e.\ the vectors $\conditionalprocess$ and $F$ are identical except that $F_\alpha=0$
  and $J_{\alpha\alpha}=1$.
  The idea is that bounds on conditional probabilities of $E_\alpha$
  can be turned into 
  estimates on point probabilities relative to each other.
  \begin{lemma}\label{lem:perturb}
    Let $x=(x_\beta,\,\beta\in\gis)$ be a vector of zeros and ones.
    For some $\alpha\in\gis$ with $x_\alpha=1$, suppose that
    $\conditionalprocess$ is coupled with $F$ and distributed as described above.
    If
    \begin{align*}
      \P[E_\alpha\mid F=x-e_\alpha]&\geq 1-\eps,\\
      \P[E_\alpha\mid \conditionalprocess=x] &\geq 1-\eps,
    \end{align*}
    then
    \begin{align*}
      (1-\eps)p_\alpha\P[F=x-e_\alpha]
      \leq \P[F=x]\leq (1-\eps)^{-1}p_\alpha\P[F=x-e_\alpha].
    \end{align*}
  \end{lemma}
  
  \begin{proof}
    These inequalities follows directly from the definitions:
    \begin{align*}
      \P[F=x] &= \P[\text{$F_\alpha=1$ and $F=x$}]\\
              &= p_\alpha\P[\conditionalprocess=x]\\
              &\geq p_\alpha\P[\text{$F=x-e_\alpha$ and $E_\alpha$}]\\
              &=p_\alpha\P[F=x-e_\alpha]\P[E_\alpha\mid F=x-e_\alpha]\geq (1-\eps)
              p_\alpha \P[F=x-e_\alpha],
    \end{align*}
    and
    \begin{align*}
      \P[F=x-e_\alpha] &\geq \P[\text{$\conditionalprocess=x$ and $E_\alpha$}]\\
      &= \P[\conditionalprocess=x]\P[E_\alpha\mid \conditionalprocess=x]\\
      &\geq (1-\eps)\P[\conditionalprocess=x]
      =(1-\eps)p_\alpha^{-1}\P[F=x].\qedhere
    \end{align*}
  \end{proof}
  \begin{remark}
    The approach to Poisson approximation in \cite{BHJ} is to show that
    \begin{align*}
      \E\biggl|\sum_\beta F_\beta + 1 - \sum_\beta J_{\alpha\beta}\biggr|
    \end{align*}
    is always
    small (see \cite[Theorem~1.B]{BHJ}). This is quite similar
    to proving that $\P[E_\alpha]$ is nearly one. 
    The gist of our method is that by bounding conditional versions of this
    probability, we obtain more information.
    
    This method was partially inspired by the use of \emph{switchings} in
    random graph models, which also relate probabilities of slightly
    perturbed events. In \cite[Theorem~2]{MWW}, for example, the authors 
    estimate
    the probability that a random regular graph has no cycles of size $r$
    or less, with a multiplicative error. 
    Similar techniques are used in \cite{Jan09}.
    See Remark~5.6 in that paper for an interpretation of
    switchings as approximate couplings.
    All of these methods are quite similar to a technique of comparing
    relative probabilities via exchangeable pairs described in
    \cite[Section~2]{St92}. 
  \end{remark}

  Our goal now is to apply Lemma~\ref{lem:perturb} to the cycle processes
  $\cycprocess(P)$ and
    $\mixedcycprocess$ defined on p.~\pageref{page:mixedcycprocessdef}.
  For some absolute constant $\Cr{C:upwardbound}$, the following
  two propositions hold:
  \begin{proposition}\label{prop:upwardbound}
    Fix some $\alpha\in\fullmixedcycspace$, and let
    $x$ be $\lambda$-neat with $x_\alpha=1$.
    Let $\conditionalprocess$ be distributed as $\mixedcycprocess$
    conditioned on $I_\alpha=1$.
    Then $\mixedcycprocess$ and $\conditionalprocess$ can be coupled with
    \begin{align*}
      \P[E_\alpha\mid \mixedcycprocess=x-e_\alpha] &\geq 
        1 - \frac{\Cl{C:upwardbound}\lambda\abs{\alpha}(d-1)^{r-1}}{n}.
    \end{align*}
  \end{proposition}
  \begin{proposition}\label{prop:downwardbound}
    Fix some $\alpha\in\fullmixedcycspace$, and let
    $x$ be $\lambda$-neat with $x_\alpha=1$.
    Let $\conditionalprocess$ be distributed as $\mixedcycprocess$
    conditioned on $I_\alpha=1$.
    If 
    Then $\mixedcycprocess$ and $\conditionalprocess$ can be coupled with
    \begin{align*}
      \P[E_\alpha\mid \conditionalprocess=x] &\geq 
        1 - \frac{\Cr{C:upwardbound}\lambda\abs{\alpha}(d-1)^{r-1}}{n}.
    \end{align*}
  \end{proposition}
  \begin{proof}[Proof of Proposition~\ref{prop:upwardbound}]
    Take $(\rHam,\rHam_\alpha)$ from Coupling~\ref{coup:Hforward}.
    We take $(\rP,\rP_\alpha)$ from Coupling~\ref{coup:Pforward},
    unless $\alpha$ is a loop.  In this case,
    we use a slight variation of Coupling~\ref{coup:Pforward}, depicted in
    Figure~\ref{fig:revisedcoupling}.
    \label{par:revisedcoupling}
    Let $\alpha$ be made up of prevertices $a\sim b$, with $a$
  and $b$ belonging to the same vertex.  If $a\sim b$ already in $\rP$, 
  then let $\rP_\alpha=\rP$.  Otherwise, suppose that $A\sim a$ and $B\sim b$ 
  in $\rP$, and choose a prevertex $A'$ uniformly
  out of the all prevertices other than $a$, $b$, $A$, and $B$.  Let $B'\sim B$
  in $\rP$.  To form $\rP_\alpha$, delete $a\sim A$,
  $b\sim B$, and $A'\sim B'$, and replace them with $a\sim b$, $A\sim A'$,
  and $B\sim B'$.  It is straightforward to check that $\rP_\alpha$
  is distributed as $\rP$ conditioned on containing $\alpha$.
  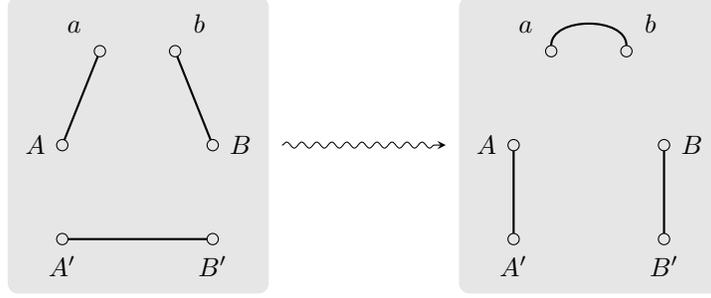
\begin{figure}
    \begin{center}
      \begin{tikzpicture}[>=stealth,label distance=-4pt]
         \node[fill=black!10, rounded corners, outer sep=5pt] (before) at (0,0)
         {
           \begin{tikzpicture}
             \path (-0.5, 0) node[vert,label=above left:$a$] (a) {}
                   ( 0.5, 0) node[vert,label=above right:$b$] (b) {}
                   (-1, -1.25) node[vert,label=left:$A$] (A) {}
                   ( 1, -1.25) node[vert,label=right:$B$] (B) {}
                   (-1, -2.5) node[vert,label=below:$A'$] (Ap) {}
                   ( 1, -2.5) node[vert,label=below:$B'$] (Bp) {};             
             \draw[thick] (a) -- (A)
                          (b) -- (B)
                          (Ap) -- (Bp);
           \end{tikzpicture}
         };
         
         \node[fill=black!10, rounded corners, outer sep=5pt] (after) at (6,0)
         {
           \begin{tikzpicture}
             \path (-0.5, 0) node[vert,label=above left:$a$] (a) {}
                   ( 0.5, 0) node[vert,label=above right:$b$] (b) {}
                   (-1, -1.25) node[vert,label=left:$A$] (A) {}
                   ( 1, -1.25) node[vert,label=right:$B$] (B) {}
                   (-1, -2.5) node[vert,label=below:$A'$] (Ap) {}
                   ( 1, -2.5) node[vert,label=below:$B'$] (Bp) {};
             \draw[thick] (a) to[bend left,out=90,in=90] (b)
                          (A) -- (Ap)
                          (B) -- (Bp);
                          
           \end{tikzpicture}
         };
         \draw[->, decorate, decoration={snake,amplitude=.4mm,
          segment length=2mm,post length=1.5mm}] (before) -- (after);
      \end{tikzpicture}
    \end{center}
    \caption{If $a$ and $b$ belong to the same vertex, then we use a 
    different coupling.  To condition a pairing to contain the loop $a\sim b$,
      we choose a random edge $A'\sim B'$ and reshuffle edges as
      shown.}\label{fig:revisedcoupling}
  \end{figure}

    Let $k=\abs{\alpha}$.  
    Recall that the edge coloring of $\alpha$ specifies which edges
    of $\alpha$ come from the pairing model part of the graph, and which
    edges come from the Hamiltonian part.
    Suppose that $\beta$ is some other $k$-cycle with the same
    {color pattern}\footnote{Recall two colored cycles have the same color
    pattern if and only if there is a graph isomorphism between them that 
    preserves the edge coloring.}
    as $\alpha$, also disjoint from
    all cycles in $x-e_\alpha$.  By the symmetry of our model and our couplings,
    \begin{align*}
      \P[E_\alpha\mid \mixedcycprocess = x-e_\alpha] = 
        \P[E_\beta\mid \mixedcycprocess=x-e_\alpha].
    \end{align*}
    Taking this one step further, this statement
    still holds if $\beta$ is chosen at random
    from all cycles with the same color pattern as $\alpha$
    that are disjoint from $x-e_\alpha$.
    
    Now,
    let $\xi$ be chosen uniformly from the set of cycles in
    $\mixedcycspace$ that share the color pattern of $\alpha$,
    independent of all other random variables.
    A good way to think of $\xi$ is as the cycle given by
    randomizing all the prevertex and vertex
    labels in $\alpha$. Define the event
    \begin{align*}
      F = \{\text{$\xi$ is disjoint from $x$ and $E_\xi$ holds}\}.
    \end{align*}
    Now $\P[E_\alpha\mid \mixedcycprocess=x-e_\alpha]\geq
    \P[F\mid \mixedcycprocess=x-e_\alpha]$, and it suffices to bound this from below.

    We break up the event $F^c$ into three parts, with
    $F^c\subseteq A_1\cup A_2\cup A_3$:
    \begin{align*}
      A_1 &= \{ \text{$\xi$ overlaps with a cycle in $x$} \},\\
      A_2 &= \{ \text{$J_{\beta\xi}=0$ for some $\beta\neq\xi$ with 
          $I_\beta=1$} \},\\
      A_3 &= \{ \text{$J_{\beta\xi}=1$ for some $\beta\neq\xi$
        with $I_\beta=0$} \}.
    \end{align*}
  
    Let $v_1,\ldots,v_k$ be the vertices of $\xi$, with the starting
    vertex and orientation of the cycle
    arbitrarily fixed.  By definition of $\xi$, these vertices are randomly
    chosen without replacement from $\{1,\ldots, n\}$.
    Let $\phi$ and $\psi$ be the number of prevertices and Hamiltonian
    vertices in $\xi$, respectively.
    
    To bound the probability of event $A_1$, we observe that each
    prevertex in $\xi$ is marginally uniform over all
    $n(d-2)$ prevertices and each Hamiltonian vertex is marginally uniform
    over $[n]$.  
    Thus the chance that any particular prevertex in $\xi$ matches one 
    found in a cycle in
    $x$ is at most $\Phi/n(d-2)$, 
    and the the chance that any particular Hamiltonian vertex in 
    $\xi$ matches one in $x$ is at most $\Psi/n$, where $\Phi$ is the total number
    of prevertices and $\Psi$ the total number of Hamiltonian vertices
    in $x$, as in
    Definition~\ref{def:lneat}.  Applying
    this to all $2k$ prevertices in $\xi$ with a union bound,
    \begin{align*}
      \P[A_1\mid \mixedcycprocess=x-e_\alpha]
        &\leq \frac{\phi\Phi}{n(d-2)}+\frac{\psi\Psi}{n}
        \leq \frac{\phi\lambda(d-1)^r}{n(d-2)}+\frac{\psi\lambda(d-1)^{r-1}}{n}\\
          &=O\left(\frac{k\lambda(d-1)^{r-1}}{n}\right).
    \end{align*}
    
    We now consider the event $A_2$.
    Assume first that $\xi$ is not a loop, and our usual coupling is in
    effect.
    Suppose that $I_\beta=1$ and $J_{\beta\xi}=0$.
    Then by Proposition~\ref{prop:couplingproperties}\ref{item:removedcycle}, 
    the cycles $\beta$ and $\xi$ have
    a prevertex or a Hamiltonian vertex in common.
    Thus $A_2\subseteq A_1$.
    If  $\xi$ is a loop and our altered coupling
    of $(\rP, \rP_\xi)$ is in effect, then the situation is similar.
    Suppose that $I_\beta=0$ but $J_{\beta\xi}=0$.
    Then either $\xi$  contains
    a prevertex in $\beta$, or the randomly chosen edge $A'\sim B'$
    in $\rP$ used to define $\rP_\alpha$ contains a prevertex in $\beta$.
    In the first case, event $A_1$ holds.  To bound the second
    case, we observe that both prevertices
    $A'$ and $B'$ are marginally distributed uniformly, and the probability
    that one of them matches a prevertex in $x$ is at most
    $2\Phi/n(d-2)$, which is  $O(\lambda(d-1)^{r-1}/n)$ since $x$
    is $\lambda$-neat. 
    In either case,
    \begin{align}\label{eq:A1A2bound}
      \P[A_1\cup A_2 \mid \mixedcycprocess=x-e_\alpha] =
      O\left(\frac{k\lambda(d-1)^{r-1}}{n}\right).
    \end{align}
    
    In the final step, we will bound the event $A_3\cap A_2^c$
    using an approach similar to the switchings argument
    in \cite{MWW}.  Let $G$ and $G_\xi$ be the pseudographs defined
    by $(\rHam, \rP)$ and $(\rHam_{\xi}, \rP_\xi)$ respectively.
    We start with the case that $\xi$ is not a loop.
    Recall the subgraphs $K_1,\ldots,K_k$ from
    Definition~\ref{def:K}.
    We claim that if
    $A_3\cap A_2^c$ holds, then the following event holds:
    \begin{align*}
      B &= \left\{
      \begin{aligned}
        &\text{For some $1\leq i\leq k$ and $1 \leq j\leq r/2$, 
          the distance in $G$}\\
          &\text{between $K_i$ and $K_{i+j}$
          is less than or equal to $r-j$.}
      \end{aligned}\right\}
    \end{align*}
    (Here and in the rest of the argument, we are considering indices
    modulo $k$.)
    Indeed, suppose $A_3\cap A_2^c$ holds, and there exists some
    $\beta\neq\xi$
    with $I_\beta=0$ and $J_{\beta\xi}=1$.
    By Proposition~\ref{prop:couplingproperties}, the only new edges
    in $G_\xi$ not found in $G$ are the ones in $\xi$, and an
    edge between $K_i$ and $K_{i+1}$ for each $i$.
    If $\beta$ is a loop, 
    then it must consist of one of these edges between 
    $K_i$ and $K_{i+1}$, in which case
    event $B$ holds because $K_i$ and $K_{i+1}$ have distance zero.
    Otherwise, $\beta$
    must contain at least one path in $G\cap G_\xi$.
    Suppose it contains only one such path.  The remainder of $\beta$
    is either a single edge between some $K_i$ and $K_{i+1}$, 
    or a portion of $\xi$.
    In both cases, the existence of this path implies
    event $B$.
    If instead $\beta$ contains more than one path in $G\cap G_\xi$,
    then one of them must have length strictly less than $r/2$.
    For some $i$ and some $0\leq j\leq r/2$, this path
    goes between $K_i$ and $K_{i+j}$.
    If $j\geq 1$, then the
    path implies event $B$.
    If $j=0$, then this path begins and ends at vertices in $K_i$.
    Along with either one or two edges present in $G$ by not in
    $G_\xi$, this forms a cycle in $G$. But then $A_2$ holds,
    contradicting our original assumption.
    
    We now estimate the probability of that $B$ occurs.
    Let $d(K, K')$ denote the distance in $G$ between the two subgraphs
    $K$ and $K'$ (that is, the length of the shortest path between a vertex 
    in one
    subgraph and a vertex in another).
    \begin{claim}\label{claim:clustersfar}
    For any $i\neq i'$, 
      \begin{align*}
        \P[d(K_i, K_{i'}) \leq D \mid \mixedcycprocess=x-e_\alpha] 
          = O\left(\frac{(d-1)^D}{n}\right).
      \end{align*}
    \end{claim}
    \begin{proof}
      Since $G$ is $d$-regular, the number of vertices within distance $D$
      of $K_i$ is $O\big((d-1)^D\big)$.
      Even after conditioning on $\rHam$, $\rP$, and $K_i$,
      the vertex $v_{i'}$ is a uniformly random choice out of
      all vertices except $v_i$. Thus the probability that it is within distance
      $D$ of $K_i$ is 
      $O\big((d-1)^D/n\big)$, as is the probability that one of its
      (at most three) neighbors in $K_i$ are within $D$ of $K_i$.
    \end{proof}
     By this claim, 
     \begin{align}
       \P[B\mid\mixedcycprocess=x-e_\alpha]
         &\leq \sum_{i=1}^k\sum_{j=1}^{r/2}O\left(\frac{(d-1)^{r-j}}{n}\right)
           \nonumber\\
         &=O\left(\frac{k(d-1)^{r-1}}{n}\right).\label{eq:halfthird}
     \end{align}
     
    It only remains to bound the probability of the event
    $A_3\cap A_2^c$ when
    $\xi$ is a loop.  Take $a$, $b$, $A$, $B$, $A'$, and $B'$
    as in the definition of the coupling on
    p.~\pageref{par:revisedcoupling}.
    Let $K$ be the subgraph of $G$ induced by the prevertices
    $A$, $a$, $B$, and $b$, and let $K'$ be the subgraph induced by $A'$
    and $B'$.
    We claim that if $A_3\cap A_2^c$ holds, then $d(K,K')\leq r-1$.
    Indeed, suppose that $A_3$ holds and there exists
    some cycle other than $\xi$
    in $G_\xi$ but not in $G$. This cycle
    must use one of the new edges $A\sim A'$
    or $B\sim B'$.  If it uses only one of them, then $G$ contains a path
    of length $r-1$ or less either from $A$ to $A'$ or from $B$ to $B'$,
    and so $d(K,K')\leq r-1$.
    If it uses both of them, then there are two possibilities:
    either $G$ contains
    a path of length $r-1$ or less between $A$ and $B'$ or $B$ and $A'$,
    in which case $d(K,K')\leq r-1$;
    or $G$  contains a cycle of length $r$ or less involving the edge 
    $A'\sim B'$,
    in which case event $A_2$ holds. 

    By the same reasoning as in
    Claim~\ref{claim:clustersfar},
    \begin{align*}
      \P[d(K,K') \leq r-1 \mid \mixedcycprocess=x-e_\alpha] = O\left(
         \frac{(d-1)^{r-1}}{n}\right).
    \end{align*}
    From this and \eqref{eq:halfthird}, we have shown that in all cases
    \begin{align*}
      \P[A_3\cap A_2^c\mid \mixedcycprocess=x-e_\alpha] = 
       O\left(\frac{k(d-1)^{r-1}}{n}\right).
    \end{align*}
    Combining this with \eqref{eq:A1A2bound} completes the proof.
  \end{proof}

  \begin{proof}[Proof of Proposition~\ref{prop:downwardbound}]
    We take 
    $(\rP,\rP_\alpha)$ and $(\rHam,\rHam_\alpha)$ 
    from Couplings~\ref{coup:Pbackward} and \ref{coup:Hbackward},
    respectively.  Let $k=\abs{\alpha}$, and let $\phi$ and $\psi$ be the
    number of prevertices and Hamiltonian vertices, respectively,
    in $\alpha$.  Let $G$ and $G_\alpha$
    be the pseudographs given by $(\rHam,\rP)$ and $(\rHam_\alpha,\rP_\alpha)$.
    
    Event $E_\alpha^c$ can happen in three ways: $G$ still contains
    $\alpha$, it is missing some cycle
    $\beta\neq \alpha$ present in $G_\alpha$, or it contains
    some cycle $\beta\neq\alpha$ not present in $G_\alpha$.
    We define three events $A_1$, $A_2$, and $A_3$ based on this,
    with $E_\alpha^c\subseteq A_1\cup A_2\cup A_3$:
    \begin{align*}
      A_1 &= \{ \text{$\rP$ or $\rHam$ contains some edge of $\alpha$} \},\\
      A_2 &= \{ \text{$I_{\beta}=0$ for some $\beta\neq\alpha$ with 
          $J_{\beta\alpha}=1$} \},\\
      A_3 &= \{ \text{$I_{\beta}=1$ for some $\beta\neq\alpha$
        with $J_{\beta\alpha}=0$} \}.
    \end{align*}
    Thus
    \begin{equation}\label{eq:EaCbound}
        \P[E_\alpha^c\mid \conditionalprocess=x]  
         \leq \P[A_1 \mid \conditionalprocess=x]
           + \P[A_2 \mid \conditionalprocess=x]
           + \P[A_3\cap A_1^c\cap A_2^c \mid \conditionalprocess=x].
    \end{equation}
    We have made event $A_1$ broader than necessary; 
    this will make it easier to bound the last term of this equation.

    At each step of Coupling~\ref{coup:Pbackward}, an edge $a_ib_i$ in $\alpha$ is
    switched with a random edge.  The edge is preserved only if $A_i = a_i.$
    Otherwise, no later switchings can cause it to return.
    Similarly, at each step of Coupling~\ref{coup:Hbackward}, an edge
    of $\alpha$ is switched with a random edge, and there are only two choices
    of this random edge that do not destroy the edge in $\alpha$.
    Thus
    \begin{align}\nonumber
      \P[A_1 \mid \conditionalprocess=x] &\leq 
      \frac{1}{n(d-2)-\phi+1}+\frac{1}{n(d-2)-\phi+3}+\cdots+\frac{1}{n(d-2)-1}\\&
      \qquad {}+\frac{2}{n-\psi} + \frac{2}{n-\psi+1}+\cdots+\frac{2}{n-1}
        \nonumber\\
      &= O\left(\frac{k}{n}\right).\label{eq:unchanged}
    \end{align}
    
    Next, we consider event $A_2$.
    At each step of Couplings~\ref{coup:Pbackward} and \ref{coup:Hbackward},
    an edge of $\alpha$ is switched with a random edge.
    If $J_{\beta\alpha}=1$, then $I_\beta=0$ only if one of these
    random edges contains
    a prevertex or a Hamiltonian vertex in $\beta$.
    This occurs for some prevertex contained in a cycle
    in $x$ with probability at most
    \begin{align*}
      \frac{2\Phi}{n(d-2)-2\phi+1}+\frac{2\Phi}{n(d-2)-2\phi+3}+\cdots+\frac{2\Phi}
      {n(d-2)-1}\leq
      \frac{2\phi\Phi}{n(d-2)-2\phi+1},
    \end{align*}
    and it occurs for some Hamiltonian vertex in a cycle
    in $x$ with probability at most
    \begin{align*}
      \frac{2\Psi}{n-\psi}+\frac{2\Psi}{n-\psi+1}+\cdots+\frac{2\Psi}{n-1}\leq
        \frac{2\psi\Psi}{n-\psi},
    \end{align*}
    where $\Phi$ is the total number
    of prevertices and $\Psi$ the total number of Hamiltonian vertices
    in $x$, as in
    Definition~\ref{def:lneat}.  
    Since $x$ is $\lambda$-neat, we can sum these to get
  \begin{align}
    \P[A_2 \mid \conditionalprocess=x] = O\left(\frac{k\lambda(d-1)^{r-1}}{n}\right).
      \label{eq:cycledestroyed}
  \end{align}

      Last, we consider the event $A_3\cap A_1^c\cap A_2^c$.
      Consider Coupling~\ref{coup:Hbackward} to take place after 
      Coupling~\ref{coup:Pbackward}, so that we can say that there
      are $k$ steps total to go from $G_\alpha$
      to $G$ and number them from $1$ to $k$.
      Suppose we have just taken the $i$th step in the coupling process,
      switching two edges, whether in the pairing part of the graph
      or the Hamiltonian cycle part of the graph.  Suppose that
      $uv$ and $UV$ are the edges deleted, with $uv$ being part of $\alpha$,
      and $uU$ and $vV$ are the edges created.
      We wish to show that it is
      unlikely that a new cycle has formed involving one of the
      new edges $uU$ and $vV$.  
      More precisely, define $C_i$ to be the event
      that a new cycle is formed in the $i$th step, and that it is the first
      new cycle formed by the coupling process.  We will bound the probability
      of $C_i\cap A_1^c \cap A_2^c$ under the assumption that
      $\conditionalprocess=x$.

      The first thing to notice is that we can ignore the possibility of a new
      cycle forming involving both $uU$ and $vV$.
      Suppose that $C_i$ holds, and that the new cycle formed uses both 
      these edges.  Then this cycle either contains paths between
      $u$ and $v$ and between $U$ and $V$, or paths between $u$ and
      $V$ and between $U$ and $v$.  In the first case,
      $uv$ is part of a cycle destroyed when the edge is switched
      with $UV$.
      This cycle must have been present in $G_\alpha$, since if
      $C_i$ holds, then no new cycles have formed before step~$i$
      in the coupling
      process.  Thus event $A_2$ holds.  In the second case, 
      suppose that $P(u,V)$ is the path from $u$ to $V$,
      and $P(v, U)$ is the path from $v$ to $U$.  
      The newly created cycle is $UuP(u,V)vP(v,U)$.
      In the previous step, $uvP(v,U)VP(V,u)$ is a cycle, and the switching
      deletes it.  If this cycle is anything other than $\alpha$,
      then event $A_2$ holds.  If the cycle is $\alpha$, then event
      $A_1$ holds, since this is the only way that $UV$ can be part of
      $\alpha$, assuming that $\conditionalprocess=x$ where $x$ contains no
      cycles that overlap.

      Thus we need only consider the possibility that $C_i\cap A_1^c\cap A_2^c$
      holds because a new cycle forms at step~$i$ involving only one of $uU$ and
      $vV$.  Before step $i$ in the coupling process, there are at most
      $(d-1)^{j-1}$ paths of length $j-1$ starting from $u$ whose first
      step is not $uv$.  The vertex $U$ must be at the end
      of one of these paths if $uU$ is to form a new cycle
      of length $j$.
      This occurs with probability at most $(d-1)^j/(n(d-2)-\phi+1)$ if
      step $i$ is part of the coupling process for $(\rP,\rP_\alpha)$
      and with probability at most $(d-1)^{j-1}/(n-\psi)$ if step
      $i$ is part of the coupling process for $(\rHam,\rHam_\alpha)$.
      The same is true for forming a new cycle involving $vV$.
      Summing this bound over all $j$ from $1$ to $r$,
      \begin{align*}
        \P[C_i\cap A_1^c\cap A_2^c\mid \conditionalprocess=x]
          &= O\left(\frac{(d-1)^{r-1}}{n}\right).
      \end{align*}
      Applying this for $C_1,\ldots,C_k$, we have
      \begin{align}\label{eq:cyclecreated}
        \P[A_3\cap A_1^c\cap A_2^c\mid \conditionalprocess=x]
        &= O\left(\frac{k(d-1)^{r-1}}{n}\right).
      \end{align}
      Applying \eqref{eq:unchanged}, \eqref{eq:cycledestroyed},
      and \eqref{eq:cyclecreated} to \eqref{eq:EaCbound}
      proves the proposition.
    \end{proof}
    We will also need these results in $\pmd$:
  \begin{proposition}\label{prop:updownconfigbound}
    Let $P\drawnfrom \pmd$.  Let $\conditionalprocess$ be distributed
    as $\cycprocess(P)$ conditioned on $I_\alpha=1$.
    Let $x$ be $\lambda$-neat, and let $x_\alpha=1$.
    Then $\cycprocess(P)$ and $\conditionalprocess$ can be coupled with
    \begin{align*}
      \P[E_\alpha\mid \cycprocess(P)=x-e_\alpha] &\geq 
        1 - \frac{\Cr{C:upwardbound}\lambda\abs{\alpha}(d-1)^{r-1}}{n},\\
      \P[E_\alpha\mid \conditionalprocess=x] &\geq
        1 - \frac{\Cr{C:upwardbound}
          \lambda\abs{\alpha}(d-1)^{r-1}}{n}.      
    \end{align*}
  \end{proposition}
  
  \begin{proof}
    The proof of Propositions~\ref{prop:upwardbound} and 
    \ref{prop:downwardbound} go through exactly.
  \end{proof}

  Propositions~\ref{prop:upwardbound}, \ref{prop:downwardbound},
  and \ref{prop:updownconfigbound}
  combine with Lemma~\ref{lem:perturb}
  to give relative estimates on the point probabilities
  of $\mixedcycprocess$ and $\cycprocess(P)$ with 
  $P\drawnfrom\pmd$:
  \begin{cor}\label{cor:relativepointprobs}
    Either suppose that
    $x=(x_\alpha,\,\alpha\in\fullmixedcycspace)$ and
    $\I=\mixedcycprocess$, or suppose that
    $x=(x_\alpha,\,\alpha\in\fullcycspace)$ and $\I=\cycprocess(P)$
    with $P\drawnfrom\pmd$. In either case, suppose that $x$
    is $\lambda$-neat and $\Cr{C:upwardbound}\lambda\abs{\alpha}(d-1)^{r-1}
    \leq n$.
    
    For any $\alpha$ with $x_\alpha=1$,
    \begin{align*}
      (1-\err(x,\alpha))p_\alpha\P[\I=x-e_\alpha]
        \leq \P[\I=x]\leq
        (1-\err(x,\alpha))^{-1}p_\alpha\P[\I=x+e_\alpha],
    \end{align*}
    where $p_\alpha=\E I_\alpha$ and
    \begin{align*}
      \err(x,\alpha) \leq \frac{\Cr{C:upwardbound}
        \lambda\abs{\alpha}(d-1)^{r-1}}{n}.
    \end{align*}
  \end{cor}
  
  By repeated application of this corollary, we can relate
  the probability of any $\lambda$-neat configuration of cycles
  to the probability that the graph contains no cycles at all
  of length $r$ or less:
  \begin{proposition}\label{prop:ratio}
    Either suppose that
    $x=(x_\alpha,\,\alpha\in\fullmixedcycspace)$ and
    $\I=\mixedcycprocess$, or suppose that
    $x=(x_\alpha,\,\alpha\in\fullcycspace)$ and $\I=\cycprocess(P)$
    with $P\drawnfrom\pmd$. 
    Suppose that $\Cr{C:upwardbound}\lambda r(d-1)^{r-1} < n/2$.
    If $x$ is $\lambda$-neat, then
    \begin{align*}
      \exp\left(-\frac{\Cl{C:ratio}\lambda^2(d-1)^{2r-1}}
        {n}\right)\prod_{\alpha\colon x_\alpha=1} p_\alpha
      \leq\frac{\P[\I=x]}{\P[\I=0]}
      \leq\exp\left(\frac{\Cr{C:ratio}\lambda^2(d-1)^{2r-1}}
        {n}\right)\prod_{\alpha\colon x_\alpha=1} p_\alpha
  \end{align*}
  for some absolute constant $\Cr{C:ratio}$.
\end{proposition}
  \begin{proof}
     Since $\Cr{C:upwardbound}\lambda r(d-1)^{r-1} < n/2$,
     \begin{align*}
       \left(1-\frac{\Cr{C:upwardbound}\lambda k(d-1)^{r-1}}{n}\right)^{-1}
       &=\exp\left(O\left( \frac{\lambda k(d-1)^{r-1}}{n}
         \right) \right)
     \end{align*}
     for any $k\leq r$.
     Let $c_k$ be the number of cycles of length $k$ in $x$.
     Let $y=(y_\alpha,\,\alpha\in\fullmixedcycspace)$ or
     $y=(y_\alpha,\,\alpha\in\fullcycspace)$, as appropriate.
     If $x$ is $\lambda$-neat and $y_\alpha\leq x_\alpha$ for all
     $\alpha$, then $y$ is also $\lambda$-neat.  Thus we can
     apply Corollary~\ref{cor:relativepointprobs} repeatedly to
     get
     \begin{align*}
       \frac{\P[\I=x]}{\P[\I=0]}       
       &\leq \prod_{k=1}^r\exp\left(O\left( \frac{\lambda k(d-1)^{r-1}}{n}
         \right) \right)^{c_k}
        \prod_{\alpha\colon x_\alpha=1} p_\alpha\\
      &=\exp\left(O\left(\frac{\lambda(d-1)^{r-1}\sum_{k=1}^r
      kc_k}{n}\right)\right)\prod_{\alpha\colon x_\alpha=1} p_\alpha\\
        &\leq \exp\left(O\left(\frac{\lambda^2(d-1)^{2r-1}}
        {n}\right)\right)\prod_{\alpha\colon x_\alpha=1} p_\alpha.
     \end{align*}
     The lower bound has a nearly identical proof.
  \end{proof}
  \begin{proposition}\label{prop:prob0}
    Either suppose that $\I=\mixedcycprocess$ and $\mu$ is
    the expected number of cycles of length
    $r$ or less in $(\rHam, \rP)$, or suppose that
    $\I=\cycprocess(P)$ with $P\drawnfrom\pmd$ and $\mu$
    is the expected number of cycles of length $r$ or less
    in $P$. In either case, for all $d\geq 3$ and $r,n$ satisfying
    $\Cr{C:lneatprob}(d-1)^{2r-1}<n/2$,
    \begin{align}
      \P[\I=0]
        = \exp \left(-\mu + O\left(\frac{(\log n)^2
          (d-1)^{2r-1}}{n}\right)\right).\label{eq:prob0}
    \end{align}
  \end{proposition}
  \begin{proof}
    Let $\Poilimit=(\poilimit{\alpha})$ be a vector of independent Poisson random
    variables with $\E \poilimit{\alpha}=\E I_\alpha$, and
    with $\alpha$ ranging over $\fullmixedcycspace$
    or $\fullcycspace$ as appropriate.
    Let $\lambda=\log n$, and
    sum the upper bound from Proposition~\ref{prop:ratio} over all
    $\lambda$-neat $x$ to get
    \begin{align*}
      \frac{\P[\text{$\I$ is $\lambda$-neat}]}{\P[\I=0]}
        &\leq \exp\left(\frac{\Cr{C:ratio}\lambda^2(d-1)^{2r-1}}
        {n}\right)
        \sum_{\text{$\lambda$-neat $x$}}\prod_{\alpha\colon x_\alpha=1}
          p_\alpha\\
        &= \exp\left(\frac{\Cr{C:ratio}\lambda^2(d-1)^{2r-1}}
        {n}\right)
        \sum_{\text{$\lambda$-neat $x$}} e^{\mu}\P[\Poilimit=x]\\
        &\leq\exp\left(\frac{\Cr{C:ratio}\lambda^2(d-1)^{2r-1}}
        {n}\right)
        e^{\mu}.
    \end{align*}
    By Proposition~\ref{prop:pmdneat} or \ref{prop:mixedneat},
    \begin{align*}
      1-\frac{\Cr{C:lneatprob}(d-1)^{2r-1}}{n}
        &\leq   \exp\left(\frac{\Cr{C:ratio}\lambda^2(d-1)^{2r-1}}
        {n}\right)
        e^{\mu}\P[\I = 0].
    \end{align*}
    Since $\Cr{C:lneatprob}(d-1)^{2r-1}<n/2$,
    \begin{align*}
      1-\frac{\Cr{C:lneatprob}(d-1)^{2r-1}}{n}
      &= \exp\left(-O\left(\frac{(d-1)^{2r-1}}{n}\right)\right),
    \end{align*}
    and so
    \begin{align}\label{eq:zerolowerbound}
      \P[\I = 0] \geq \exp\left(-\mu + 
      O\left(\frac{(\log n)^2(d-1)^{2r-1}}{n}\right)\right).
    \end{align}
    
    For the other direction, we use the lower bound from
    Proposition~\ref{prop:ratio} to get
    \begin{align*}
      \frac{1}{\P[\I=0]} \geq \frac{\P[\text{$\I$
         is $\lambda$-neat}]}{\P[\I=0]}
           &\geq 
      \exp\left(-\frac{\Cr{C:ratio}\lambda^2(d-1)^{2r-1}}
        {n}\right) e^{\mu}
        \sum_{\text{$\lambda$-neat $x$}} \P[\Poilimit=x]\\
          &= \exp\left(-\frac{\Cr{C:ratio}\lambda^2(d-1)^{2r-1}}
        {n}\right) e^{\mu}\P[\text{$\Poilimit$ is $\lambda$-neat}].
    \end{align*}
    We just need to bound $\P[\text{$\Poilimit$ is $\lambda$-neat}]$.
    To handle the case where $\I=\mixedcycprocess$, see
    Proposition~\ref{prop:mixedneat}, where we considered a Poisson
    field $\Poilimit$ with means differing very slightly from
    the $\Poilimit$ in this proof. This makes no difference,
    and \eqref{eq:logsobmany} applies and shows that
    the probability that $\Poilimit$ fails to be $\lambda$-neat
    on account of containing too many prevertices or Hamiltonian 
    vertices is easily $O(n^{-1})$.
    Similarly, the same argument used in
    \eqref{eq:varpapb} shows that that the probability
    that $\Poilimit$ contains overlapping cycles is
    $O\big((d-1)^{2r-1}/n\big)$. Taking the constant
    here to be $\Cr{C:lneatprob}$ (increasing it if necessary), it follows
    as with \eqref{eq:zerolowerbound} that
    \begin{align*}
      \P[\I = 0] &\leq \exp\left(-\mu + 
      O\left(\frac{(\log n)^2(d-1)^{2r-1}}{n}\right)\right).\qedhere
    \end{align*}
  \end{proof}

  We now put all the pieces together and give the main
  result of this section.
  \begin{proof}[Proof of Proposition~\ref{prop:multpoiapprox}]
    We start with $\cycprocess(P)$, proving
    \eqref{eq:pmdapprox}. Let $\lambda=\log n$.
    By Lemma~\ref{lem:ffstirling},
    \begin{align*}
      \prod_{\alpha\colon x_\alpha=1} p_\alpha
        &= \exp\left(O\left(\frac{\lambda^2(d-1)^{2r-1}}{n}
          \right)\right)
          \prod_{\alpha\colon x_\alpha=1} 
      (nd)^{-\abs{\alpha}}.
    \end{align*}  
    By Proposition~\ref{prop:ratio}, 
    \begin{align}
      \frac{\P[\cycprocess(P)=x]}{\P[\cycprocess(P)=0]}
      &= \exp\left(O\left(\frac{\lambda^2(d-1)^{2r-1}}{n}
          \right)\right)
          \prod_{\alpha\colon x_\alpha=1} 
      (nd)^{-\abs{\alpha}}.\label{eq:finalratio}
    \end{align}
    We wish to replace $\mu$ in \eqref{eq:prob0} with
    $\sum_{\alpha\in\fullcycspace}(nd)^{-\abs{\alpha}}$.
    By Lemma~\ref{lem:ffstirling},
    \begin{align*}
      \sum_{\alpha\in\fullcycspace}(nd)^{-\abs{\alpha}}
        &=\left(1+O\bigg(\frac{r^2}{n}\bigg)\right)\mu.
    \end{align*}
    This together with Proposition~\ref{prop:prob0} proves
    \begin{align}
      \P[\cycprocess(P)=0] &= \exp\left(
        -      \sum_{\alpha\in\fullcycspace}(nd)^{-\abs{\alpha}}
        + O\left(\frac{(\log n)^2
          (d-1)^{2r-1}}{n}\right)\right).\label{eq:changemu}
    \end{align}
    Applying this to \eqref{eq:finalratio}, we have shown that
    \begin{align*}
      \P[\cycprocess(P)=x] &\leq 
      \exp\left(\frac{\frac12\Cr{C:mpa}(\log n)^2
          (d-1)^{2r-1}}{n}\right)\P[\pcycpf=x]\\
          \intertext{and}
      \P[\cycprocess(P)=x] &\geq
             \exp\left(\frac{-\frac12\Cr{C:mpa}(\log n)^2
          (d-1)^{2r-1}}{n}\right)\P[\pcycpf=x]
    \end{align*}
    for some absolute constant $\Cr{C:mpa}$.
    For $\abs{x}<1/2$,
\(
	1+x/2 \leq e^{x/2} \leq 1+x.
\)
    Since $\Cr{C:mpa}(\log n)^2(d-1)^{2r-1} < n/2$, this proves
    \eqref{eq:pmdapprox}.
        
    The proof of \eqref{eq:phammdapprox} is similar, 
    but has a few more complications. The first is that 
    we need to take into account the
    scrambling of the prevertices in each bin in $\phammd$.
    Suppose that $\ell$ is a coloring of the edges of cycles contained
    in $x$. Let $y_\ell=(y_\alpha,\,\alpha\in\fullmixedcycspace)$ consist of
    the cycles in $x$, colored according to $\ell$.
    Let $\P[ \cdot\mid \ell]$ denote probability conditional on
    the prevertex
    scrambling inducing the coloring $\ell$ on the edges in cycles contained
    in $x$. 
    Let $b$ be the total number of vertices in cycles in $y_\ell$ 
    that are incident
    to either one or two blue edges in the cycle.
    Conditional on the coloring $\ell$, there are $2^b$ ways to assign
    prevertices for these Hamiltonian vertices, and thus
    \begin{align*}
      \P\big[\cycprocess(\phamrv)=x \mid \ell\big] &=
        2^{-b}\P\big[\mixedcycprocess=y_\ell\big].
    \end{align*}
    The probability that the cycles in $x$ get colored $\ell$
    by the scrambling is
    \begin{align*}
      \prod_{\alpha\colon x_\alpha=1} \left(\frac{2}{\ff{d}{2}}\right)^{\abs{\alpha}
         -2r_1-r_2}
       \left(\frac{2(d-2)}{\ff{d}{2}}\right)^{2r_1}\left(\frac{\ff{d-2}{2}}
         {\ff{d}{2}}\right)^{r_2},
    \end{align*}
    where $r_1$ and $r_2$ are as in Lemma~\ref{lem:cycleprob}, 
    applied to the color pattern
    of $\alpha$. (This depends on the cycles in $x$ not overlapping
    even at a vertex.) Note that
    \begin{align*}
      b=\sum_{\alpha\colon x_\alpha=1} (\abs{\alpha}-r_2).
    \end{align*}
    Summing over all possible $\ell$, we have
    \begin{align}
      \P[\cycprocess(\phamrv) = x] &= \sum_\ell \rho(\ell)\P[\mixedcycprocess=y_\ell],\label{eq:summing}
    \end{align}
    where
    \begin{align*}
      \rho(\ell) &=
             2^{-b}\prod_{\alpha\colon (y_\ell)_\alpha=1} 
             \left(\frac{2}{\ff{d}{2}}\right)^{\abs{\alpha}
         -2r_1-r_2}
       \left(\frac{2(d-2)}{\ff{d}{2}}\right)^{2r_1}\left(\frac{\ff{d-2}{2}}
         {\ff{d}{2}}\right)^{r_2}.
    \end{align*}
    We would like to apply Proposition~\ref{prop:ratio} to estimate
    $\P[\mixedcycprocess=y_\ell]$, but there is a complication:
    just because $x$ is $\lambda$-neat does not necessarily mean that
    $y_\ell$ is, because it could contain more than $\lambda(d-1)^{r-1}$
    Hamiltonian vertices. The best we can say is that
    $y_\ell$ is $d\lambda$-neat, but using only this bound would
    introduce an extra factor of $d$ in the error term.
    
    \newcommand{\goodcolorings}{\mathcal{C}_{\text{good}}}
    \newcommand{\badcolorings}{\mathcal{C}_{\text{bad}}}
    To deal with this, let $\goodcolorings$ be the set of colorings
    $\ell$ such that $y_\ell$ is $4\lambda$-neat, and let
    $\badcolorings$ be the remaining colorings.
    Let $\err_{\ell}$ be defined by
    \begin{align*}
      \frac{\P[\mixedcycprocess=y_\ell]}{\P[\mixedcycprocess=0]}
        &=\err_{\ell}\prod_{\alpha\colon (y_\ell)_\alpha=1}
        \frac{2^{r_1}}{n^{\abs{\alpha}}(d-2)^{r_1+r_2}}.
    \end{align*}
    If $\ell\in\goodcolorings$, then by Lemma~\ref{lem:cycleprob}
    \begin{align*}
      \prod_{\alpha\colon (y_\ell)_\alpha=1}
        p_\alpha
    &=\exp\left(O\left(\frac{\lambda^2(d-1)^{2r-1}}{n}\right)\right)
      \prod_{\alpha\colon (y_\ell)_\alpha=1}
        \frac{2^{r_1}}{n^{\abs{\alpha}}(d-2)^{r_1+r_2}},
    \end{align*}
    and if $\ell\in\badcolorings$,
    \begin{align*}
      \prod_{\alpha\colon (y_\ell)_\alpha=1}
        p_\alpha
    &=\exp\left(O\left(\frac{\lambda^2(d-1)^{2r}}{n}\right)\right)
      \prod_{\alpha\colon (y_\ell)_\alpha=1}
        \frac{2^{r_1}}{n^{\abs{\alpha}}(d-2)^{r_1+r_2}}.
    \end{align*}
    By this and Proposition~\ref{prop:ratio},
    \begin{align*}
      \err_\ell &= \exp\left(O\left(\frac{\lambda^2(d-1)^{2r-1}}{n}\right)
      \right),& \ell\in\goodcolorings,\\
      \err_\ell &= \exp\left(O\left(\frac{\lambda^2d^2(d-1)^{2r-1}}{n}\right)
      \right),
        &\ell\in\badcolorings.
    \end{align*}
    For ease of presentation, we just show an upper bound on
    $\P[\cycprocess(\phamrv)=x]$. The lower bound
    has an identical proof.
    We first note that
    \begin{align*}
      \rho(l)      \prod_{\alpha\colon (y_\ell)_\alpha=1}
        \frac{2^{r_1}}{n^{\abs{\alpha}}(d-2)^{r_1+r_2}}
        &= \prod_{\alpha\colon (y_\ell)_\alpha=1}
        \frac{\big(2(d-2)\big)^{r_1}(d-3)^{r_2}}
        {\big(n\ff{d}{2}\big)^{\abs{\alpha}}},
    \end{align*}
    and that
    \begin{align*}
        \sum_{\ell} \prod_{\alpha\colon (y_\ell)_\alpha=1}
        \frac{\big(2(d-2)\big)^{r_1}(d-3)^{r_2}}
        {\big(n\ff{d}{2}\big)^{\abs{\alpha}}}
        &=\prod_{\alpha\colon x_\alpha=1}
        \frac{p_{\abs{\alpha}}\big(d-3,2(d-2)\big)-1}
        {\big(n\ff{d}{2}\big)^{\abs{\alpha}}}\\
        &=\prod_{\alpha\colon x_\alpha=1}\left(
        \frac{1}{(nd)^{\abs{\alpha}}} + \frac{(-1)^{\abs{\alpha}}-1}{\big(
        n\ff{d}{2}\big)^{\abs{\alpha}}}\right)\\
        &= e^\mu\P[\phamcycpf = x],
    \end{align*}
    where $\mu=\sum_\alpha \E\tilde{Z}_\alpha$.
    By the same reasoning as \eqref{eq:changemu},
    Proposition~\ref{prop:prob0} holds with its definition of
    $\mu$ changed to this one.
    Applying all of this to \eqref{eq:summing},
    \begin{align}
      \P[\cycprocess(\phamrv) = x] &= \P[\mixedcycprocess=0]
      \sum_\ell \err_{\ell}\rho(\ell)\prod_{\alpha\colon (y_\ell)_\alpha=1}
        \frac{2^{r_1}}{n^{\abs{\alpha}}(d-2)^{r_1+r_2}}\nonumber\\
        \begin{split}
          &\leq  
         \exp\left(-\mu + \frac{\Cr{C:me}\lambda^2d^2(d-1)^{2r-1}}{n}\right)
        \sum_{\ell\in\badcolorings} \prod_{\substack{\alpha\colon\\
           (y_\ell)_\alpha=1}}
        \frac{\big(2(d-2)\big)^{r_1}(d-3)^{r_2}}
        {\big(n\ff{d}{2}\big)^{\abs{\alpha}}}\\
        &\qquad +
        \exp\left(-\mu + \frac{\Cl{C:me}\lambda^2(d-1)^{2r-1}}{n}\right)
        \sum_{\ell} \prod_{\substack{\alpha\colon\\
           (y_\ell)_\alpha=1}}
        \frac{\big(2(d-2)\big)^{r_1}(d-3)^{r_2}}
        {\big(n\ff{d}{2}\big)^{\abs{\alpha}}}
        \end{split}\nonumber\\
        \begin{split}
          &\leq  
         \exp\left(-\mu + \frac{\Cr{C:me}\lambda^2d^2(d-1)^{2r-1}}{n}\right)
        \sum_{\ell\in\badcolorings} \prod_{\substack{\alpha\colon\\
           (y_\ell)_\alpha=1}}
        \frac{\big(2(d-2)\big)^{r_1}(d-3)^{r_2}}
        {\big(n\ff{d}{2}\big)^{\abs{\alpha}}}\\
        &\qquad +
        \exp\left(\frac{\Cr{C:me}\lambda^2(d-1)^{2r-1}}{n}\right)
        \P[\phamcycpf = x]
        \end{split}        \label{eq:twoterms}
    \end{align}
    for some absolute constant $\Cr{C:me}$.
    
    Thus, we need to show that the first term of
    \eqref{eq:twoterms} is negligible compared to the second one.
    Intuitively, this should hold because $\goodcolorings$
    contains the overwhelming majority of colorings.
    More precisely, we will show the following:
    \begin{claim}\label{claim:negligibleclaim}
    \begin{equation*}
      \begin{split}
      \sum_{\ell\in\badcolorings}
      \prod_{\alpha\colon (y_\ell)_\alpha=1}
       \big(2&(d-2)\big)^{r_1}(d-3)^{r_2}\\
       &\leq  e^{-\lambda(d-1)^{r-2}}\sum_{\ell} \prod_{\alpha\colon (y_\ell)_\alpha=1}
        \big(2(d-2)\big)^{r_1}(d-3)^{r_2}.
      \end{split}
    \end{equation*}
    \end{claim}
    \begin{proof}
      When $d=3,4$, the set $\badcolorings$ is empty, since
      every coloring of a $\lambda$-neat $x$ is $d\lambda$-neat,
      and $\badcolorings$ consists of all colorings that fail
      to be $4\lambda$-neat. Thus we can assume that $d\geq 5$.
    
    We will treat the sums probabilistically. Of course, each sum
    has a probabilistic interpretation in the first place, but we give
    a simpler one:
    For each edge $e$ in a cycle in $x$, interpret $\omega_e=1$ to mean
    that $e$ is colored blue, and $\omega_e=0$ to mean that it is colored
    red. We will put a product measure on $(\omega_e)$, assigning each edge
    blue with probability $3/d$ and red with probability $1-3/d$. (There
    is nothing special about these probabilities, and others would work
    as well.) Let $R_1$ be the total number of $RB$s in the color patterns
    of all cycles in the coloring
    given by $(\omega_e)$. Let $R_2$ be the total number
    of $RR$s in these patterns. Let $m$ be the total
    number of edges in all cycles in $x$.
    We define $X$ to be zero if any cycle is colored all blue
    by $(\omega_e)$; otherwise,
    \begin{align*}
      X\Def \left(\frac{6(d-2)}{d-3}\right)^{R_1}3^{R_2}\left(\frac{d}{3}\right)
      ^m.
    \end{align*}
    Since the total number of red edges is $R_1+R_2$ and the total
    number of blue edges is $m-R_1-R_2$,
    this makes
    \begin{align*}
      \E X &= \sum_{(\omega_e)\in \{0,1\}^m}\left(\frac{3}{d}\right)^{m-R_1-R_2}
        \left(\frac{d-3}{d}\right)^{R_1+R_2}\\
        &\qquad\qquad\qquad\quad\times
        \left(\frac{6(d-2)}{d-3}\right)^{R_1}3^{R_2}\left(\frac{d}{3}\right)
      ^m\one{\text{no blue cycles}}\\
        &=          \sum_{\ell} \prod_{\alpha\colon (y_\ell)_\alpha=1}
        \big(2(d-2)\big)^{r_1}(d-3)^{r_2}.
    \end{align*}
    
    The number of Hamiltonian vertices in the random coloring is $m-R_2$.
    So, the claim takes on the form
    \begin{align}
      \E[X\one{m-R_2>4\lambda(d-1)^{r-1}}] \leq e^{-\lambda(d-1)^{r-2}}\E X.
      \label{eq:neggoal}
    \end{align}
    The random variable $X$ is a decreasing function of $(\omega_e)$:
    indeed, changing $\omega_e$ from zero to one causes one of the following
    changes to $R_1$ and $R_2$, depending on the coloring of the neighbors
    of $e$:
    \begin{enumerate}[i)]
      \item \label{item:RRR2RBR}$RRR \to RBR$:
        $R_2$ decreases by two, $R_1$ increases by one;
      \item$RRB\to RBB$:
        $R_2$ decreases by one;
      \item$BRB\to BBB$:
        $R_1$ decreases by one.
    \end{enumerate}
    $X$ decreases in all of these cases (we use the assumption
    that $d\geq 5$ in case \ref{item:RRR2RBR}).
    Changing $\omega_e$ from zero
    to one might also cause a cycle to be colored all blue, in which
    case $X$ decreases to zero. The random variable
    $\one{m-R_2>4\lambda(d-1)^{r-1}}$ is an increasing function of
    $(\omega_e)$. By the FKG inequality,
    \begin{align}
      \E[X\one{m-R_2>4\lambda(d-1)^{r-1}}] &\leq (\E X) \P[m- R_2 >
        4\lambda(d-1)^{r-1}].\label{eq:FKGresult}
    \end{align}
    If $m-R_2 > 4\lambda(d-1)^{r-1}$, then $m-R_2-R_1 > 2\lambda(d-1)^{r-1}$;
    this is because $m-R_2$ is the number of Hamiltonian vertices,
    and $m-R_2-R_1$ is the number of blue edges in the coloring, and there
    are at most twice as many Hamiltonian vertices as blue edges.
    Thus
    \begin{align*}
      \P[m- R_2 >
        4\lambda(d-1)^{r-1}] &\leq
          \P[m- R_2 -R_1 >
        2\lambda(d-1)^{r-1}].
    \end{align*}
    The number of blue edges, $m-R_2-R_1$, is distributed as
    $\Binom(m,3/d)$.
    Since $x$ is $\lambda$-neat, the inequality $m\leq \lambda(d-1)^r/2$
    holds. Thus
    \begin{align*}
      \P[m - &R_2 - R_1 > 2\lambda(d-1)^{r-1}]\\
        &= \P\left[m - R_2 - R_1 - \E[m - R_2 - R_1] > 
          2\lambda(d-1)^{r-1} - \frac{3m}{d}\right]\\
        &\leq \P\left[m - R_2 - R_1 - \E[m - R_2 - R_1] > 
          \frac{\lambda}{2}(d-1)^{r-1}\right].
    \end{align*}
    By Hoeffding's inequality,
    \begin{align*}
      \P[m - R_2 - R_1 > 2\lambda(d-1)^{r-1}]
        &\leq \exp\left(-\frac{\lambda^2(d-1)^{2r-2}}{2m}\right)\\
        &\leq \exp\left(-\lambda(d-1)^{r-2}\right).
    \end{align*}
    With \eqref{eq:FKGresult}, this proves \eqref{eq:neggoal}.
    \end{proof}
    
    Applying the claim to \eqref{eq:twoterms}, we have shown that
    \begin{align*}
      \P[\cycprocess(\phamrv)]
          &\leq  \P[\phamcycpf=x] \Bigg(
         \exp\left(\frac{\Cr{C:me}\lambda^2d^2(d-1)^{2r-1}}{n}
           - \lambda(d-1)^{r-2}\right)\\
         &\qquad\qquad\qquad\quad+
        \exp\left(\frac{\Cr{C:me}\lambda^2(d-1)^{2r-1}}{n}\right)\Bigg)
    \end{align*}
    Using our assumptions that $r\geq 4$ and
    $\Cr{C:mpa}\lambda^2(d-1)^{2r-1}<n/2$, and assuming that we choose
    $\Cr{C:mpa}$ sufficiently larger than $\Cr{C:me}$, we have
    \begin{align*}
      \exp\left(\lambda(d-1)^{r-2}\left(\frac{\Cr{C:me}\lambda d^2(d-1)^{r+1}}{n}
          - 1\right)\right)
            &\leq  \exp\left(-\frac{\lambda}{2}(d-1)^{r-2}\right)\\
            &= n^{-(d-1)^{r-2}/2}= O(n^{-1}),
    \end{align*}
    and
    \begin{align*}
      \exp\left(\frac{\Cr{C:me}\lambda^2(d-1)^{2r-1}}{n}\right)
        = 1 + O\left(\frac{\lambda^2(d-1)^{2r-1}}{n}\right).
    \end{align*}
    This and an identically derived lower bound
    complete the proof.
  \end{proof}

\begin{proof}[Proof of Corollary~\ref{cor:tvbound}]
Suppose that $\mu$ and $\nu$ are probability measures on a discrete space $\Omega,$ and suppose that for some set $A \subset \Omega,$ 
\[
\sum_{x \in A} | \mu({x}) - \nu({x}) | \leq \epsilon_1
\]
and $\mu(A^c) \leq \epsilon_2$.
Then it is easily checked that $\dtv(\mu,\nu) \leq \epsilon_1 + \epsilon_2$.  By virtue of Propositions~\ref{prop:multpoiapprox} and~\ref{prop:strictlneat}, this is precisely the situation in which we are here.  We note that we may assume that $\Cr{C:mpa}((\log n)^2)(d-1)^{2r-1} < n/2,$ for by adjusting $\Cr{C:mptv}$ to be sufficiently large, we may make the bound trivial.
\end{proof}

\section{Variance calculation}
\label{sec:variance}

\newcommand{\QHVC}{R}
\newcommand{\QHVH}{B}
\newcommand{\QHVc}{r}
\newcommand{\QHVh}{b}
\newcommand{\QHVch}{rb}
\newcommand{\QHVhc}{br}
\newcommand{\CX}{\tilde{X}}
\newcommand{\CV}{\tilde{V}}
\newcommand{\CZ}{\tilde{Z}}
\newcommand{\phony}{\ensuremath{\Lambda}}

An alternative formulation of the second moment calculation that we need to make comes from the mixed model $\phammd$.
The quantity we need to estimate is $\Exp \hams^2(P)$ with $P$ drawn from the pairing model $\pmd$.
From \eqref{eq:radonnikodymstatement}, which states that
$\hams$ is the rescaled Radon-Nikodym derivative 
of $\phammd$ with respect to $\pmd$, it follows that 
\[
\frac{
\Exp_{\pmd}[\hams^2]
}
{
\left( \Exp_{\pmd}[\hams] \right)^2
}
=
\frac{
\Exp_{\phammd}[\hams]
}
{
\Exp_{\pmd} [\hams].
}
\]
By the symmetry of both models, every fixed Hamiltonian cycle is equally probable in either $\pmd$ or in $\phammd,$ and therefore, dividing through by the number of Hamiltonian cycles, it is equivalent to consider the ratio of probabilities of a fixed Hamiltonian cycle appearing.  Thus, we fix distinct prevertices $v_1,v_2,v_3,\ldots, v_{2n}$ where $v_{2i}, v_{2i-1}$ come from vertex bin $i,$ and we consider the graph \phony~on $\{ v_i \}_{i=1}^{2n}$ with edges $\edgesetof(\phony) = \{ v_{2i}v_{2i+1} \}_{i=1}^n,$ where we let $v_{2n+1} = v_1$.  Let $E$ denote the event that a pairing contains \phony~as a subgraph.  By the note above,
\[
\frac{
\Exp_{\pmd}[\hams^2]
}
{
\left( \Exp_{\pmd}[ \hams ]\right)^2
}
=
\frac{
\Pr_{\phammd}[E]
}
{
\Pr_{\pmd}[E]
}
\]

In $\phammd,$ the orderings of prevertices within each bin are uniformly and independently randomized,
so the source of any given prevertex $v_i$ might be the configuration graph or
the superimposed Hamiltonian cycle. As in previous sections, call a prevertex \emph{red} if its 
source is the configuration graph and \emph{blue} if it is the Hamiltonian cycle. 
For a given coloring $\ell\colon \edgesetof(\phony ) \to \{\QHVC,\QHVH\}$,
let $E_\ell$ be the event that for all $i$, both prevertices $v_{2i}$ and $v_{2i+1}$ 
have the color $\ell(\{v_{2i},v_{2i+1}\})$. For $E$ to even have a chance of happening,
we need $E_\ell$ to occur for some coloring $\ell$. Indeed, if $v_{2i}$ and $v_{2i+1}$ have different
colors, then they cannot possibly form an edge in the graph sampled from $\phammd$.

Now, we consider the probability of $E$ conditional on $E_\ell$. 
Define $V_n = \sum_{i=1}^n \one{ \ell\{ v_{2k}v_{2k+1}\} = \QHVC}$.  
It is straightforward to compute
\begin{align}
\label{eq:nvcondprob}
\Pr_{\phammd}
\left[
E
~\middle\vert~
E_\ell 
\right] &=\frac{1}{\dff{(d-2)n}{V_n}}
\frac{1}{\dff{2n-1}{n-V_n}}.
\end{align}

Meanwhile, it is possible to compute the exact probability of $E_\ell$ for any fixed
coloring~$\ell$.  
Let $\QHVh_2(\ell)$ be the number of vertex bins $i$ for which $\ell(\{v_{2i}v_{2i+1}\})=\QHVH$ and  $\ell(\{v_{2i-2}v_{2i-1}\})=\QHVH$. Likewise, let $\QHVc_2(\ell)$ be the number of vertex bins for which $\ell(\{v_{2i}v_{2i+1}\})=\QHVC$ and $\ell(\{v_{2i-2}v_{2i-1}\})=\QHVC$.  From the independence of the 
ordering of prevertices in each vertex bin, 
\begin{equation}
\label{eq:nvcolorprob}
\Pr_{\phammd}
\left[
E_\ell
\right]=
\left(\frac{2}{\ff{d}{2}}\right)^{\QHVh_2(\ell)}
\left(\frac{\ff{d-2}{2}}{\ff{d}{2}}\right)^{\QHVc_2(\ell)}
\left(\frac{2(d-2)}{\ff{d}{2}}\right)^{n-\QHVh_2(\ell)-\QHVc_2(\ell)}.
\end{equation}

Combining \eqref{eq:nvcondprob} and \eqref{eq:nvcolorprob}, we have our first formula for $\Pr_{\phammd}(E),$ given by
\[
\Pr_{\phammd}[E]
=
\sum_{
\ell}
\left(\frac{2}{\ff{d}{2}}\right)^{\QHVh_2(\ell)}
\left(\frac{\ff{d-2}{2}}{\ff{d}{2}}\right)^{\QHVc_2(\ell)}
\left(\frac{2(d-2)}{\ff{d}{2}}\right)^{n-\QHVh_2(\ell)-\QHVc_2(\ell)}
\Pr_{\phammd}
\left[
E
~\middle\vert~
 E_\ell 
\right],
\]
where the sum runs over all possible edge colorings $\ell$.  However, this formula is ill-suited to asymptotic analysis, because exponentially rare $\ell$ contribute the majority of the sum.  To rectify this, we define a new distribution on random colorings and use it to develop an alternate expression for $\Pr_{\phammd}(E)$.  We will need to rescale 
\(\Pr_{\phammd}
\left[
E
~\middle\vert~
E_\ell 
\right]
\)
by $2^{n-V_n(\ell)}(d-2)^{V_n(\ell)}$.  As $V_n(\ell)$ counts the total number of edges of the cycle colored $\QHVC,$ we can express $V_n(\ell) = \QHVh_2(\ell) + (n-\QHVh_2(\ell)-\QHVc_2(\ell))/2$.  Thus we define
\begin{equation}
\label{eq:nvscaledpf}
\nvspf
\Def
\sum_{
\ell
}
\left(\frac{1}{\ff{d}{2}}\right)^{\QHVh_2(\ell)}
\left(\frac{(d-3)}{\ff{d}{2}}\right)^{\QHVc_2(\ell)}
\left(\frac{\sqrt{2(d-2)}}{\ff{d}{2}}\right)^{n-\QHVh_2(\ell)-\QHVc_2(\ell)},
\end{equation}
again summing over all edge colorings.

 Viewing $\{\QHVC,\QHVH\}^n$ as edge-colorings of an $n$-cycle, we define
a probability measure on this space by 
\[
\nvec( \{f\} )
\Def \frac{
\left(\sqrt{2(d-2)}\right)^{n-\QHVh_2(f)-\QHVc_2(f)}
\left(d-3\right)^{\QHVc_2(f)}
}
{
\nvpf
},
\]
where $\nvpf$ is a normalizing constant, $\QHVc_i(f)$ is the number of vertices with $i$ incident edges labeled $\QHVC$ and $\QHVh_i(f)$ is the number of vertices with $i$ incident edges labeled $\QHVH$.

Letting $V_n$ denote the number of $\QHVC$-labeled edges in a coloring sampled from $\{\QHVC,\QHVH\}^n$, this allows us to write
\begin{equation}
\frac{\Pr_{\phammd}\left[E\right]}
{
\nvspf
}
=
\Exp_\phi
\frac{(d-2)^{V_n}}{\dff{(d-2)n}{V_n}}
\frac{2^{n-V_n}}{\dff{2n-1}{n-V_n}},
\end{equation}
where we recall that $\dff{2n-1}{n} \Def \dff{2n-1}{n-1}$.
 As $n-b_2-r_2=2r_1$, in the notation defining \eqref{eq:cdpdef}, we calculate $\nvspf$ as 
 \[
\nvspf = 
\cdp_n\left(
\tfrac{1}{\ff{d}{2}},
\tfrac{{2(d-2)}}{\ff{d}{2}^2},
\tfrac{d-3}{\ff{d}{2}}\right)=
\left(\frac{1}{d}\right)^n + \left(\frac{-1}{\ff{d}{2}}\right)^n,
\]
by \eqref{eq:cdp}. Recalling that $\Pr_{\pmd}(E)$ is precisely 
$1/\dff{nd}{n},$ we can finally write
\begin{equation}
\label{eq:nvlastexact}
\frac{\Pr_{\phammd}\left[E\right]}
{
\Pr_{\pmd}\left[E\right]
}=
\left(
\Exp_\phi
\frac{\dff{nd}{n}}{d^n}
\frac{(d-2)^{V_n}}{\dff{(d-2)n}{V_n}}
\frac{2^{n-V_n}}{\dff{2n-1}{n-V_n}}
\right)\left(1 + \left(\frac{-1}{(d-1)}\right)^{n}\right).
\end{equation}

To estimate this expectation, we begin by approximating the integrand by something less complicated.  This amounts to just applying Stirling's approximation to each of the terms.

\begin{lemma}
Define $Z_n \Def \sqrt{\frac{d^3}{2n(d-2)^2}}\left(V_n - n\tfrac{d-2}{d}\right)$.  Then,
\label{lem:nvstirling}
\begin{equation}
\frac{\dff{nd}{n}}{d^n}
\frac{(d-2)^{V_n}}{\dff{(d-2)n}{V_n}}
\frac{2^{n-V_n}}{\dff{2n-1}{n-V_n}}
\leq
\exp\left( \frac{Z_n^2}{d} + \nvserror\right)
\sqrt{
\frac{2n+1}{2V_n+1}
},
\end{equation}
where $\nvserror$ satisfies a bound of the form
\[
\nvserror \leq \Cl{c:stirling}\left(\frac{1}{n}+ \frac{1}{V_n + 1} + \frac{1}{1+n-V_n}\right),
\]
for some absolute constant $\Cr{c:stirling}$.
\end{lemma}
\begin{proof}

By standard Stirling's approximation, which we write in the form
\[
n! = \sqrt{2\pi n} \left(\frac{n}{e}\right)^n e^{\lambda_n}
\]
for some $\frac{1}{12n+1} \leq \lambda_n \leq \frac{1}{12n},$ we may approximate the $\dff{a}{b}$ terms.  Specifically, we have
\begin{equation}
\label{eq:dff}
\left|\log \frac{\dff{a}{b}}{\left(\frac{a}{e}\right)^b\left(\frac{a-2b}{a}\right)^{-\tfrac{a}{2}+b}}\right| \leq \frac{\Cl{c:stirling1}}{a} + \frac{\Cr{c:stirling1}}{(2+a-2b)}.
\end{equation}
for some absolute constant $\Cr{c:stirling1}$ and any $a \geq 2$ and $b \geq 0$ so that $a-2b\geq 0$.  We take the convention here that $0^{0} = 1$.

By applying this approximation, we get that 

\begin{equation}
\label{eq:nvfirststirling}
\begin{aligned}
\frac{\dff{nd}{n}}{d^n}
\frac{(d-2)^{V_n}}{\dff{(d-2)n}{V_n}}
\frac{2^{n-V_n}}{\dff{2n}{n-V_n}}
&=\\
& \hspace{-1.0in} \left(1+\frac{d}{d-2}\frac{\tilde V_n}{n}\right)^{\tfrac{(d-2)n}{d} + \tilde V_n}
\left(1-\frac{2d}{(d-2)^2}\frac{\tilde V_n}{n}\right)^{\tfrac{(d-2)^2n}{2d}-\tilde V_n}
e^{\nvserror},
\end{aligned}
\end{equation}
where $\tilde V_n = V_n-\tfrac{d-2}{d}n$ and $\nvserror$ is defined implicitly to make this an equality.  Note that $\nvserror$ satisfies the desired error bound by \eqref{eq:dff}.  Also note that the left hand side is not exactly the expression we need to approximate, as we have replaced $\dff{2n-1}{n-V_n}$ by $\dff{2n}{n-V_n}$.

By applying the bound $1+a \leq e^a$ to \eqref{eq:nvfirststirling} we get that
\begin{equation*}
\frac{\dff{nd}{n}}{d^n}
\frac{(d-2)^{V_n}}{\dff{(d-2)n}{V_n}}
\frac{2^{n-V_n}}{\dff{2n-1}{n-V_n}}
\leq
\frac{\dff{2n}{n-V_n}}{\dff{2n-1}{n-V_n}}
\exp\left(\frac{Z_n^2}{d} + \nvserror\right),
\end{equation*}
and hence it suffices to show that there is some other error bound $\nvserror'$ of the right form so that
\[
\frac{\dff{2n}{n-V_n}}{\dff{2n-1}{n-V_n}}
\leq \sqrt{\frac{2n+1}{2V_n+1}} e^{\nvserror'}.
\]
For $V_n \geq 1,$ we have that
\[
\frac{\dff{2n}{n-V_n}}{\dff{2n-1}{n-V_n}}
=
\left( \frac{2n}{2n-1} \right)^n
\left( \frac{2V_n-1}{2V_n} \right)^{V_n}
\sqrt{\frac{2n-1}{2V_n-1}} e^{\nvserror'}.
\]
We bound the exponentials using $1+a \leq e^a$.  As for the radical,
there is an absolute constant $\Cl{c:stirling2}$ so that
for $V_n \geq 1$ we have
\[
\sqrt{\frac{2n-1}{2V_n-1}}
\leq
\sqrt{\frac{2n+1}{2V_n+1}}\left(1+\Cr{c:stirling2}(1/n + 1/V_n)\right).
\]
Hence we get
\[
\frac{\dff{2n}{n-V_n}}{\dff{2n-1}{n-V_n}}
\leq \sqrt{\frac{2n+1}{2V_n+1}} e^{\nvserror''},
\]
for some other error term $\nvserror''$ of the right form.  In the case that $V_n = 0,$ we have
\[
\frac{\dff{2n}{n-V_n}}{\dff{2n-1}{n-V_n}} = { 2n \choose n} \frac{2n}{4^n},
\]
which by direct approximation, is $2\sqrt{n/\pi}(1+O(1/n))$.  This is bounded by $\sqrt{2n}(1+O(1/n)),$ and by adjusting constants, we get that
\[
\frac{\dff{2n}{n-V_n}}{\dff{2n-1}{n-V_n}} 
\leq
\sqrt{\frac{2n+1}{2V_n+1}}\exp\left( \Cr{c:stirling}/n + \Cr{c:stirling}/(V_n+1) \right)
\]
for some absolute constant $\Cr{c:stirling}$.
\end{proof}
We will see that $Z_n$ is approximately standard normal and $\nvserror$ is negligible; making these replacements would give the desired $d/(d-2)$ in this expression.  Executing the actual approximation is delicate, however, due to the Gaussian integral term; especially, we require a very strong Gaussian tail bound on $Z_n$.  This rules out many available techniques for showing Gaussian concentration, as they do not provide sufficiently sharp constants.  We prove a tail bound by a detailed analysis of the Laplace transform that is good enough for these purposes.
\begin{lemma}
\label{lem:nvsharptail}
For all $t \geq 0$
\[
\Pr_{ \phi} \left[
\left| V_n - \tfrac{d-2}{d} n \right| \geq t
\right]
\leq 8 \exp\left( -\frac{t^2}{2nc_d} \right),
\]
where 
\[
c_d = \begin{cases}
\frac{\sqrt{3}}{18} & \text{if } d = 3, \\
\frac{2\sqrt{3}}{27} & \text{if } d = 4, \\
\frac{\sqrt{2(d-3)}}{8\sqrt{d-2}} & \text{if } d \geq 5.
\end{cases}
\]
\end{lemma}
\begin{remark}
\label{remark:insufficiency} This tail bound is the principal reason that the error term in Proposition~\ref{prop:Variance} has suboptimal $d$-dependence.  The term $Z_n$ is chosen to have limiting variance $1,$ and thus $c_d$ would ideally behave more like $1/d$.  
\end{remark}
\begin{proof}
The key to computing the Laplace transform is the polynomials $\cdp_k(a,b,c)$ from Section~\ref{sec:supporting}.  These polynomials give an explicit expression for the Laplace transform of $V_n$.  
Observe that $V_n$ can be written as 
$r_2 + r_1,$ 
so that
\[
  \Exp_\phi \exp(s V_n) = \frac{\cdp_n(ae^{s},be^{s},c)}{Z_\phi},
\]
with $a = d-3,$ $b={2(d-2)}$ and $c=1$.  Note that $Z_\phi = \cdp_n(a,b,c)$.  In both cases, these polynomials can be written as $\cdproot_+^n + \cdproot_{-}^n$ for certain expressions in $a,b,c$.  Explicitly, we recall
\eqref{eq:cdp}:
\[
\cdp_n(a,b,c) = \cdproot_+^n + \cdproot_-^n \text{ for } n \geq 1,
\text{ where }
\cdproot_{\pm} = \frac{c+a \pm \sqrt{(c-a)^2 + 4b}}{2},
\]
for all $a,b,c$.  For all non-negative values of $a,b,c,$ we have that $\cdproot_+ \geq |\cdproot_{-}|$.  

For these specific values; $a = d-3,$ $b={2(d-2)}$ and $c=1,$ we have that 
\[
Z_\phi = \cdp_n(a,b,c) = 
\left({d-1}\right)^n 
+ 
\left({-1}\right)^n 
\geq
\tfrac{1}{2}\left({d-1}\right)^n,
\]
for $n \geq 1$ and $d \geq 3$.
Combining these observations, we have
\[
\Exp_\phi \exp(s V_n) \leq 4 \left(\lplus(s)\right)^n,
\]
where 
\(
\lplus(s) = \frac{ae^{s}+c + \sqrt{(ae^{s}-c)^2 + 4be^{s}}}{2},
\)
with  $a = \tfrac{d-3}{d-1},$ $b=\tfrac{{2(d-2)}}{(d-1)^2}$ and $c=\tfrac{1}{d-1}$. 
We note that $\lplus(0) = 1$ and that 
\[
\lim_{s\to 0}
\frac{
\log\lplus(s)}
{s} = \frac{d-2}{d}.
\] 
We proceed to estimating the derivative $(\log \lplus(s) / s)',$ which we would like to bound by a constant.  First, we note that we can pull out a factor of $e^{s/2}$ and keep the derivative the same, i.e.
\[
\left(\frac{\log \lplus(s)}{s}\right)'
=\left(\frac{\log e^{-s/2}\lplus(s)}{s}\right)'.
\]
So, we define $q(s) = e^{-s/2}\lplus(s)$. 
By doing integration by parts, we have that
\[
\log q(s) = s (\log q(s))' - s\int_0^s t \left( \log q(t) \right)'' dt,
\]
and thus
\begin{equation}
\label{eq:logderivative}
\left(\frac{\log q(s)}{s}\right)'=\frac{1}{s^2}\int_0^s t
\left( \log q(t) \right)'' dt.
\end{equation}
Therefore, it suffices to bound $\left( \log q(t) \right)''$ above.  Let $f(s) = ae^{s/2} - ce^{-s/2},$ in terms of which we can write
\[
q(s) = \frac{2f'(s) + \sqrt{ f(s)^2 + 4b}}{2}.
\] 
Noting that $f''(s) = \tfrac{1}{4}f(s),$ it is easily verified that
\[
(\log q(s))'' = \frac{2b f'(s)}{(f(s)^2 + 4b)^{3/2}}.
\]

This expression is $C^1$ for all $s \in \R$.  Further, it tends to $0$ at both $\pm \infty,$  and so its maximum occurs at one of its critical points.  By squaring and differentiating, it follows that its extrema occur at the roots of
\[
\tfrac12 f'(s) f(s) (f(s)^2 + 4b)^3 - 6(f'(s))^2 (f(s)^2 + 4b)^2 f'(s)f(s) = 0.
\]
When $d \geq 4,$ there are three possible roots, given by the root of $f(s) = 0$ and possibly $2$ roots of $f(s)^2 + 4b - 12 (f'(s))^2 = 0$.  These values are given by
\[
e^s = \frac{1}{d-3}, \text{ or }
e^{s} = \frac{4 \pm \sqrt{16 - 4(d-3)^2}}{2(d-3)^2}.
\]
Thus for $d \geq 6,$ the maximum is given by the first root.  For $d=5,$ the roots all coincide at $e^s = \tfrac12$.  For $d=4,$ there are $3$ distinct roots to check.  

In the $d=3$ case, it is no longer possible for $f(s) = 0,$ but the equation $f(s)^2 + 4b - 12 (f'(s))^2 = 0$ still has a root; however, the expression is no longer quadratic.  We summarize the results of this calculus in the following table~\ref{tab:calc}.
\begin{center}
\begin{tabular}{r  l l l}
\label{tab:calc}
$d$ & Critical points & Maximizers & Maximum \\
\hline
$3$ & $e^s = \tfrac14$ & $\tfrac14$ & $\tfrac{\sqrt{3}}{18}$ \\
$4$ & $e^s = \tfrac{1}{d-3}, 2 \pm \sqrt{3}$ & $2 + \sqrt{3}$ & $\tfrac{2\sqrt{3}}{27}$ \\
$5$ & $e^s = \tfrac12$ & $\tfrac12$ & $\tfrac{\sqrt{2(d-3)}}{8\sqrt{d-2}}$\\
$d \geq 6$ & $e^s = \tfrac{1}{d-3}$ & $\tfrac{1}{d-3}$ & $\tfrac{\sqrt{2(d-3)}}{8\sqrt{d-2}}$ \\
\end{tabular}
\end{center}
All together this shows that, recalling equation~\eqref{eq:logderivative}, that
\[
\left(\frac{\log \lplus(s)}{s}\right)' \leq \frac{c_d}{2}.
\]
Integrating, we have that 
\[
\lplus(s) \leq \exp \left( \tfrac{d-2}{d} s + c_d \tfrac{s^2}{2} \right),
\]
for all $s$ and hence, by Markov's inequality,
\[
\Pr\left[
\left| V_n - \tfrac{d-2}{d}n\right| 
\geq t
\right]
\leq 8\exp \left( n c_d \tfrac{s^2}{2} - st\right),
\]
for all $s$. Optimizing in $s$ produces the stated bound.
\end{proof}

As a consequence, we are able to estimate some small moments of $\exp(Z_n^2/d)$ uniformly in $d$ and $n$.
\begin{lemma}
\label{lem:nvgaussianmoments}
For every $\alpha$ with $1 \leq \alpha < 2/\sqrt{3},$ there is a constant $M_\alpha$ so that
\[
\Exp_\phi \exp(\alpha \tfrac{Z_n^2}{d}) \leq M_\alpha.
\]
Further, for every $\alpha$ with $1 \leq \alpha < 8/\sqrt{2},$ there is a $d_0(\alpha)$ and a constant $M_\alpha$ so that 
\[
\Exp_\phi \exp(\alpha \tfrac{Z_n^2}{d}) \leq M_\alpha
\]
for all $d \geq d_0$. 
\end{lemma}
\begin{proof}
By scaling the tail bound in Lemma \ref{lem:nvsharptail}, we have
\[
\Pr \left[
\tfrac{\alpha}{d}|Z_n|^2 \geq t^2
\right]
\leq 8\exp \left(-t^2 \beta \right),
\]
where
\[
\beta \Def
\frac{(d-2)^2}{\alpha d^2 c_d}.
\]
with equality when $d=3$.
Since, we now take
\[
\Exp_\phi \exp(\alpha \tfrac{Z_n^2}{d})
= \int_{0}^\infty e^t \Pr \left[
\tfrac{\alpha}{d}|Z_n|^2 \geq t
\right]
=\frac{8}{\beta - 1},
\]
provided $\beta > 1$. Thus it suffices to bound $\beta$ from below to control this constant.  On the one hand, we have that for all $d \geq 3,$
\(
\beta
 \geq \frac{ 2 }{\sqrt{3}\alpha},
\)
with equality when $d=3$.  On the other hand, we have that $\beta \to \tfrac{8}{\sqrt{2}\alpha}$ as $d \to \infty,$ from which follows the second statement.
\end{proof}

Using this tail bound, we are able to estimate the contributions of the subexponential terms to the expectation, so that we have
\begin{lemma}
For $d \leq n^{1/2}/\log n,$ 
\label{lem:nvidealization}
\begin{equation*}
\Exp_\phi
\frac{\dff{nd}{n}}{d^n}
\frac{(d-2)^{V_n}}{\dff{(d-2)n}{V_n}}
\frac{2^{n-V_n}}{\dff{2n-1}{n-V_n}}
\leq 
\sqrt{\frac{d}{d-2}}
\Exp_\phi
\exp\left(\frac{Z_n^2}{d}\right) + O\left(\frac{1}{\sqrt{n}}\right).
\end{equation*} 
\end{lemma}
\begin{proof}
Our starting point is Lemma~\ref{lem:nvstirling}; we must bound
\[
\Exp_\phi 
\exp\left( \frac{Z_n^2}{d} + \nvserror\right)
\sqrt{
\frac{2n+1}{2V_n+1}
}.
\]
We first approximate this sum by replacing the $V_n$ in the square root by $\tfrac{d-2}{d}n$.  Thus, we seek to estimate 
\[
E_1 \Def
\Exp_\phi 
\exp\left( \frac{Z_n^2}{d} + \nvserror\right)
\left(
\sqrt{
\frac{2n+1}{2V_n+1}
}
-
\sqrt{
\frac{2n+1}{2\tfrac{d-2}{d}n+1}
}
\right),
\]
from above. Let $f(x) = \sqrt{
\frac{2n+1}{2x+1}
}$. Note that there is a constant $\Cl{c:nvs}$ so that $\exp(\nvserror) \leq \Cr{c:nvs}$ with probability $1$ for all $n$ and $d$,
so that
\[
E_1 \leq \Cr{c:nvs}
\Exp_\phi 
\exp\left(\frac{Z_n^2}{d}\right)
\left|f(V_n) - f(\tfrac{d-2}{d}n)\right|. 
\]
Fix some $\alpha$ with $1 < \alpha < 2/\sqrt{3}$ and apply H\"older's inequality with exponent $\alpha$ and conjugate $\alpha^*$ to get
\begin{equation}
\label{eq:nvfirstholder}
E_1 \leq 
\Cr{c:nvs}
\left(
\Exp_\phi
\exp\left(\alpha \frac{Z_n^2}{d}\right)
\right)^{1/\alpha}
\left(
\Exp_\phi
\left|f(V_n) - f(\tfrac{d-2}{d}n)\right|^{\alpha^*} 
\right)^{1/\alpha^*}.
\end{equation}
We note that $\tfrac{d-2}{d} \geq \tfrac13$ for all $d\geq 3$, and therefore by Lemma~\ref{lem:nvsharptail}, there is some absolute constant $\Cl{c:nvt1}$ so that
\begin{equation}
\label{eq:nvdumbbound}
\Pr \left[ V_n \leq \tfrac 16 n\right] \leq \tfrac{1}{\Cr{c:nvt1}} e^{-\Cr{c:nvt1} n}.
\end{equation}
The largest possible value of $f(V_n)$ is $\sqrt{2n + 1},$ and thus we have 
\[
\Exp_\phi
\left|f(V_n) - f(\tfrac{d-2}{d}n)\right|^{\alpha^*}
= 
O(n^{\alpha^*/2}e^{-\Cr{c:nvt1}n})
+
\Exp_\phi
\left|f(V_n) - f(\tfrac{d-2}{d}n)\right|^{\alpha^*} \one{V_n \geq \tfrac{1}{6} n}.
\] 
To estimate this other bit, we note that $|f'(x)/n|$ is bounded uniformly in $n$ for $x \geq \tfrac{1}{6}n,$ and hence
\[
\Exp_\phi
\left|f(V_n) - f(\tfrac{d-2}{d}n)\right|^{\alpha^*} \one{V_n \geq \tfrac{1}{6} n}
= O\biggl(\Exp_\phi
\biggl|\frac{V_n - \tfrac{d-2}{d}n}{n}\biggr|^{\alpha^*} 
\biggr)
= O( n^{-\alpha^*/2}).
\]
Combining everything, we have that
\(
E_1 = O(n^{-1/2}),
\)
and we have therefore shown that
\[
\Exp_\phi 
\exp\left( \frac{Z_n^2}{d} + \nvserror\right)
\sqrt{
\frac{2n+1}{2V_n+1}
} \leq
\Exp_\phi 
\exp\left( \frac{Z_n^2}{d} + \nvserror\right)
\sqrt{
\frac{2n+1}{2\tfrac{d-2}{d}n+1}} + O(n^{-1/2}).
\]
Note that this radical is always less than $\sqrt{
\frac{d}{d-2}},$ and so we turn to removing $\nvserror;$ we must now bound
\[
E_2 \Def
\sqrt{
\frac{d}{d-2}
}
\Exp_\phi
\exp\left( \frac{Z_n^2}{d}\right)
\left|
e^{\nvserror}
-1 \right|.
\]
Again, we apply H\"older's inequality with the same $\alpha$ and in the same way as~\eqref{eq:nvfirstholder} to get
\[
E_2 = O\biggl(
\left(
\Exp_\phi
\left| e^{\nvserror} - 1 \right|^{\alpha^*}
\right)^{1/\alpha^*}
\biggr).
\]
Since $\nvserror$ is bounded uniformly in $n$, we have by Taylor approximation that 
\[
\left| e^{\nvserror} - 1 \right|
\leq \Cl{nv:nvsexp}\left( \frac{1}{V_n+1} + \frac{1}{1+n-V_n}\right)
\]
for some absolute constant $\Cr{nv:nvsexp},$ which follows from Lemma~\ref{lem:nvstirling}.  Thus, by the triangle inequality, it suffices to bound
\[
E_2 = O\biggl(
\biggl(
\Exp_\phi
\left| \frac{1}{V_n+1} \right|^{\alpha^*}
\biggr)^{1/\alpha^*}
+\biggl(
\Exp_\phi
\left| \frac{1}{1+n-V_n} \right|^{\alpha^*}
\biggr)^{1/\alpha^*}
\biggr).
\]

For the first one, we have that by~\eqref{eq:nvdumbbound}, the $\frac{1}{V_n+1}$ term is $O(n^{-1})$ except for with probability $O(e^{-\Cr{c:nvt1}n})$.  For the second, we note that $\frac{1}{1+n-V_n}$ has more complicated $d$ dependence, as when $d$ is large, the mean of $V_n$ is nearly $n$.  That said, there is some absolute constant $\Cl{c:nvt2}>0$ so that
\begin{equation*}
\Pr \left[ V_n \geq \tfrac{d-1}{d} n\right] \leq \tfrac{1}{\Cr{c:nvt2}} e^{-\Cr{c:nvt2}n/d^2},
\end{equation*}
which follows immediately by Lemma~\ref{lem:nvsharptail}.  Thus 
the $\frac{1}{1+n-V_n}$ term is $O(d/n)$ except for with probability $O(e^{-\Cr{c:nvt2}n/d^2})$.  By assumption that
$d \leq \sqrt{n}/\log n$, this probability decays faster than any power of $n,$ and certainly it is $O(1/\sqrt{n})$.   

Combining these bounds, we get that
\[
\left(\Exp_\phi \left| e^{\nvserror} - 1 \right|^{\alpha^*}\right)^{1/\alpha^{*}}
=O\left(\frac{1}{\sqrt{n}}\right).
\]
Thus $E_1 = O(1/\sqrt{n})$ and $E_2 =O(1/\sqrt{n}),$ which completes the proof.
\end{proof}

\subsection{Markov chain approximation}

We will replace $\nvec$ with a distribution that is amenable to easier analysis.  Underlying this replacement is the idea that a random coloring $f$ drawn from $\nvec$ produces a vector $(f(1), f(2), \ldots, f(n))$ that has nearly the same distribution as $(\nvmc{1},\nvmc{2}, \ldots, \nvmc{n})$ where $\nvmc{k}$ is 
the Markov chain on $\{\QHVC,\QHVH\}$ with transition probabilities 
\(
\Pr \left[ \nvmc{k+1} = y
~\mid~
\nvmc{k} = x\right] = p(x,y),
\)
and where $p(x,y)$ is given by
\begin{align*}
\label{eq:nvmc_transitions}
p(\QHVC,\QHVC) &= \frac{d-3}{d-1}, 
&
p(\QHVC,\QHVH) &= \frac{2}{d-1}, \\
p(\QHVH, \QHVC) &= \frac{d-2}{d-1},
&
p(\QHVH,\QHVH) &= \frac{1}{d-1}.
\end{align*}
This chain is easily checked to have stationary distribution that puts mass $(d-2)/d$ on $\{\QHVC\}$ and mass $2/d$ on $\{\QHVH\},$ and we will consider this chain started from stationarity.  

This is a rapidly mixing chain, and its mixing properties can be controlled by the contraction coefficient $\ccoeff,$ which for this chain is
\begin{equation}
\label{eq:ccoeff}
\ccoeff \Def \dtv
\left(
\lawof(X_2 \mid X_1 = \QHVC),\,
\lawof(X_2 \mid X_1 = \QHVH)
\right) = \frac{1}{d-1},
\end{equation}
with $\lawof$ denoting the law of a random variable.
This gives a simple bound for the rate at which two 
Markov chains with the same transitions as $X_k$ can be coupled.  Suppose that $\{X_k^{1}\}$ and 
$\{X_{k}^2\}$ are two chains with the same transitions as $X_k$ but with different starting states.  There is a coupling of these two chains so that $\tau = \inf \left\{ k \geq 0~\middle\vert~ X_k^1 = X_k^2\right\}$ has $\Pr\left[ \tau > k \right] \leq \ccoeff^k$.

The chain implicitly defines a distribution on edge colorings by simply defining a coloring 
$f \in \{\QHVC,\QHVH\}^n$ by $f(k) \Def \nvmc{k}$.  We will refer to the law on colorings defined in this way as $\nvmcpi$.  The precise relationship between $\nvec$ and $\nvmcpi$ is that $\nvec$ is absolutely continuous with respect to $\nvmcpi$,
and the unscaled Radon-Nikodym derivative of $\nvec$ with respect to $\nvmcpi$ is
\begin{equation}
\label{eq:nvmcrnd_def}
\nvmcrnd(f)
\Def\begin{cases}
2(d-3) & \text{if }f(1)=f(n)=\QHVC \\
2(d-2) & \text{if }(f(n),f(1)) = (\QHVC,\QHVH) \\
2(d-2) & \text{if }(f(n),f(1)) = (\QHVH,\QHVC) \\
(d-2) & \text{if }f(1)=f(n)=\QHVH.
\end{cases}
\end{equation}
\begin{lemma}
With $\pi$ as defined above,
\label{lem:nvmarkovrnd}
\[
\frac{d\nvec}{d\nvmcpi}(f)
=
\frac{\nvmcrnd(f)}
{
\Exp_{\nvmcpi} \nvmcrnd(f) 
}.
\]
\end{lemma}
\begin{proof}
For an edge coloring $f$ of the cycle, recall that $\QHVc_i$ denotes the number of vertices with $i$ neighboring edges colored $\QHVC$ and $\QHVh_i$ denotes the number of vertices with $i$ neighboring edges colored $\QHVH$.  Likewise, let $\QHVch$ denote the number of vertices $j$ with $f(j-1) = \QHVC$ and $f(j) = \QHVH,$ with the addition done mod $n$.  Similarily, let $\QHVhc$ denote the number of vertices $j$ with $f(j-1) = \QHVH$ and $f(j) = \QHVC,$ with the addition done mod $n$.  Then, it follows that $\QHVc_1 = \QHVh_1 = \QHVhc + \QHVch,$ but also, because this a cycle, it must be that $\QHVhc = \QHVch$.  

For any coloring $f,$
\[
\nvmcpi(\{f\}) \nvmcrnd(f)~\propto~(d-3)^{\QHVc_2} 2^{\QHVch} (d-2)^{\QHVhc} 1^{\QHVh_2}.
\]
On the other hand, 
\[
\nvec(\{f\})~\propto~(d-3)^{\QHVc_2} \sqrt{2(d-2)}^{\QHVc_1} 1^{\QHVh_2}.
\]
Using that $\QHVc_1 = 2 \QHVhc = 2 \QHVch,$ it now follows that $\nvmcrnd$ is the unscaled Radon-Nikodym derivative.

\end{proof}
Using the Radon-Nikodym derivatives, we can transfer moment estimates from $\nvec$ to $\nvmcpi$ with little effort. 
\begin{lemma}
\label{lem:nvgaussianmomentspi}
For every $\alpha$ with $1 \leq \alpha < 8/\sqrt{2},$ there is a constant $M_\alpha$ and a constant $d_0(\alpha)$ so that
\[
\Exp_\nvmcpi \exp(\alpha \tfrac{Z_n^2}{d}) \leq M_\alpha d
\]
for all $n$ and all $d \geq d_0$.  If $\alpha < 2/\sqrt{3},$ we can take $d_0 = 3$.
Furthermore, we have that for all $t \geq 0,$ 
\[
\Pr \left[
\left|Z_n\right| \geq t\sqrt{d}
\right] \leq dM_\alpha \exp(-\alpha t^2).
\]
\end{lemma}
\begin{proof}
The second conclusion of the lemma follows immediately from the first by Markov's inequality.  As for the first, in the case that $d\geq 4,$ this is simply a consequence of Lemma~\ref{lem:nvgaussianmoments} and the fact that $\nvmcrnd$ is bounded below by $1$ $\pi$-almost surely; note
\[
M_\alpha \geq
\Exp_\nvec \exp\left(\alpha \frac{Z_n^2}{d} \right)
=\Exp_\nvmcpi \frac{\nvmcrnd(f)}{\Exp_{\nvmcpi} \nvmcrnd(f)}
\exp\left( \alpha \frac{Z_n^2}{d} \right)
\geq \frac{1}{\Exp_\nvmcpi \nvmcrnd(f)}\Exp_\nvmcpi {\exp\left(\alpha \frac{Z_n^2}{d} \right)},
\]
so that rearranging,
\[
\Exp_\nvmcpi {\exp\left(\alpha \frac{Z_n^2}{d} \right)} 
\leq \Exp_\nvmcpi \nvmcrnd(f) M_\alpha,
\]
and the result now follows from having $\Exp_\nvmcpi \nvmcrnd(f) = O(d)$.

However, when $d=3,$ we require an additional argument, because $\rho$ can be $0$.  Consider the involution $\iota$ on colorings that swaps the color $f(n)$ between $\QHVC$ and $\QHVH$.  Let $\iota^*(Z_n(f))$ denote the random variable $Z_n(\iota(f)),$ so that we have
\[
\exp\left(\alpha \frac{Z_n^2}{3} \right)
\one{f(n)=\QHVC}
=
\exp\left(\alpha \frac{\iota^*(Z_n)^2 +2\iota^*(Z_n)q + q^2}{3} \right)
\one{\iota(f)(n)=\QHVH},
\]
where $q = \sqrt{\frac{3^3}{2n(3-2)^2}}$.  For any coloring $f$ with $f(n)=\QHVC,$ meanwhile, it must be that $f(n-1)=\QHVH$ else $\nvmcpi(\{f\})=0$.  Thus, for any coloring with $f(n-1)=\QHVH$ and $f(n)=\QHVC,$ we have that
\[
\nvmcpi(\{f\}) = {\nvmcpi(\{\iota(f)\})}.
\]
Thus, we can change the integration and get that
\[
\Exp_{\nvmcpi}
\exp\left(\alpha \frac{Z_n^2}{3} \right)
\one{f(n)=\QHVC}
=
\Exp_{\nvmcpi}
\exp\left(\alpha \frac{Z_n^2 +2Z_nq + q^2}{3} \right)
\one{f(n)=\QHVH}.
\]
This right hand side can now be bounded in terms of $\nvec$ by
\[
\Exp_{\nvmcpi}
\exp\left(\alpha \frac{Z_n^2 +2Z_nq + q^2}{3} \right)
\one{f(n)=\QHVH}
\leq
\Cl{c:nvpi3}
\Exp_{\nvec}
\exp\left(\alpha \frac{Z_n^2 +2Z_nq}{3} \right),
\]
for some absolute constant $\Cr{c:nvpi3}$ as when $f(n)=\QHVH,$ $\nvmcrnd(f)$ is bounded below.  
Pick $\alpha'$ so that $\frac{2}{\sqrt{3}} > \alpha' > \alpha$.  By H\"older's inequality, we have that
\[
\Exp_{\nvec}
\exp\left(\alpha \frac{Z_n^2 +2Z_nq}{3} \right)
\leq 
\biggl(
\Exp_{\nvec}
\exp\left(\alpha' \frac{Z_n^2}{3} \right)
\biggr)^{\tfrac{\alpha}{\alpha'}}
\biggl(
\Exp_{\nvec}
\exp\left(\frac{\alpha'\alpha}{\alpha'-\alpha} \frac{2Z_nq}{3} \right)
\biggr)^{\tfrac{\alpha'-\alpha}{\alpha'}},
\]
which is bounded uniformly in $n$ by Lemma~\ref{lem:nvgaussianmoments}.
\end{proof}

The Radon-Nikodym derivative can be seen to be approximately independent of $Z_n,$ as $Z_n$ is insensitive to a change of only $2$ coordinates.  For this reason, we can prove
\begin{lemma}
\label{lem:nvmcpi_replacement}
\[
\Exp_\nvec \exp\left( \frac{Z_n^2}{d} \right)
=\Exp_\nvmcpi \frac{\nvmcrnd(f)}{\Exp_{\nvmcpi} \nvmcrnd(f)}
\exp\left( \frac{Z_n^2}{d} \right)
=\Exp_\nvmcpi \exp\left( \frac{Z_n^2}{d}\right)
+O\left(\sqrt{\frac{d}{{n}}}\right).
\]
\end{lemma}
\begin{proof}
  We need to prove that 
\[
\Exp_\nvmcpi \left[\exp\left( \frac{Z_n^2}{d}\right)~\middle\vert~X_1=x,X_n=y
\right] - 
\Exp_\nvmcpi \exp\left( \frac{Z_n^2}{d}\right)
\] 
is small, regardless of $x$ and $y$.
  To simplify notation, replace $\QHVH$ and $\QHVC$ with $0$ and $1$.
  Let $\{Y_1,\ldots,Y_n\}$ be a Markov chain with the same transition probabilities
  as $\{X_1,\ldots,X_n\}$, but started at $Y_1=x$.
  We take the two chains to have the optimal Markovian coupling:
  conditional on $X_i$ and $Y_i$, the random variables $X_{i+1}$ and $Y_{i+1}$ 
  are coupled by the optimal total variation coupling. Let $\tau$ be the first
  time that the two chains coincide (after which they stay together), or $\infty$
  if they never do. For a chain on $\{0,1\}$ with transition probability from
  $0$ to $0$ smaller than from $1$ to $0$, this coupling has the property that
  \begin{align*}
    (X_i,Y_i)&=(1-x,\,x) \qquad\text{for all odd $i<\tau$,}\\
    (X_i,Y_i)&=(x,\,1-x) \qquad\text{for all even $i<\tau$.}
  \end{align*}
  Thus the sums of two chains differ by at most one,
  indicating that this statistic is quite insensitive to the starting point of the chain.
  We will write $\E[\cdot]$ with no subscript to indicate expectations with
  respect to this coupling, reserving the notation $\E_\pi[\cdot]$ for expectations
  that depend only on the first chain.
  
  Let $q =\sqrt{\frac{d^3}{2n(d-2)^2}}$.
Let $V_n'=\sum_{i=1}^nY_i$,
and let $Z_n'=q(V_n' - dn/(d-2))$.
We rewrite the conditional expectation as
\begin{align}
  \Exp \left[\exp\left( \frac{Z_n^2}{d}\right)~\middle\vert~X_1=x,X_n=y
\right] &= 
   \frac{\E\left[\exp\left(\frac{Z_n'^2}{d}\right)\1\{Y_n=y\} \right]}
   {\P[Y_n=y]}\nonumber\\
   &\leq
      \frac{\E\left[\exp\left(\frac{Z_n'^2}{d}\right)\1\{Y_n=y\} \right]}
   {\mu(y)-(d-1)^{1-n}}
\label{eq:term1}
\end{align}
By the properties of the coupling mentioned above, $Z'_n\leq Z_n+q$.
So long as $\tau\neq\infty$, we have $X_n=Y_n$, and so
\begin{align}\label{eq:tausplit}
  \begin{split}
  \E\left[\exp\left(\frac{Z_n'^2}{d}\right)\1\{Y_n=y\} \right]
    &\leq \E_\nvmcpi\biggl[\exp\left(\frac{(Z_n+q)^2}{d}\right)\1\{X_n=y\}\biggr] \\
      &\qquad\qquad\qquad+\E\biggl[\exp\left(\frac{Z_n'^2}{d}\right)\1\{\tau=\infty\}\biggr].
  \end{split}
\end{align}
If $\tau=\infty$ and $n$ is even, then $V_n'=n/2$, and 
\begin{align*}
  \frac{Z_n'^2}{d} &= \frac{(d-4)^2}{8(d-2)^2}n \leq \frac{n}{8}.
\end{align*}
If $\tau=\infty$ and $n$ is odd, then $V_n'=(n\pm 1)/2$, and some algebra shows that
$Z_n'^2/d\leq (n+11)/8 $.
Thus
\begin{align}\label{eq:tt1}
  \E\left[\exp\left(\frac{Z_n'^2}{d}\right)\1\{\tau=\infty\}\right]
    &\leq\mu(1-x)(d-1)^{1-n}\exp\left(\frac{n+11}{8}\right),
\end{align}
which is easily $O(1/n)$.

To deal with the first term of \eqref{eq:tausplit}, we use the reversibility
of the Markov chain to rewrite it as
\begin{align*}
  \E_\nvmcpi\left[\exp\left(\frac{(Z_n+q)^2}{d}\right)\1\{X_n=y\}\right]
  &= \mu(y)\E_\nvmcpi\left[\exp\left(\frac{(Z_n+q)^2}{d}\right)\:\middle\vert\:
   X_1=y\right].
\end{align*}
As before, there exists a coupling of $Z_n$ with a random variable $Z_n''$ such that
$Z''_n$ is distributed as $Z_n$ conditioned on $X_1=y$, and $Z''_n\leq Z_n+q$.
Thus
\begin{align}\label{eq:tt3}
  \E_\nvmcpi\left[\exp\left(\frac{(Z_n+q)^2}{d}\right)\1\{X_n=y\}\right]
  &\leq \mu(y)\E_\nvmcpi\exp\left(\frac{(Z_n+2q)^2}{d}\right).
\end{align}
Fix some $1<\alpha<2/\sqrt{3}$ and apply H\"older's inequality to get
\begin{align*}
  \E_\nvmcpi\exp\left(\frac{(Z_n+2q)^2}{d}\right) -
  \E_\nvmcpi\exp\left(\frac{Z_n^2}{d}\right)
  &= \E_\nvmcpi\left[\exp\left(\frac{Z_n^2}{d}\right)\left(
    \exp\left(\frac{4qZ_n+4q^2}{d}\right)-1\right)
    \right]\\
  &\leq (d M_\alpha)^{1/\alpha}
    \biggl(
\Exp_\nvmcpi
\left|
\exp\left(\frac{4qZ_n + 4q^2}{d}\right)-1 
\right|^{\alpha^*}
\biggr)^{1/\alpha^*}.
\end{align*}
By applying the bounds that $|e^x - 1| \leq |x|e^{|x|}$ and that $q^2/d = O(1/n)$,
there is some absolute constant $\Cl{c:371}$ so that
\[
\Exp_\nvmcpi
\left|
\exp\left(\frac{4qZ_n + 4q^2}{d}\right)-1 
\right|^{\alpha^*}
\leq 
\Cr{c:371}
\Exp_\nvmcpi
\left|
\frac{Z_nq}{d}
\exp\left(\frac{\Cr{c:371}Z_nq}{d}\right) 
\right|^{\alpha^*}.
\]
Note that $\frac{q}{d} = O(1/\sqrt{dn})$ and hence by once again applying H\"older's inequality and using the second part of Lemma~\ref{lem:nvgaussianmomentspi}, we conclude that 
\[
\Exp_\nvmcpi
\left|
\frac{Z_nq}{d}
\exp\left(\frac{\Cr{c:371}Z_nq}{d}\right) 
\right|^{\alpha^*}
=O(d/n^{\alpha^*}).
\]
This shows that
\begin{align*}
    \E_\nvmcpi\exp\left(\frac{(Z_n+2q)^2}{d}\right) -
  \E_\nvmcpi\exp\left(\frac{Z_n^2}{d}\right)
  &\leq O\left(\frac{d}{n}\right).
\end{align*}
Applying \eqref{eq:tt1} and \eqref{eq:tt3} to \eqref{eq:tausplit}
and substituting into \eqref{eq:term1},
\begin{align*}
  \Exp_\nvmcpi \left[\exp\left( \frac{Z_n^2}{d}\right)~\middle\vert~X_1=x,X_n=y
    \right] &\leq
      \frac{\mu(y)}{\mu(y)-(d-1)^{1-n}}\left(
         \E_\nvmcpi\exp\left(\frac{Z_n^2}{d}\right) + O\left(\frac{d}{n}\right)\right)\\
    &\leq \E_\nvmcpi\exp\left(\frac{Z_n^2}{d}\right) + \Cl{c:372}\frac{d}{n}.
\end{align*}
 for some absolute constant $\Cr{c:372}$, uniformly in $x$ and $y$.
The conclusion of the lemma now follows by integrating
\begin{align*}
\Exp_\nvmcpi \frac{\nvmcrnd(f)}{\Exp_{\nvmcpi} \nvmcrnd(f)}
\exp\left( \frac{Z_n^2}{d}\right)
&=
\Exp_\nvmcpi
\biggl[
\Exp_\nvmcpi\left[
 \frac{\nvmcrnd(f)}{\Exp_{\nvmcpi} \nvmcrnd(f)}
\Exp_\nvmcpi \exp\left( \frac{Z_n^2}{d}\right)
~\middle\vert~
X_1,X_2
\right]
\biggr]
 \\
&\leq
\Exp_\nvmcpi
 \frac{\nvmcrnd(f)}{\Exp_{\nvmcpi} \nvmcrnd(f)}
\left[
\Exp_\nvmcpi \exp\left( \frac{Z_n^2}{d}\right)
+\Cr{c:372}\frac{d}{n}
\right] \\
&=
\Exp_\nvmcpi \exp\left( \frac{Z_n^2}{d}\right)
+O(d/n).\qedhere
\end{align*}
\end{proof}

\subsection{Comparison with a standard normal by size-bias coupling}

The remainder of the work is to compare these expectations in $Z_n$ with that which we would get for a standard normal.  For this task, we develop a modification of Stein's method for normal approximation that allows us to directly compare these expectations.  The basic outline of this approach follows the general method of size-bias couplings for normal approximation.%
\footnote{
See Ross's excellent survey \cite{Ross} for an overview; we will frequently reference general results surrounding Stein methodology from this source.}

We define $h(w) = \exp(w^2/d),$ and let $\Phi(h)=\sqrt{\frac{d}{d-2}}$ denote the expectation of $h$ applied to a standard normal.  We let $\nvstein$ be the solution to the differential equation
\begin{equation}
\label{eq:nvsteineq}
\nvstein'(w) - w \nvstein(w) = h(w) - \Phi(h)
\end{equation}
that is given by the formulae
\begin{align*}
\nvstein 
&= \exp(w^2/2)\int_{w}^\infty \exp(-t^2/2) \left( \Phi(h) - h(t) \right)\,dt \\
&= -\exp(w^2/2)\int_{-\infty}^w \exp(-t^2/2) \left( \Phi(h) - h(t) \right)\,dt. 
\end{align*}

In the usual Stein's method setup, the function $h$ is bounded, from which it follows that $\nvstein'$ and $\nvstein''$ are also bounded.  This is not the case here, but it is easily verified that the growth rates of $\nvstein$ and its derivatives are commensurate to the growth rate of $h$.
\begin{lemma}
\label{lem:nvsteingrowth}
There is an absolute constant $\Cl{c:resolvent}$ so that
\begin{align*}
\left| \nvstein (w) \right| &\leq \Cr{c:resolvent} \Phi(h)(1+|w|)^{-1}h(w) \\ 
\left| \nvstein' (w) \right| &\leq \Cr{c:resolvent} \Phi(h)h(w) \\ 
\left| \nvstein'' (w) \right| &\leq \Cr{c:resolvent} \Phi(h)(1+|w|)h(w). 
\end{align*}
\end{lemma}
\begin{proof}
We begin by noting that for all $w \neq 0,$
\begin{equation*}
\exp(\alpha w^2)\int_{|w|}^\infty
\exp(-\alpha x^2)~dx \leq 
\exp(\alpha w^2)\int_{|w|}^\infty
\frac{x}{|w|}\exp(-\alpha x^2)~dx \leq 
\frac{1}{2\alpha |w|}.
\end{equation*}
From this, we observe that for all $w,$ 
\[
\exp(\alpha w^2)\int_{|w|}^\infty
\exp(-\alpha x^2)~dx \leq \sqrt{\pi/4\alpha},
\]
as its derivative in $w$ is negative for $w>0$.  It follows that there is an absolute constant $\Cl{c:res1}$ so that
\[
\left|
\nvstein(w)
\right| 
\leq \Cr{c:res1}\Phi(h)h(w)(1 \wedge \tfrac{1}{|w|}).
\]
From the differential equation~\eqref{eq:nvsteineq}, we have that
\[
\left|\nvstein'(w)\right|
\leq |w\nvstein |
+ h(w) + \Phi(h) \leq \Cl{c:res2}\Phi(h)h(w)
\] 
for some larger absolute constant $\Cr{c:res2}$.  By differentiating the Stein equation \eqref{eq:nvsteineq}, we may also bound
\[
\left|\nvstein''(w)\right|
\leq |\nvstein | + |w \nvstein' |
+ |h'(w)| \leq \Cl{c:res3}\Phi(h)(1+|w|) h(w)
\]
for some other absolute constant $\Cr{c:res3}$.
\end{proof}

Using the basic Stein's method setup for size-bias coupling (see equation (3.25) of~\cite{Ross}), we have the following lemma, 
which refers to a size-bias coupling $(V_n^s,V_n)$ and an associated probability
space constructed in the appendix.
\begin{lemma}
\label{lem:sizebiaslemma}
Let $\mu = n(d-2)/d = \Exp_\nvmcpi V_n$ and $\sigma^2 = \frac{2n(d-2)^2}{d^3}$.
For any $\sigma$-algebra~$\nvsteinsigma$ containing $\sigma(Z_n)$, 
\begin{align*}
\left| \Exp_\nvmcpi h(Z_n) - \Phi(h) \right|
&\leq
\Exp
\left|
f_h'(Z_n)
\left(
1 - \frac{\mu}{\sigma^2}
\Exp \left[
V_n^s - V_n ~\middle\vert~ \nvsteinsigma
\right]
\right)
\right| \\
&\quad\quad+ \frac{\mu}{2\sigma^3}
\Exp
\left|
f_h''\left(Z_n^*\right) (V_n^s - V_n)^2
\right|,
\end{align*}
where $Z_n^*$ is in the interval with endpoints $Z_n$ and $(V_n^s-\mu)/\sigma$.
\end{lemma}
Using this lemma, we finally estimate the difference in the expectations.
\begin{lemma}
\label{lem:normal_comparison}
For $d \leq \sqrt{n},$ and for any $\alpha < 8/\sqrt{2},$ we have that 
\[
\Exp_\nvmcpi 
\left[ \exp\left(\frac{Z_n^2}{d}\right)
\right] = \sqrt{ \frac{d}{d-2}} +
O\left( \frac{d^{\frac{3}{2}(1+1/\alpha)}}{\sqrt{n}}\right).
\]
\end{lemma}
\begin{proof}
We consider the size-bias coupling considered in the appendix, and the only probability space under consideration in this proof will be the one constructed there.
We start from Lemma~\ref{lem:sizebiaslemma}, by virtue of which we need only bound
\[
E_1 \Def\Exp
\left|
f_h'(Z_n)
\left(
1 - \frac{\mu}{\sigma^2}
\Exp \left[
V_n^s - V_n \,\middle\vert\, \nvsteinsigma
\right]
\right)
\right|,
\quad\text{where }\nvsteinsigma\Def\sigma(X_1,\ldots,X_n),
\]
and 
\[
E_2 \Def
\frac{\mu}{2\sigma^3}
\Exp
\left|
f_h''\left(Z_n^*\right) (V_n^s - V_n)^2
\right|.
\]
For $E_1,$ it will turn out that the expectation of $V_n^s - V_n$ is not exactly $\sigma^2/\mu$.  On the other hand, by Proposition~\ref{prop:sbc_expectation}, we have an exact expression for $\Exp [V_n^s - V_n]$.  We note that, in the notation of that
section, $\lambda = -1/(d-1)$ and that $p = (d-2)/d$.  It follows that
\[
\Exp [V_n^s - V_n]
= \frac{2(d-2)}{(d-1)^2}
+ \frac{4(d-1)}{d^3n}\left( 1 - \left(\frac{-1}{d-1}\right)^n\right)
= \frac{\sigma^2}{\mu} + O\left(\frac{1}{d^2n}\right).
\]
From Lemmas~\ref{lem:nvsteingrowth} and \ref{lem:nvgaussianmomentspi} we have that 
\[
\Exp |f_h'(Z_n)| = O\left( \Exp \left[\exp(Z_n^2/d)\right]\right) = O(d).
\]
Applying this to $E_1,$ we conclude that
\[
E_1 = 
\frac{\mu}{\sigma^2}
\Exp
\left|
f_h'(Z_n)
\left(
\Exp \left[
V_n^s - V_n
\right] - 
\Exp \left[
V_n^s - V_n ~\middle\vert~ \nvsteinsigma
\right]
\right)
\right| + O\left(\frac{1}{n}\right).
\]

From Corollary~\ref{cor:sbc_tail2}, we have a uniform Gaussian tail bound on 
\[
\Exp \left[
V_n^s - V_n
\right] - 
\Exp \left[
V_n^s - V_n ~\middle\vert~ \nvsteinsigma
\right].
\]
In the notation of that corollary, we have  $\ccoeff = \tfrac{1}{d-1}$.
If $d=3$, then $\delta=2$ and $\gamma=\tfrac{3}{4}$, and if $d>3$, then
$\delta=1$ and $\gamma=\tfrac{2}{3}$.
Thus the corollary implies that there is an absolute constant $\Cl{c:mtail}>0$ so that for any $t \geq 0,$
\[
\Pr \left(
\left|
\Exp\left[ V_n^s - V_n \right]
-
\Exp \left[V_n^s - V_n ~\middle\vert~ \nvsteinsigma\right]
\right| \geq \frac{t}{\sqrt{n}}
\right) \leq 2 
\exp \left(
-\Cr{c:mtail}t^2
\right).
\]
In particular, this implies that for each fixed $t > 0,$
\[
\Exp
\left|
\Exp\left[ V_n^s - V_n \right]
-
\Exp \left[V_n^s - V_n ~\middle\vert~ \nvsteinsigma\right]
\right|^{t} = O(n^{-t/2}).
\]
By applying H\"older's inequality for $1 < \alpha < 8/\sqrt{2}$ we get
\begin{multline*}
\Exp
\left|
f_h'(Z_n)
\left(
\Exp \left[
V_n^s - V_n
\right] - 
\Exp \left[
V_n^s - V_n ~\middle\vert~ \nvsteinsigma
\right]
\right)
\right| \\
\leq\left( 
\Exp \left| f_h'(Z_n) \right|^{\alpha}
\right)^{1/\alpha}
\left( 
\Exp \left|
\Exp \left[V_n^s - V_n\right] - 
\Exp \left[V_n^s - V_n ~\middle\vert~ \nvsteinsigma\right]
 \right|^{\alpha^*}
\right)^{1/\alpha^*}
=O\left(\frac{d^{1/\alpha}}{\sqrt{n}}\right).
\end{multline*}
Note that for $\alpha > 2/\sqrt{3},$ this only holds for $d \geq d_0$ for some $d_0$, while for $\alpha < 2/\sqrt{3}$, this holds for all $d \geq 3.$  Thus for any $\alpha < 8/\sqrt{2},$ we may choose the implied constants sufficiently large that the inequality holds for all $d \geq 3.$  Hence,
\[
E_1 = O\left(\frac{d^{1+1/\alpha}}{\sqrt{n}}\right).
\]
We now turn to bounding $E_2,$ which we recall is given by
\[
E_2 =
\frac{\mu}{2\sigma^3}
\Exp
\left|
f_h''\left(Z_n^*\right) (V_n^s - V_n)^2
\right|.
\]
From Lemma~\ref{lem:nvsteingrowth}, we have that 
\(
f_h''\left(Z_n^*\right)
=O( (1+|Z_n^*|)h(Z_n^*) ).
\)
This is a monotone upper bound, and hence it suffices to bound
\[
E_3 \Def
\frac{\mu}{2\sigma^3}
\Exp
\left|
(1+|Z_n|)h(Z_n)(V_n^s - V_n)^2
\right| 
\]
and
\[
E_4 \Def
\frac{\mu}{2\sigma^3}
\Exp
\left|
(1+|Z_n^s|)h(Z_n^s)(V_n^s - V_n)^2
\right|,
\]
where we let $Z_n^s = (V_n^s - \mu)/\sigma$.  In either case, we proceed along the usual line of applying H\"older's inequality for $1 < \alpha < 8/\sqrt{2}$.  We show the bound for $E_4,$ as the bound for $E_3$ follows from a nearly identical argument. 
Thus we have
\[
E_4 \leq 
\frac{\mu}{2\sigma^3}
\left(
\Exp\left[
(1+|Z_n^s|)h(Z_n^s)\right]^\alpha
\right)^{1/\alpha}
\left(
\Exp\left|
V_n^s - V_n
\right|^{2\alpha^*}
\right)^{1/\alpha^*}.
\]
By Proposition~\ref{prop:sbc_tail1}, the variable is nonzero with probability at most $O(1/d),$ and conditional on being nonzero, it has a subgeometric tail that is uniform in $n$ and $d$.
Therefore, all the absolute moments of $V_n^s - V_n$ are of order $O(1/d)$.  Meanwhile from the definition of the size-bias distribution, we have that 
\begin{align*}
\Exp\left[
(1+|Z_n^s|)h(Z_n^s)\right]^\alpha
&=
\Exp\frac{V_n}{\mu}\left[
(1+|Z_n|)h(Z_n)\right]^\alpha \\
&\leq
\Exp
(1+\tfrac{|Z_n|}{\mu})\left[(1+|Z_n|)
h(Z_n)\right]^\alpha \\
&=
O\left( \left(1+\tfrac{\sqrt{d}}{n}\right) d^{\alpha +1/2}\right).
\end{align*}
Using that $\mu/\sigma^3 = O(d^{3/2}/\sqrt{n})$ and that $\sqrt{d} \leq n,$ we have that
\[
E_4 = O\left(\frac{d^{5/2+ 1/2\alpha - 1/\alpha^*}}{\sqrt{n}}\right)
= O\left(\frac{d^{\frac{3}{2}(1+1/\alpha)}}{\sqrt{n}}\right).\qedhere
\]
\end{proof}

\subsection{Summary}

These lemmas taken together prove
the needed variance bound.  We will recapitulate them to prove Proposition~\ref{prop:Variance}.
\begin{proof}[Proof of Proposition~\ref{prop:Variance}]
We start with 
\eqref{eq:nvlastexact}.
\begin{equation*}
\frac{\Pr_{\phammd}\left[E\right]}
{
\Pr_{\pmd}\left[E\right]
}=
\left(
\Exp_\phi
\frac{\dff{nd}{n}}{d^n}
\frac{(d-2)^{V_n}}{\dff{(d-2)n}{V_n}}
\frac{2^{n-V_n}}{\dff{2n-1}{n-V_n}}
\right)\left(1 + \left(\frac{-1}{(d-1)}\right)^{n}\right).
\end{equation*}
We apply Stirling's approximation and bound away the subexponential factors
using Lemma~\ref{lem:nvidealization}, so that
\[
\frac{\Pr_{\phammd}\left[E\right]}
{
\Pr_{\pmd}\left[E\right]
}
\leq 
\sqrt{\frac{d}{d-2}}
\Exp_\phi\left[
\exp\left(\frac{Z_n^2}{d}\right)\right] + O\left(\frac{1}{\sqrt{n}}\right).
\]
We then change the measure in the expectation from $\nvec$ to the Markov chain measure $\nvmcpi,$ using Lemma~\ref{lem:nvmcpi_replacement}, to get
\[
\frac{\Pr_{\phammd}\left[E\right]}
{
\Pr_{\pmd}\left[E\right]
}
\leq 
\sqrt{\frac{d}{d-2}}
\Exp_\nvmcpi \left[\exp\left( \frac{Z_n^2}{d}\right)\right]
+O\left(\sqrt{\frac{d}{n}}\right).
\]
Finally, we apply Stein's method machinery to approximate the expectation by one with respect to Gaussian measure to conclude
\[
\frac{\Pr_{\phammd}\left[E\right]}
{
\Pr_{\pmd}\left[E\right]
}
\leq
{\frac{d}{d-2}}
+O\left(\frac{d^{\frac{3}{2}(1+1/\alpha)}}{\sqrt{n}}\right).\qedhere
\]
\end{proof}

\section{Main results}
\label{sec:proof}
\newcommand{\Simple}{\textsc{Simple}}

We will now turn to proving our main results.   
We start with a few definitions. Recall that $f_n=H_n/\E_{\pmd} H_n$ is the Radon-Nikodym
derivative of $\phammd$ with respect to $\pmd$, as explained on p.~\pageref{page:radonnikodym}.
For any 
$x \in \left\{ 0,1 \right\}^{\fullcycspace},$
we define
\begin{align}
  \rndcyc(x) \Def \E_{\pmd} [f_n \mid \cycprocess = x],
  \label{eq:rndcyc}
\end{align}
recalling that $\cycprocess$ is the process of indicators defined in Section~\ref{subsec:quant_est}.
It follows that for any $x \in \left\{ 0,1 \right\}^{\fullcycspace},$
\begin{align*}
  \frac{\P_{\phammd} [\cycprocess=x]}{\P_{\pmd}[\cycprocess=x]} &= \rndcyc(x).
\end{align*}
In other words, $\rndcyc$ can be viewed as the Radon-Nikodym derivative between
the push-forwards of $\phammd$ and $\pmd$ under $\cycprocess.$
We let $\rndpf[r][n]$ refer to the Radon-Nikodym derivative between the Poisson laws of $\phamcycpf$ and $\pcycpf,$  defined in Section~\ref{subsec:quant_est}.  This Radon-Nikodym derivative
has an explicit form that we will need to use.
  Let $x=(x_\alpha,\,\alpha\in\fullcycspace)$, and let
  $c_k=\sum_{\alpha\in\cycspace[k]} x_\alpha$, the number of $k$-cycles
  represented by $x$.
  Recalling \eqref{eq:pcycmean} and \eqref{eq:phamcycmean},
\begin{equation}
  \begin{aligned}
    \rndpf(x) &= \prod_{\alpha\in\fullcycspace}e^{\pcycmean-\phamcycmean}
      \left(\frac{\phamcycmean}{\pcycmean}\right)^{x_\alpha}\\
      &= \prod_{\substack{1\leq k\leq r\\\text{$k$ odd}}}
        \exp\left( \frac{\ff{n}{k}}{kn^k}\right)
          \left(1-\frac{2}{(d-1)^k}\right)^{c_k} \\
      &= \prod_{\substack{1\leq k\leq r\\\text{$k$ odd}}}
        \exp\left(\frac{1}{k} +O\Big(\frac{k}{n}\Big)\right) 
          \left(1-\frac{2}{(d-1)^k}\right)^{c_k}\\
      &= e^{O(r^2/n)} \prod_{\substack{1\leq k\leq r\\\text{$k$ odd}}}
        e^{1/k}
          \left(1-\frac{2}{(d-1)^k}\right)^{c_k}.  
  \end{aligned}
\label{eq:rndpf}
\end{equation}
Note that $\rndpf$ is always positive for $d \geq 4$.  For $d=3,$ we have $\rndpf=0$ precisely when $c_1 > 0.$

For any $r,$ define the limiting second moment expression
\begin{equation}
\label{eq:varDefinition}
\log\limvar \Def \sum_{k=1}^r \frac{\left( (-1)^k - 1 \right)^2}{2k(d-1)^k}.
\end{equation}
Our strategy in this section is to show that the second moment of 
$\rndcyc$ is approximately $\limvar$ (see Lemma~\ref{lem:cvb}).  
This is a truncation of a convergent series for
\begin{equation}
\log\limvar[\infty] \Def \sum_{k=1}^\infty \frac{\left( (-1)^k - 1 \right)^2}{2k(d-1)^k}=\log\left( \frac{d}{d-2}\right).
\end{equation}
Our approximation for $\rndcyc$ in terms of $\limvar$ will lead to an approximation of $f_n$ in terms
of $\limvar[\infty]$, which is the contents of Proposition~\ref{prop:Variance}.

  As we will need to condition on graphs being simple, we define the pairing event
\(
\Simple = \{\text{$P$ simple}\}.
\)
Applying Proposition~\ref{prop:prob0} with $r=2$, we have that
\begin{align}
\label{eq:pmdsimple}
\Pr_{\pmd}[\Simple] &= \exp\left( -\frac{d-1}{2} - \frac{(d-1)^2}{4} + O\left(\frac{(\log n)^2
          d^3}{n}\right)\right), \\
\label{eq:phammdsimple}
\Pr_{\phammd}[\Simple] &=\exp\left( -\frac{d-3}{2} - \frac{(d-1)^2}{4} + O\left(\frac{(\log n)^2
          d^3}{n}\right)\right).
\end{align}
Equation~\eqref{eq:pmdsimple} is also obtained in \cite{MW91} without the $(\log n)^2$ in the error term.

\begin{lemma}
\label{lem:cvb}
We set $\cvberror$ to be 
\[
\cvberror \Def \frac{\left[d + \Cr{C:mpa}(\log n)^2\right](d-1)^{2r-1}}{n}.
\]
There is an absolute constant $\Cl{C:cvb}$ so that for $r\geq 4$ 
\[
  \Exp_{\pmd} [\rndcyc^2\circ \cycprocess]
\geq
\limvar\left( 1 - \Cr{C:cvb}\cvberror\right).
\]
\end{lemma}

\begin{proof}
We may assume that $\cvberror \leq \tfrac12,$ for by adjusting $\Cr{C:cvb}$ to be at least $2$, we may then make the bound trivial.  Further we take $\lambda = \log n,$ so that
for any strictly $\lambda$-neat cycle space point $x$, Proposition~\ref{prop:multpoiapprox} implies that
\[
\rndcyc(x)^2
\Pr_{\pmd} \left[ \cycprocess = x \right] 
= \frac{\Pr_{\phammd} \left[ \cycprocess = x \right] ^2}
  {\Pr_{\pmd} \left[ \cycprocess = x \right] }
\geq \rndpf[r][n](x)^2
\Pr\left[ \pcycpf = x \right]
\tfrac{\left(1-\cvberror\right)^2}
{\left(1+\cvberror\right)}.
\]
By ignoring the non-neat cycle space points, we can immediately bound
\begin{equation}
\label{eq:cvbOne}
\Exp_{\pmd} \left[\rndcyc^2 \circ \cycprocess \right]
\geq
\Exp\left[\rndpf[r][n]^2(\pcycpf)
\one{\pcycpf~\text{strictly $\lambda$-neat}} 
\right]
(1-O(\cvbeps))
.
\end{equation}

To complete the lower bound, we need to estimate the contribution of the non-neat cycles to right hand side, and so we estimate 
\(
\Exp\left[\rndpf[r][n]^2(\pcycpf)
\one{\pcycpf~\text{not strictly $\lambda$-neat}} 
\right]
\) from above.

The key to making this estimate is to realize that $\rndpf[r][n]^2(x)$ is a rescaled Radon-Nikodym for yet another Poisson law. 
Let $\cvbpf=(\cvbvar,\,\alpha\in\fullcycspace)$ be a vector whose coordinates are independent Poisson random variables with $\E \cvbvar=\tfrac{\phamcycmean^2}{\pcycmean}$ for $\alpha\in\fullcycspace$.  It is easily checked that for any cycle space point $x$
\[
\rndpf[r][n]^2(x)
\Pr \left[ \pcycpf = x \right] 
=
\exp\biggl(
\sum_{ \substack{\alpha \in \cycspace \\ k \leq r}}
\frac{(\phamcycmean - \pcycmean)^2}{\pcycmean}
\biggr)
\Pr \left[ \cvbpf = x \right].
\]  

Further, we note that this renormalization constant is precisely 
\begin{align}
\label{eq:cvbTwo}
\Exp
\left[
\rndpf[r][n]^2(\pcycpf)
\right]
&=
\exp\biggl(
\sum_{k=1}^r
\frac{\left((-1)^k-1\right)^2\ff{n}{k}}{2k(d-1)^kn^k}
\biggr) \\
&=\limvar(1-O(r^2/n))=\limvar(1-O(\cvbeps)). \notag
\end{align}
Therefore, we have reduced the problem to estimating 
\(
\Pr \left[
\text{$\cvbpf$ not strictly $\lambda$-neat}
\right]
\).
We first apply Lemma~\ref{lem:PoissonTail} to bound the probability of $\cvbpf$ having too many cycles.  Specifically, we define 
\[
F(x) \Def\sum_{ \substack{\alpha \in \cycspace \\ k \leq r}} 
|\alpha| \cvbvar.
\]
We note that $\|\nabla_\alpha F\| = |\alpha|,$ that $\tfrac{\phamcycmean^2}{\pcycmean} \leq \pcycmean$ and hence that
\begin{align*}
\sum_{\substack{\alpha \in \cycspace \\ k \leq r}} 
|\alpha|^2 \tfrac{\phamcycmean^2}{\pcycmean}
&\leq
\sum_{\substack{\alpha \in \cycspace \\ k \leq r}} 
|\alpha|^2 \pcycmean \\
&=
\sum_{k=1}^r
\frac{k^2\ff{n}{k}(\ff{d}{2})^k}{2k} \frac{1}{(nd)^k}=O(r(d-1)^r).
\end{align*}
By applying Lemma~\ref{lem:PoissonTail}, we conclude that for $t>0,$
\[
\Pr \left[ F(\cvbpf) \geq \expect F(\cvbpf) + t\right] \leq
\exp\left(
-\frac{t}{2r}\log\left(1+\frac{t}{\Cl{c:cvb1}(d-1)^r}\right)
\right),
\]
for some absolute constant $\Cr{c:cvb1}$.  Finally we bound the expectation of $F(\cvbpf)$ with
\[
\Exp F(\cvbpf)
=
\sum_{\substack{\alpha \in \cycspace \\ k \leq r}} 
|\alpha| \tfrac{\phamcycmean^2}{\pcycmean}
\leq
\sum_{k=1}^r
\frac{k\ff{n}{k}\ff{d}{2}^k}{2k} \frac{1}{(nd)^k}\leq (d-1)^r.
\]
Thus we conclude that 
\begin{equation}
\label{eq:cvbTail}
\Pr \left[  F(\cvbpf) \geq \lambda (d-1)^r\right]
\leq\frac{1}{\Cl{c:cvb2}}
\exp\left(
-\Cr{c:cvb2}\left(\frac{(d-1)^r}{r} \lambda \log\lambda\right)
\right)
\end{equation}
for some absolute constant $\Cr{c:cvb2}$.
It remains to estimate the probability under $\cvbpf$ that two cycles share a vertex.  There are $\ff{n}{k-1}\ff{d}{2}^{k}/2k$ many $\alpha\in\cycspace$ that use any given vertex.  Thus,
taking a union bound over all $1\leq k \leq r$ and all $1 \leq l \leq r,$
\begin{align}
\notag
\Pr\left[
\exists~\alpha\in\cycspace[k],\beta\in\cycspace[l]\text{ sharing a vertex so that }\cvbvar[\alpha] = \cvbvar[\beta] = 1
\right]
\hspace{-3in} & \hspace{3in} \\
\label{eq:cvbOverlap}
&
=\sum_{1 \leq k,l \leq r}
n
\frac{\ff{n}{k-1}\ff{d}{2}^{k}}{2k}
\frac{\ff{n}{l-1}\ff{d}{2}^{l}}{2l}
\frac{1}{(nd)^k}
\frac{1}{(nd)^l}
\leq \frac{(d-1)^{2r}}{n},
\end{align}
where we have used that $\Pr[ \cvbvar = 1] \leq \Exp \cvbvar \leq \frac{1}{(nd)^{|\alpha|}}$. By combining equations~\eqref{eq:cvbTail} and \eqref{eq:cvbOverlap}, we conclude that
\begin{equation}
\label{eq:cvbNotNeat}
\Pr\left[
\text{$\cvbpf$ not strictly $\lambda$-neat}
\right]
\leq \frac{(d-1)^{2r}}{n}
+ \frac{1}{\Cr{c:cvb2}}\exp\left(
-\Cr{c:cvb2}\frac{(d-1)^r}{r} \lambda \log\lambda
\right).
\end{equation}
By applying this bound, we conclude that
\begin{align}
\notag
\Exp
\left[
\rndpf[r][n]^2(\pcycpf)
\one{\pcycpf~\text{not strictly $\lambda$-neat}} 
\right]
\hspace{-1in} & \hspace{1in}
\\
\notag
&=
\limvar(1+O(\cvbeps))
\Pr \left[
\text{$\cvbpf$ not strictly $\lambda$-neat}
\right] \\ 
\label{eq:cvbNotNeatVar}
&=
\limvar
O(\cvbeps).
\end{align}
We now combine \eqref{eq:cvbOne}, \eqref{eq:cvbTwo}, and \eqref{eq:cvbNotNeatVar} to derive the lower bound
\begin{equation*}
  \Exp_{\pmd} \left[\rndcyc^2\circ \cycprocess \right]
\geq
\limvar(1-O(\cvbeps)),
\end{equation*}
which completes the proof.
\end{proof}

The lower bound on the conditional variance combined with the upper bound on the variance (Proposition~\ref{prop:Variance}) shows that $\rnd$ and $\rndcyc$ are close in $L^2(\pmd)$.
\begin{lemma}
\label{lem:rndl2}
For every $\alpha$ with $1 < \alpha < 8/\sqrt{2}$ there is a constant $M_\alpha$ so that for all $3 \leq d \leq \sqrt{n}/\log n$ and all $r \geq 4,$
\[
  \Exp_{\pmd} \left| \rnd - \rndcyc \circ \cycprocess \right|^2 \leq M_\alpha \left( (d-1)^{-r-1} + \cvberror + d^{\frac{3}{2}(1+1/\alpha)}/\sqrt{n}\right).
\]
\end{lemma}
\begin{proof}
By orthogonality, we have that
\(
\Exp_{\pmd} \left| \rnd - \rndcyc \circ \cycprocess \right|^2 
=\Exp_{\pmd}
\left[ 
\rnd^2
-\rndcyc^2\circ \cycprocess
\right],
\)
and so by Proposition~\ref{prop:Variance} and Lemma~\ref{lem:cvb}, we have that for any $\alpha$ with $1 < \alpha < 8/\sqrt{2}$
\[
  \Exp_{\pmd} \left| \rnd- \rndcyc\circ \cycprocess \right|^2 
= \limvar[\infty] - \limvar + O(\cvberror + d^{\frac{3}{2}(1+1/\alpha)}/\sqrt{n}).
\] 
We note that there is some absolute constant $\Cl{c:analy}$ so that $0 \leq \log \limvar[\infty] - \log \limvar \leq \Cr{c:analy}(d-1)^{-r-1}$.  This in turn implies that $\limvar[\infty] - \limvar = O( (d-1)^{-r-1}),$ and hence we have completed the proof.
\end{proof}

We now develop estimates for $\rndcyc$ by comparing with the limiting Poisson structure.

\begin{lemma}
\label{lem:rndcyctail}
There is a constant $\Cl{C:cvb_small}$ so that for $\delta \geq \log r \geq \log 4,$ $\cvberror \leq \tfrac 12$ and $d \geq 4,$

\begin{align*}
  \Pr_{\pmd}\left[
\left|\log \rndcyc\circ\cycprocess\right| \geq \delta~\wedge~
\text{$\cycprocess$ $(\log n)$-neat}
\right]
\hspace{-1in}&\hspace{1in} \leq
\frac{1}{\Cr{C:cvb_small}}\exp\left(
-\Cr{C:cvb_small}d\delta\log\delta
\right).
\intertext{
In the case that $d=3,$ }
\Pr_{\pmd} \left[
\left|\log \rndcyc \circ \cycprocess\right| \geq \delta
~\wedge~\text{$\cycprocess$ $(\log n)$-neat}
~\wedge~ \Simple
\right] 
\hspace{-1in} 
& \hspace{ 1in} \leq
\frac{\Pr_{\pmd}[\Simple]}{\Cr{C:cvb_small}}\exp\left(
-\Cr{C:cvb_small}d\delta\log\delta
\right).
\end{align*}
\end{lemma}
\begin{remark}
The same bound holds for $\phammd$ as well, and the proof is identical, but we will not need it.
\end{remark}
\begin{proof}
We will show the proof for $d \geq 4$.  The proof for $d=3$ follows by the same argument.
We apply the multiplicative Poisson bound (Proposition~\ref{prop:multpoiapprox})
 to get that
for any strictly $(\log n)$-neat cycle space point $x$, 
\[
\rndpf[r][n](x)
\tfrac{1+\cvberror}
{1-\cvberror}
\geq
\rndcyc(x)
\geq \rndpf[r][n](x)
\tfrac{1-\cvberror}
{1+\cvberror}.
\]
We may therefore bound 
\begin{multline*}
  \Pr_{\pmd} \left[
\left| \log \rndcyc \circ \cycprocess  \right| \geq \delta~\wedge~
\text{$\cycprocess$ strictly $(\log n)$-neat}
\right]
\\
\leq
\Pr_{\pmd} \left[
\log |\rndpf | \geq \delta - \log\tfrac{1+\cvberror}
{1-\cvberror}~\wedge~
\text{$\cycprocess$ strictly $(\log n)$-neat}
\right].
\end{multline*}
As this probability is restricted to strictly $(\log n)$-neat $x,$ the multiplicative Poisson bound implies that 
\begin{multline*}
  \Pr_{\pmd } \left[
\left| \log \rndpf\right| \geq \delta +\log\tfrac{1+\cvberror}
{1-\cvberror}
~\wedge~
\text{$\cycprocess$ strictly $(\log n)$-neat}
\right]
\leq
\Pr \left[
\left| \log \rndpf(\pcycpf)\right|
\geq \delta + \log\tfrac{1+\cvberror}
{1-\cvberror}
\right] 
(1+\cvberror).
\end{multline*}
As we have that $\cvberror \leq \tfrac12,$ it suffices to prove that there is an absolute constant $\Cl{c:cvb_s0} > 0$ so that for all $\delta \geq \log r \wedge 1$
\[
\Pr\left[
\left| \log \rndpf(\pcycpf)\right|
\geq \delta 
\right] 
\leq
\frac{1}{\Cr{c:cvb_s0}}\exp\left(
-\Cr{c:cvb_s0}d\delta\log\delta
\right)
\]
by adjusting constants.

For this purpose we note that the identity that for any cycle space 
point~$x$,
\[
-\log \rndpf(x)
=
\sum_{\substack{\alpha \in \cycspace,\\
1\leq k \leq r}}
\left[-x_\alpha \log \tfrac{\phamcycmean}{\pcycmean}-\pcycmean + \phamcycmean\right].
\]
There is an absolute constant $\Cr{C:cvb_taylor}>0$ so that for all $\alpha,$ $\frac{1}{\Cl{C:cvb_taylor}}(d-1)^{-|\alpha|} \leq -\log \tfrac{\phamcycmean}{\pcycmean} \leq \Cr{C:cvb_taylor} (d-1)^{-|\alpha|}$.  Thus, we define
\[
F(x) \Def
\sum_{\substack{\alpha \in \cycspace,\\ 1\leq k \leq r}}
\frac{x_\alpha}{ (d-1)^{|\alpha|}},
\]
for cycle space point $x$.  Note that the added constant is
\[
0 \leq 
\sum_{\substack{\alpha \in \cycspace,\\
1\leq k \leq r}}
\left[-\pcycmean + \phamcycmean\right]
\leq 
\sum_{k=1}^r
\frac{\ff{n}{k}\ff{d}{2}^k}{n^k\ff{d}{2}^k2k}=O\left(\log r \right).
\]
Also note that the expectation of $F(\pcycpf)$ is
\[
\Exp F(\pcycpf)
=
\sum_{\substack{\alpha \in \cycspace \\ k \leq r}} 
(d-1)^{-k}{\pcycmean}
\leq
\sum_{k=1}^r
\frac{\ff{n}{k}\ff{d}{2}^k}{(d-1)^{k}2k} \frac{1}{(nd)^k}=O(\log r).
\]
Combining these observations, we note that it suffices to prove that there is an absolute constant $\Cl{C:cvb_s}>0$ so that for all $\delta \geq \log r \wedge 1,$
\[
\Pr\left[
\left|
F(\pcycpf)
-
\Exp F(\pcycpf)\right| \geq \delta
\right]
\leq
\frac{1}{\Cr{C:cvb_s}}\exp\left(
-\Cr{C:cvb_s}d\delta\log\delta
\right),
\]
for by again adjusting constants, we may conclude the desired inequality.


This now follows from the modified log-Sobolev inequality bounds.  We note that $\|\nabla_\alpha F\| = (d-1)^{-|\alpha|}$  and hence that
\begin{align*}
\sum_{\substack{\alpha \in \cycspace \\ k \leq r}} 
(d-1)^{-2k}{\pcycmean}
&=
\sum_{k=1}^r
{\ff{n}{k}\ff{d}{2}^k}{(d-1)^{2k}2k} \frac{1}{(nd)^k}=O\left(\frac{1}{d}\right).
\end{align*}
By applying Lemma~\ref{lem:PoissonTail}, we conclude that for $t>0,$
\[
\Pr \left[ |F(\pcycpf) - \expect F(\pcycpf)| \geq t\right] \leq
2\exp\left(
-\frac{dt}{\Cl{C:cvb_tail}}\log\left(1+\frac{t}{\Cr{C:cvb_tail}}\right)
\right),
\]
for some absolute constant $\Cr{C:cvb_tail}>0$. 
\end{proof}

\begin{lemma}
\label{lem:b1}
There is an absolute constant $\Cl{c:b1}$ so that for any $\delta \geq \log r\geq \log 4,$ any $d \geq 4$ and any pairing event $A,$
\[
\Pr_{\pmd}[ A ]
\leq\Cr{c:b1}\left(
\cvberror
+
e^{2\delta}\Var_{\pmd} [\rnd -\rndcyc \circ \cycprocess]
+
\exp\left(
-\Cr{C:cvb_small}d\delta\log\delta
\right)
+e^{\delta}\Pr_{\phammd}[ A ]
\right).
\]
If $A \subseteq \Simple$ the same statement holds for $d=3$.
\end{lemma}
\begin{proof}
We set $\lambda  = \log n,$ and we bound
\begin{align*}
\Pr_{\pmd}
\left[
A
\right]
\leq
&
\Pr_{\pmd}\left[
\rnd \leq e^{-\delta}/2~\wedge~
\text{$\cycprocess$ strictly $\lambda$-neat}
\right] \\
&+\Pr_{\pmd} \left[
\text{$\cycprocess$ not strictly $\lambda$-neat}
\right] \\
&+
\Exp_{\pmd} \left[
2e^{\delta}f_n \one{A}
\right].
\end{align*}
We note that the second line is $O(\cvberror)$ by Proposition~\ref{prop:strictlneat}.  The third line is precisely $2e^{\delta}\Pr_{\phammd} [A]$.  To bound the first line, we write 
\begin{multline*}
\Pr_{\pmd}\left[
\rnd \leq e^{-\delta}/2~\wedge~
\text{$\cycprocess$ strictly $\lambda$-neat}
\right] \\
\leq
\Pr_{\pmd}\left[
\rndcyc \circ \cycprocess \leq e^{-\delta}~\wedge~
\text{$\cycprocess$ strictly $\lambda$-neat}
\right]
+
\Pr_{\pmd}\left[
|
\rnd
-\rndcyc \circ \cycprocess
| \geq e^{-\delta}/2.
\right]
\end{multline*}
The first of these we bound by Lemma~\ref{lem:rndcyctail}, and the second we bound by Chebyshev's inequality, completing the Lemma.
\end{proof}

\begin{lemma}
\label{lem:b2}
There is an absolute constant $\Cl{c:b2}$ so that for any pairing event $A$ and any $r\geq 4,$
\[
\Pr_{\phammd}[ A ]
\leq\Cr{c:b2}\left(
\cvberror
+
\sqrt{
\Var_{\pmd} [\rnd - \rndcyc \circ \cycprocess] 
}
\sqrt{
\Pr_{\pmd}[ A ]
}
+r\Pr_{\pmd}[ A ]
\right).
\]
\end{lemma}
\begin{proof}
We may assume that $\cvberror \leq \tfrac12,$ for by adjusting $\Cr{c:b2} \geq 2,$ we may make the bound trivial.
Let $E$ be the event $E=\{\text{$\cycprocess$ strictly $\lambda$-neat } \},$  
\[
\Pr_{\phammd}[ A ]
\leq
\Pr_{\phammd}[ A \cap E ]
+
\Pr_{\phammd}[ E ].
\]
As we have that $\Pr_{\phammd}[ E ]=O(\cvberror)$ from Proposition~\ref{prop:strictlneat}, it suffices to show the bound for $A \subseteq E$ by passing to $A \cap E$. 

In this case, we have that for any strictly $\lambda$-neat cycle space point $x,$
\[
\rndcyc(x)
\leq 
\tfrac{1+\cvberror}
{1-\cvberror}
\rndpf[r][n](x)
\leq 3
\prod_{\substack{\alpha \in \cycspace,\\
1\leq k \leq r}}
 \exp\left(-\phamcycmean + \pcycmean\right)
=O(r).
\]
By applying Cauchy-Schwarz, we have that 
\begin{align*}
\Exp_{\pmd} \rnd \one{A}
&\leq
\sqrt{
\Var_{\pmd} [\rnd - \rndcyc \circ \cycprocess] 
}
\sqrt{
\Pr_{\pmd}[ A ]
}
+
\Exp_{\pmd} \rndcyc \circ \cycprocess \one{A},
\end{align*}
and we conclude the lemma, noting that $\rndcyc \circ \cycprocess\one{A}$ can be bounded by $Cr\one{A}$ for an absolute constant $C$.
\end{proof}

We now turn to proving the main theorems.

\begin{proof}[Proof of Theorem~\ref{thm:contig}]
Fix a given sequence $D(n) \to \infty$ with $\log D(n)/ \log n \to 0$.  By passing to subsequences, it suffices to show the cases, where $d(n) \leq D(n)$ and where $d(n) \geq D(n)$.  In the latter case, we need only prove the total variation bound.  This, in turn follows from the simple inequality
\begin{align*}
\left| 
\Pr_{\pmd}( A )
-\Pr_{\phammd}( A )
\right|
\leq \sqrt{ \Var_{\pmd} [\rnd] }.
\end{align*}
From Proposition~\ref{prop:Variance}, we therefore have the bound that for any $1 < \alpha < 8/\sqrt{2}$
\[
\dtv( \pmd, \phammd) = O\left(\frac{2}{d-2} + \frac{d^{\frac{3}{2}(1+1/\alpha)}}{\sqrt{n}} \right).
\]
For $d(n) \leq n^{\alpha_0-\epsilon}$ where $\alpha_0 = \frac{8}{3(8+\sqrt{2})}$ we may therefore choose an $\alpha$ so that this tends to 0.

In the former case, we show the contiguity arguments one bound at a time.  We start by assuming that $\Pr_{\pmd}[A_n]\to0$.  We then choose an integer sequence $r(n) \to \infty$ sufficiently slowly that $r(n)\Pr_{\pmd}[A_n]\to0$ and $\cvberror \to 0$.  From Lemma~\ref{lem:rndl2}, we have that 
\[
\Var_{\pmd} [\rnd - \rndcyc \circ \cycprocess] \to 0,
\]
and hence by Lemma~\ref{lem:b2}, $\Pr_{\phammd}[A_n] \to 0$.  

Suppose now that $\Pr_{\phammd}[A_n]\to0$.  We may choose $r(n)$ an integer sequence so that $r(n) \to \infty,$ $r(n)\Pr_{\phammd}[A_n] \to 0$ and $r(n)^2 \cvberror \to 0$.  Apply Lemma~\ref{lem:b1} with $\delta=\log r(n),$ and note that we have 
\[
e^{2\delta(n)}\Var_{\pmd} [\rnd - \rndcyc \circ \cycprocess] \to 0,
\]
so that $\Pr_{\pmd}[A_n]\to0$. 
%
\end{proof}

\begin{proof}[Proof of Theorem~\ref{thm:qcontig}]
The statememts for $\umd$ and $\phamcondmd$ follow immediately from those for $\pmd$ and $\phammd$ together with the observation that for $d=o(\sqrt{\log n}),$ both $\log \Pr_{\pmd}[\Simple] = o(\log n)$ and $\log \Pr_{\phammd}[\Simple] = o(\log n)$ (see~\eqref{eq:pmdsimple} and~\eqref{eq:phammdsimple}).

For the $\pmd$ case, we assume that $\log d(n)/ \log n \to 0,$ and we may choose $r(n) = \left\lfloor \frac{\log n}{3\log(d(n)-1)} \right\rfloor.$
Note that this implies that for all $\epsilon > 0,$ 
\[
\cvberror =\frac{\left[d + \Cr{C:mpa}(\log n)^2\right](d-1)^{2r-1}}{n}=O(n^{-1/3+\epsilon}).
\]
Likewise, by Lemma~\ref{lem:rndl2} we have that for all $\epsilon > 0,$
\[
\Var_{\pmd} [\rnd - \rndcyc \circ \cycprocess] = 
O( (d-1)^{-r-1} + \cvberror + d^3/\sqrt{n})
=
O(n^{-1/3+\epsilon}).
\]
Suppose that $\Pr_{\phammd}[A_n]=O(n^{-\alpha})$.  Then by taking $\delta = \epsilon \log n,$ we conclude by Lemma~\ref{lem:b1} that
\[
\Pr_{\pmd}[A_n] = O( n^{-1/3 + 2\epsilon} + n^{-\alpha+2\epsilon}),
\]
completing the proof.

On the other hand, suppose that $\Pr_{\pmd}[A_n]=O(n^{-\alpha})$.  By Lemma~\ref{lem:b2} we conclude that
\[
\Pr_{\phammd}[A_n] = O( n^{-\alpha/2 - 1/6 + \epsilon/2} + n^{-1/3+\epsilon} + n^{-\alpha+\epsilon}).
\]
Note that $\frac{\alpha}{2}+\frac{1}{6} \geq \alpha \wedge \frac13,$ completing the proof.
\end{proof}

\begin{lemma}
\label{lem:rndlognorm}
Let $r=r(n) \geq 4$ and $d=d(n) \to \infty$ satisfy
\[
(d+(\log n)^2) (d-1)^{2r-1} = o(n),
\]
then
\[
\frac{\log \rndpf( P )}{\sqrt{2/d}} \weakto N(0,1),
\]
where $P \sim \pmd$ and $N(0,1)$ is a standard normal.
\end{lemma}

\begin{proof}[Proof of Lemma~\ref{lem:rndlognorm}]
We once again write 
\[
\cvberror = \frac{\left[d + \Cr{C:mpa}(\log n)^2\right](d-1)^{2r-1}}{n}.
\]
From the total variation bound Corollary~\ref{cor:tvbound}, we can construct a probability space on which $\cycprocess[r](P)$ and $\pcycpf$ are defined and satisfy 
\[
\Pr \left[
\cycprocess[r](P) \neq \pcycpf
\right] = \dtv(\cycprocess(P),\pcycpf) = O(\cvberror).
\]
On this space we have that
\[
\Pr \left[
\log \rndpf ( \cycprocess[r](P)) \neq
\log \rndpf (\pcycpf )
\right] = O(\cvberror).
\]
Writing $Z_k=\sum_{\alpha\in\cycspace[k]} Z_\alpha,$ we have by \eqref{eq:rndpf} that
\[
\log \rndpf (\pcycpf)
= O(r^2/n) + \sum_{\substack{1 \leq k \leq r \\ k \text{ odd}}} 
\left[
\frac{1}{k} + Z_k \log\left(1- \frac{2}{(d-1)^k}\right)
\right].
\]
Note that $Z_k \sim \Poisson\left( \frac{\ff{n}{k}(d-1)^k}{2kn^k}\right),$ and hence
\begin{multline*}
\Exp
\left|
\frac{1}{k} + Z_k \log\left(1- \frac{2}{(d-1)^k}\right)
\right| \\
\leq 
\left|\frac{1}{k} + \frac{\ff{n}{k}(d-1)^k}{2kn^k}\log\left(1- \frac{2}{(d-1)^k}\right)  \right|
+\left|\log\left(1- \frac{2}{(d-1)^k}\right)\right|\sqrt{\Var Z_k}.
\end{multline*}
As the $\log$ terms can be approximated by $\frac{-2}{(d-1)^k}$ uniformly in $k$ and $d,$ it follows that
\[
\Exp
\left|\log \rndpf (\pcycpf)
- 1 + \frac{2}{d-1}Z_1\right| = O(r^3/n + 1/d).
\]
Now from the classical central limit theorem, we have that
\[
\frac{\tfrac{d-1}{2} - Z_1}{\sqrt{d/2}} \weakto N(0,1),
\]
which completes the proof.
\end{proof}

\begin{proof}[Proof of Theorem~\ref{thm:lognorm}]

Take
\(
r(n) = \left\lfloor \frac{\log n}{3\log(d(n)-1)} \right\rfloor.\)
By Lemma~\ref{lem:rndlognorm} it suffices to show that 
\[
\frac{\log \rnd(P) -
\log \rndpf ( \cycprocess[r](P))
}{\sqrt{2/d}} \probto 0.
\]
On the one hand, we have that
\[
\frac{\log \rndcyc(P) -
\log \rndpf ( \cycprocess[r](P))
}{\sqrt{2/d}} \probto 0,
\]
as for a cycle space point $x,$ that is $(\log n)$--neat, we have by Proposition~\ref{prop:multpoiapprox}
\[
\rndpf[r][n](x)
\tfrac{1+\cvberror}
{1-\cvberror}
\geq
\rndcyc(x)
\geq \rndpf[r][n](x)
\tfrac{1-\cvberror}
{1+\cvberror}.
\]
As $\cycprocess(P)$ is $(\log n)$--neat with probability going to $1$ and $\cvberror d^{1/2} \to 0,$ the desired statement on the logarithms follow.  This reduces the problem to showing that  
\[
\frac{\log \rnd(P) -
\log \rndcyc(\cycprocess(P))
}{\sqrt{2/d}} \probto 0.
\]

We note that we have
\[
\Pr_{\pmd} \left[
\left| \rnd - 1 \right| \geq \tfrac12
\right] \leq 4\Var_{\pmd}[\rnd] \to 0
\]
by Proposition~\ref{prop:Variance}.  Likewise, we get
\[
\Pr_{\pmd} \left[
\left| \rndcyc \circ \cycprocess- 1 \right| \geq \tfrac12
\right] \leq 4 \Var_{\pmd}[\rndcyc \circ \cycprocess]
\leq 4 \Var_{\pmd}[\rnd] \to 0,
\]
by the contraction properties of the conditional expectation.  By the mean value theorem, we have that
\begin{align*}
\Pr_{\pmd} \left[ 
\left|
\log \rnd - \log \rndcyc \circ \cycprocess
\right| > t \sqrt{2/d} 
\right]
& \\
& \hspace{-1in}\leq
\Pr_{\pmd} \left[ 
2
\left|
\rnd - \rndcyc \circ \cycprocess
\right| > t\sqrt{2/d}
\right] + 8 \Var_{\pmd}[\rnd].
\end{align*}
Hence, by Lemma \ref{lem:rndl2}, we have that for all $k$,  $\Var_{\pmd}[ \rnd - \rndcyc \circ \cycprocess] = o(d^{-k})$, and hence this term goes to $0.$
\end{proof}

  \newcommand{\CNBW}[2][\infty]{\mathrm{CNBW}_{#2}^{(#1)}}
  \newcommand{\Cy}[2][\infty]{C_{#2}^{(#1)}}
  It remains to prove Theorem~\ref{thm:evalhams} on approximating
  the number of Hamiltonian cycles by a graph's eigenvalues.
  We give some combinatorial definitions.
A closed non-backtracking walk on a graph is a walk
that begins and ends at the same vertex, and that never follows
an edge and immediately follows that same edge backwards.
If the last step of a closed non-backtracking walk is anything other
than the reverse of the first step, we say that the walk is \emph{cyclically
  non-backtracking}.
Let $P\drawnfrom\pmd$, and 
let $\CNBW[n]{k}$ denote the number of closed cyclically
non-backtracking walks of length $k$ on the pseudograph given by $P$.

Let $T_i(x)$ be the Chebyshev polynomial of the first kind of degree $i$
on the interval $[-1,1]$.
We define a set of polynomials
\begin{align*}
    \Gamma_0(x)&=1,\\
    \Gamma_{2i}(x)&=2T_{2i}\Big(\frac{x}{2}\Big) + \frac{d-2}{(d-1)^i} \quad \text{for $i\geq 1$,}\\
    \Gamma_{2i+1}(x)&=2T_{2i+1}\Big(\frac{x}{2}\Big) \quad \text{for $i\geq 0$.}
\end{align*}
These polynomials allow us to count cyclically non-backtracking walks from a graph's
eigenvalues.
Let $\tr f(P)=\sum_{i=1}^nf(\lambda_i/\sqrt{d-1})$, where $\lambda_1,\ldots,\lambda_n$
are the eigenvalues of the adjacency matrix of $P$.
\begin{proposition}[Proposition~32 in \cite{DJPP}]
  \label{prop:CNBWpoly}
  \begin{align*}
    \tr \Gamma_k(P) &= (d-1)^{-k/2}\CNBW[n]{k}.
  \end{align*}
\end{proposition}

We will define another set of polynomials whose traces give
the cycle counts of $P$,
with high probability. 
Let $\mu$ be the M\"obius function, given by
  \begin{align*}
    \mu(n) = \begin{cases}
       1 & \text{if $n=1$,}\\
       (-1)^a&\text{if $n$ is the square-free product of $a$  primes,}\\
       0&\text{otherwise.}
    \end{cases}
  \end{align*}
For $k\geq 1$, define
\begin{align}
    \Xi_k(x) \Def \frac{1}{2k}\sum_{j\divides k}\mu
    \left(\frac{k}{j}\right)
      (d-1)^{j/2}\Gamma_j(x).\label{eq:fbasis}
  \end{align}
\begin{proposition}\label{prop:Xipoly}
  With probability at least $1-O(r(d-1)^r/n)$, 
  the number of cycles of length $k$ in $P$
  is $\tr\Xi_k(P)$ for all $1\leq k\leq r$.
\end{proposition}
\begin{proof}
  We will show that with high probability, all cyclically non-backtracking walks
  in $P$ are repeated walks around cycles. For this to fail, $P$ must contain
  cycles of length $k$ and $j$ at distance $l$ (possibly zero) from each other,
  with $k+j+2l\leq r$.
  
  By a slight variation of \eqref{eq:pmdoverlap} and \eqref{eq:pmdvertoverlap}, the
  probability that $P$ contains cycles of length $j$ and $k$ with $k+j\leq r$ that
  overlap is $O(r(d-1)^r / n)$. The number of possible edge-labeled subgraphs consisting
  of cycles of length $k$ and $j$ with a path of length $l\geq 1$ between them is at most
  \begin{align*}
    \ff{n}{j+k+l-1}(d-1)^{j+k+l+1}d^{j+k+l-1},
  \end{align*}
  and each subgraph is contained in $P$ with probability $1/\dff{nd}{j+k+l}$.
  By a union bound, $P$ contains some such subgraph with probability $O\big((d-1)^{j+k+l}/n\big)$.
  The sum of this over all $1\leq j,k\leq r$ and $l\geq 1$ satisfying
  $j+k+2l\leq r$  is $O(r(d-1)^{r-2}/n)$.
  
  Let $\Cy[n]{k}$ denote the number of cycles of length $k$ in $P$.
  If all cyclically non-backtracking walks are repeated walks around cycles, then
  \begin{align*}
    \CNBW[n]{k} = \sum_{j\divides k}2j\Cy[n]{k}.
  \end{align*}
  The proposition follows by applying the M\"obius inversion formula to write
  $\Cy[n]{k}$ in terms of $\CNBW[n]{j}$ for $j\divides k$, and then applying
  Proposition~\ref{prop:CNBWpoly}.
\end{proof}
We define one last polynomial:
\begin{align}\label{eq:Hampoly}
  \Hampoly(x) = \sum_k \left[\frac{1}{kn}+
    \log\Big(1-\frac{2}{(d-1)^k}\Big)\Xi_k(x)\right],
\end{align}
where the sum ranges over all odd values from $1$ to
$\floor{\log n / 3\log(d-1)}$. 
In $\Hampoly[3,n](x)$, the coefficient
of $\Xi_1(x)$  is $\log 0$. We interpret this as
$-\infty$, and we say that $\tr \Hampoly[3,n](P)=-\infty$ unless
$\tr \Xi_1(P)=0$.
\begin{lemma}\label{lem:Hampoly}
  Let $r = \floor{\log n / 3\log(d-1)}$.
  Suppose that $d-1\leq  n^{1/3}$, so that $r\geq 1$ and
  $\Hampoly$ is nonzero.
  For some absolute constants $\Cr{C:Hampolydist}$ and $\Cr{C:Hampolyprob}$,
  \begin{align*}
    \Pr\left[\abs{\rndcyc(\cycprocess(P)) - \exp(\tr\Hampoly(P))}  > \Cl{C:Hampolydist}
              (\log n)^{5/2}n^{-1/3}
    \right] < \Cl{C:Hampolyprob}n^{-1/3}.
  \end{align*}
\end{lemma}
\begin{proof}  
  Suppose that the event of
  Proposition~\ref{prop:Xipoly} holds, which occurs with probability
  $1-O\big(r(d-1)^{r}/n\big)$, and that
  $\cycprocess(P)$ is strictly $(\log n)$-neat, which occurs with
  probability $1-O\bigl((d-1)^{2r}/n\bigr)$ by Proposition~\ref{prop:strictlneat}.
  On this event, by~\eqref{eq:rndpf},
  \begin{align*}
    \rndpf(\cycprocess(P)) = \big(1+O(\log^2 n/n)\big)\exp(\tr \Hampoly(P)).
  \end{align*}
  By Proposition~\ref{prop:multpoiapprox}, 
  \begin{align*}
    \rndcyc(\cycprocess(P))=\big(1+O((\log n)^2n^{-1/3})\big)\rndpf(\cycprocess(P)).
  \end{align*}
  Now,
  \begin{align*}
    \abs{\rndcyc(\cycprocess(P))-\exp(\tr\Hampoly(P)} &= 
      \abs{\rndcyc(\cycprocess(P))-\big(1+O(\log^2 n/n)\big)\rndpf(\cycprocess(P))}\\
    &= \rndpf(\cycprocess(P)) O\big((\log n)^2n^{-1/3}\big)).
  \end{align*}
  As
  \begin{align*}
    \rndpf(\cycprocess(P)) &\leq \prod_{\substack{1\leq k\leq r\\\text{$k$ odd}}}
        e^{1/k}
        \leq \exp\left( \frac{\log r}{2} + O(1) \right)
        = O\big(\sqrt{\log n}\big),
  \end{align*}
  it holds that
  \begin{align*}
    \abs{\rndcyc(\cycprocess(P))-\exp(\tr\Hampoly(P)} &= O\big((\log n)^{5/2}n^{-1/3}\big)
  \end{align*}
  on an event which occurs with probability $1-O\bigl((d-1)^{2r}/n\bigr) = 1-O\bigl(n^{-1/3}\bigr)$.
\end{proof}
\begin{proof}[Proof of Theorem~\ref{thm:evalhams}]
  Let
    $r = \floor{\frac{\log n}{3\log(d-1)}}$.
  Since $r\geq 4$, we can apply
  Lemma~\ref{lem:rndl2} with $\alpha=5$ to get
  \begin{align*}
    \E \abs{\rnd(P) - \rndcyc(\cycprocess(P))}^2
      &= O\big(n^{-1/3} + (\log n)^2n^{-1/3}
        + n^{-7/20}\big)\\
      &= O\big((\log n)^2n^{-1/3}\big).
  \end{align*}
  By Chebyshev's inequality,
  \begin{align*}
    \P\bigl[\abs{\rnd(P) - \rndcyc(\cycprocess(P))} > n^{-1/12}
      -\Cr{C:Hampolydist}(\log n)^{5/2} n^{-1/3}\bigr] &=
     O\bigl((\log n)^2n^{-1/6}\bigr).
  \end{align*}
  This and Lemma~\ref{lem:Hampoly} combine to prove the theorem.
\end{proof}


\appendix
\section{Size-bias coupling of a $2$-state Markov chain}
\label{sec:markov}

Let $X_1, X_2, \ldots, X_n$ be $n$ steps of a stationary, reversible Markov chain on two states $\{0,1\},$ and let $\mu$ denote its stationary measure.  We let $p = \mu(\{1\}),$ and we let $\ccoeff$ be the contraction coefficient of the chain, which as this is a $2$-element space, is simply
\[
\ccoeff = \dtv
\left(
\lawof(X_2 \mid X_1 = 1),\,
\lawof(X_2 \mid X_1 = 0)
\right),
\]
where $\lawof$ denotes the law of a variable.
This regulates the optimal rate at which two chains with the same transition rule can be coupled in a Markovian fashion.  Let $Y_1,Y_2,\ldots,Y_n$ be an independent copy of $X_1, X_2, \ldots, X_n.$  Then we define
\[
\indcoeff = \Pr
\left[
X_2 \neq Y_2 \mid X_1 = 1, Y_1 = 0
\right].
\]
It is always the case that $\ccoeff \leq \indcoeff,$ and we will work in the case that $\indcoeff < 1.$  

Let $V_n$ denote the number of $X_i$ that are $1,$ so that $V_n = \sum_{i=1}^n X_i$. We will show a general construction for a size-bias coupling for $V_n$ and show some estimates for this coupling that can be used for normal approximation of $V_n.$
While we will not directly need this normal approximation, it follows immediately from the estimates that we do need, and so we state it as a result of possibly independent interest.
\begin{proposition}
\label{prop:markovclt}
With $V_n$ as above, with $0< p < 1,$ and with $\beta < 1$, there is a constant $C > 0$,
depending on the law of the Markov chain, so that
\[
\dw\left(
\frac{V_n - \Exp V_n}{\sqrt{\Var V_n}},
Z
\right)
\leq \frac{C}{\sqrt{n}},
\]
where $\dw$ denotes the Wasserstein $1$-distance (see Section 1.1.1 of~\cite{Ross}) and where $Z$ is a standard normal.
\end{proposition}

\subsection{Construction of the coupling}

Since $V_n$ is a sum of indicators with equal means, its size-bias distribution $V_n^s$ can be realized by choosing $I \sim \Unif([n])$ independently of the chain and defining
\[
V_n^s = 1 + \sum_{i \neq I} {Y_i}, 
\]
where the collection $\{Y_i\}_{i=1}^n$ has the distribution of $\{X_i\}_{i=1}^n$ conditioned on $X_I = 1$ (this follows directly from \cite[Corollary~3.24]{Ross}).  
This conditioning can be accomplished by defining $\{Z_i\}_{-\infty}^\infty$ to be a Markov chain independent of $\{X_i\}$ and with with its same transition probabilities, but started at $Z_0 = 1$.
As the chain is reversible, we can shift coordinates to obtain a chain $\{Z_{i-I}\}_{i=1}^{n}$ with same distribution as the original chain $\{X_i\}$ conditioned on $X_I=1$.

To couple the conditioned chain $\{Z_{i-I}\}$ back to $\{X_i\}$, we have it join back up
with $X_i$ at the first time $k$ before and after time $I$ that $X_k=Z_{k-I}$, taking
advantage of the reversibility of our chains. Formally, let
\newcommand{\tp}{\tau_+}
\newcommand{\tm}{\tau_-}
\begin{equation}\label{eq:tptm}
  \begin{split}
    \tp &= \inf\{k\geq 0 \colon X_{I+k} = Z_k\} \wedge (n-I+1),\\
    \tm &= \inf\{k\geq 0 \colon X_{I-k} = Z_k\} \wedge I,
  \end{split}
\end{equation}
and define
\begin{align}\label{eq:Ydef}
  Y_k = \begin{cases}
    Z_{k-I} & \text{if $\tm\leq k-I \leq \tp$,}\\
    X_k & \text{otherwise.}
  \end{cases}
\end{align}
The pair $(\{Y_i\}_1^n, \{X_i\}_1^n)$ is the desired coupling of the underlying state space, and setting $V_n^s = \sum_{i=1}^n Y_i$, we obtain a size-bias coupling $(V_n^s,V_n)$.

\subsection{Estimates}
To apply Stein's method, there are two quantities that need to be controlled.  Roughly, $V_n^s - V_n$ needs to be at constant order and $ \Exp \left[ V_n^s - V_n \;\middle\vert\; V_n \right]$ needs to be shrinking.  Bounding the first of the two is the more straightforward.  We will show that
\begin{proposition}
\label{prop:sbc_tail1}
For any $k \geq 1,$
\[
\Prob \big[
| V_n^s - V_n | \geq k
\big] \leq (1-p)k\indcoeff^{k-1},
\] 
and
\[
\Prob \left[ V_n^s = V_n \right] \geq p.
\]
\end{proposition}
\begin{proof}
Recalling the coupling times $\tp$ and $\tm$ defined in \eqref{eq:tptm}, the chains $X_i$
and $Y_i$ differ only for $I-\tm<i<I+\tp$.
Thus
\[
| V_n^s - V_n | \leq (\tp + \tm - 1) \vee 0.
\]
Each of these stopping times is $0$ if and only if $X_I = 1,$ and thus we have that
\[
\Prob \left[ V_n^s = V_n \right] \geq
\Prob \left[ X_I = 1\right] = 
 p
\]
by stationarity.

Regardless of $I$, the tails of each of these stopping times can be controlled by the constant $\indcoeff$, as for any non-negative integer $k$,
\[
\Pr \left[
\tp > k \mid X_I = 0
\right] \leq \indcoeff^k.
\]
More generally, as this is simple worst-case behavior, the same bound holds for $\tm$ and $\tp$ jointly in that for any non-negative integers $k$ and $l,$
\[
\Pr \left[
\tp > k
~\text{and}~
\tm > l
\mid X_I = 0
\right] \leq \indcoeff^{k+l}.
\]
We can now sum this bound to conclude that
\begin{align*}
\Pr \big[
|V_n^s - V_n| \geq k
\;\big|\; X_I = 0
\big]
&\leq \Pr[\tp+\tm-1\geq k]\\
&\leq 
\sum_{l=0}^{k-1} 
\Pr \left[
\tp > k-1-l
~\text{and}~
\tm > l
\mid X_I = 0
\right]  \\
&\leq k\indcoeff^{k-1}.\qedhere
\end{align*}
\end{proof}

It remains to control the conditional expectation
\[
\Exp \left[ V_n^s - V_n \;\middle\vert\; V_n \right].
\]
An estimate of the variance of this expression would suffice for the usual application of Stein's method
(see \cite[Theorem~3.20]{Ross}), 
but this will not quite be sufficient for our purposes, as we will need some higher moments.  We will use a functional equality to control the deviations of this expression.   We recall that the Hamming distance on $\{0,1\}^n$ is given by the minimum number of coordinate changes required to change one string into another; in this case, it coincides with the $L^1$ distance.  A function $f \colon \{0,1\}^n \to \R$ is called $K$-Lipschitz if 
\[
|f(x) - f(y)| \leq K \|x-y\|_1
\]
for all $x$ and $y$ in $\{0,1\}^n.$  

We define
\begin{align*}
  F(x_1,\ldots,x_n) \Def \E[V_n^s - V_n \mid X_1=x_1,\ldots,X_n=x_n],
\end{align*}
so that
\begin{align*}
  \Exp [ V_n^s - V_n \mid X_1,\dots,X_n] = F(X_1,\dots,X_n).
\end{align*}
On its face,  $F(x_1,\ldots,x_n)$ is undefined
when $\P[(X_1,\ldots,X_n)=(x_1,\ldots,x_n)]=0$. As
the chain $Y_1,\ldots,Y_n$ can still be defined under the assumption
$(X_1,\ldots,X_n)=(x_1,\ldots,x_n)$, and $V_n$ can be interpreted as $\sum x_i$, we will
take $F$ to be defined everywhere and turn to estimating its Lipschitz constant.
Let $P$ be the transition matrix of the chain $X_1,\ldots,X_n$.
Our assumptions that the chain is reversible and that $\beta<1$
imply that either all entries of $P$ are less than one, or all entries
of $P^2$ are. Let $\delta$ equal $1$ in the first case and $2$ in the second,
and let $\gamma$ be the maximum entry of $P^\delta$.
\begin{proposition}
\label{prop:sbc_lipschitz_constant}
The function $F$ is  Lipschitz with constant
$(1-\gamma+2\delta)\delta/(1-\gamma)^2n$.
\end{proposition}
\begin{proof}
  Fix some $x_1,\ldots,x_n$ and some $1\leq j\leq n$, and let $x_j'=1-x_j$.
To simplify notation, we assume for this proof that all random variables
defined previously, such as $Y_1,\ldots,Y_n$, $V_n^s$, $V_n$, are distributed
conditional on $X_1=x_1,\ldots,X_n=x_n$.
Define $Y_1',\ldots,Y_n'$ as in \eqref{eq:Ydef}, using the same random index~$I$ and the same Markov chain $\{Z_i\}_{-\infty}^{\infty}$, but conditional on
\begin{align*}
  X_j &= x_j',\\
  X_i&= x_i \qquad\text{for $i\neq j$.}
\end{align*}
This defines a coupling of the two conditional expectations $F(x_1,\ldots,x_n)$
and $F(x_1,\ldots,x_{j-1},1-x_j,x_{j+1},\ldots,x_n)$.
Define \begin{align*}
  W_n^s &= \sum_{i=1}^n Y'_i,\qquad\qquad
  W_n=\sum_{i=1}^n x_i - x_j + x_j'.
\end{align*}
We now have
\begin{align*}
  F(x_1,\ldots,x_n) - F(x_1,\ldots,x_{j-1},x_j',x_{j+1},\ldots,x_n)
    = \E[V_n^s- V_n - (W_n^s - W_n)].
\end{align*}
Define $\tp'$ and $\tm'$ analogously to $\tau_+$ and $\tau_-$.  
The process $Y_\omitted$ matches up with $Z_{\omitted-I}$ from time $I-\tm$ to $I+\tp$,
and with its base sequence $x_1,\ldots,x_n$ outside of those times.
Similarly, $Y'_\omitted$ matches up with $Z_{\omitted-I}$ from time $I-\tm'$ to $I+\tp'$,
and with its base sequence $x_1,\ldots, x_j',\ldots,x_n$ outside of those times.
If $j\notin \{I-\tm,\ldots,I+\tp\}$, then $\tau_{\pm}=\tau_{\pm}'$, and
$Y_\omitted$ and $Y'_\omitted$ are equal to each other from $I-\tm$ to $I+\tp$
and to their base sequences outside of those times. It thus holds that
$V_n^s - V_n = W_n^s - W_n$.
Furthermore, this is what usually occurs, as
by direct calculation,
\begin{align*}
  \Pr[j \in \{I-\tau_-,\ldots,I+\tau_+\}]
    &=\frac{1}{n}\sum_{i=1}^n
        \Pr[j \in \{I-\tau_-,\ldots,I+\tau_+\}\mid I=i]\\
    &\leq \frac{1}{n}\sum_{i=1}^n \left(\gamma^{\floor{\abs{j-i}-1/\delta}}\wedge 1\right)\\
    &\leq \frac{1}{n}\left( 1 + 2\sum_{i=0}^{\infty}\gamma^{\floor{i/\delta}} \right)
     = \frac{1-\gamma+2\delta}{(1-\gamma)n}.
\end{align*}

If $j\in\{I-\tau_-,\ldots,I+\tau_+\}$, then one of the two chains
$Y_\omitted$ and $Y_\omitted'$ rejoins its base sequence at time $j$, and the other
does not. Let $\upsilon$ be the number of extra steps it takes for this chain
to rejoin its base sequence. Then $V_n^s-V_n$ and $W_n^s-W_n$ differ by at most
$\upsilon$. We bound the probability that $\upsilon$ is large:
\begin{align*}
  \Pr\big[\upsilon > k \mid j\in\{I-\tau_-,\ldots,I+\tau_+\}\big] &\leq \gamma^{\floor{k/\delta}}.
\end{align*}
It follows that
\begin{align*}
  \E\big[\upsilon\mid j\in\{I-\tau_-,\ldots,I+\tau_+\}\big]
    &= \sum_{k=0}^{\infty}
      \Pr\big[\upsilon > k \mid j\in\{I-\tau_-,\ldots,I+\tau_+\}\big]\\
    &\leq \sum_{k=0}^{\infty}\gamma^{\floor{k/\delta}} = \frac{\delta}{1-\gamma}.
\end{align*}
%
%
Thus,
\begin{align*}
  \abs{\E[V_n^s- V_n - (W_n^s - W_n)]} \hspace{-2.5cm}&\hspace{2.5cm}\\
  &\leq \frac{1-\gamma+2\delta}{(1-\gamma)n}
   \E\abs{V_n^s- V_n - (W_n^s - W_n)\mid j\in\{I-\tau_-,\ldots,I+\tau_+\}}
     \\
   &\leq \frac{(1-\gamma+2\delta)\delta}{(1-\gamma)^2n},
\end{align*}
thus showing that $F$ is Lipschitz with the above constant.
\end{proof}
The advantage of knowing that this function is Lipschitz is that we immediately get strong concentration in terms of the contraction coefficient of the chain.

\begin{proposition}
\label{prop:sbc_generalgaussiantail}
For a $K$-Lipschitz function $F$ of $\{0,1\}^n$
\[
\Pr \left(
\left| F(\X) - \Exp F(\X)\right| \geq t
\right) \leq 
2 \exp \left(
-\frac{t^2(1-\ccoeff)^2}{2nK^2}
\right),
\]
where $\X = (X_1, X_2,\ldots, X_n).$
\end{proposition}
\begin{proof}
See Theorem~1.2 of~\cite{Kontorovich} and the paragraph following it.
\end{proof}
\begin{corollary}
\label{cor:sbc_tail2}
For any $t \geq 0,$
\[
\Pr \left(
\left|
\Exp\left[ V_n^s - V_n \right]
-
\Exp \left[V_n^s - V_n \;\middle\vert\; \X\right]
\right| \geq \frac{t}{\sqrt{n}}
\right) \leq 2 
\exp \left(
-\frac{t^2(1-\gamma)^4(1-\theta)^2}{2(1-\gamma+2\delta)^2\delta^2}
\right)
\]
\end{corollary}
\begin{proof}
This is simply Proposition~\ref{prop:sbc_lipschitz_constant}, Proposition~\ref{prop:sbc_generalgaussiantail} and a change of variables. 
\end{proof}

\begin{proposition}
\label{prop:sbc_expectation}
 Let $\lambda$ be the second
  eigenvalue of the Markov kernel, which is given by
  \begin{align}
    \lambda=\Pr[X_1=1\mid X_0=1]
     -\Pr[X_1=1\mid X_0=0]. \label{eq:lambda}
  \end{align}
  (Note that $\abs{\lambda}=\ccoeff$.)  Then
  \begin{align*}
    \E[V_n^s - V_n] &= \frac{(1-p)(1+\lambda)}{1-\lambda}
    -\frac{2(1-p)\lambda(1-\lambda^n)}{(1-\lambda)^2n}.
  \end{align*}
\end{proposition}
\begin{proof}
  \newcommand{\tph}{{\hat{\tau}_{+}}}
  \newcommand{\tmh}{{\hat{\tau}_{-}}}
  Let $\tph=\tp\wedge (n-I)$ and $\tmh=\tm\wedge (I-1)$. This only makes a difference
  when $Z_{\omitted-I}$ never rejoins $X_{\omitted}$, and $\tp=n-I+1$ or $\tm=I$.
  We first use symmetry so that we can ignore one of $\tph$ and $\tmh$:
  \begin{align}
    \E[V_n^s - V_n] &= \E\left[\sum_{i=-\tmh}^{\tph}
                       (Z_i - X_{I+i}) \right]\nonumber\\
       &= \E\left[ \sum_{i=0}^{\tph} (Z_i - X_{I+i}) + 
                   \sum_{i=0}^{\tmh} (Z_{-i} - X_{I-i}) - Z_0 + X_I \right]
                   \nonumber\\
       &= 2\E\left[ \sum_{i=0}^{\tph} (Z_i - X_{I+i}) \right] - 1 + p.
       \label{eq:symreduce}
  \end{align}
  
  Define
  \begin{align*}
    M_i &\Def \sum_{j=0}^i(Z_j - X_{I+j}) - \sum_{j=0}^{i-1}\E[Z_{j+1} - X_{I+j+1}
           \mid \chainfield_j],
  \end{align*}
  where $\chainfield_j=\sigma(I,\, X_I,\,\ldots,\,X_{I+j},\,Z_0,\ldots,\,Z_j)$.    The process $(M_i,\,i\geq 0)$ is a martingale with respect to the filtration $(\chainfield_i,\,i\geq 0)$.
  By explicitly computing this conditional expectation, we see that
  \begin{align*}
    M_i &= (1 - \lambda)\sum_{j=0}^{i-1}(Z_j - X_{I+j}) + Z_i - X_{I+i}.
  \end{align*}

  Since $\tph$ is a bounded stopping time with respect to this filtration,
  \begin{align*}
    1 - p = \E M_0 = \E M_{\tph}
       &= (1 -\lambda)\E\left[\sum_{i=0}^{\tph-1}
          (Z_i - X_{I+i})\right] + \E\left[Z_{\tph} - X_{I+\tph}\right]\\
      &=(1 -\lambda)\E\left[\sum_{i=0}^{\tph}
          (Z_i - X_{I+i})\right] + \lambda
            \E\left[Z_{\tph} - X_{I+\tph}\right]
  \end{align*}
  by the optional stopping theorem.  Thus,
  \begin{align}
    \E\left[\sum_{i=0}^{\tph}
          (Z_i - X_{I+i})\right] &= \frac{1-p - \lambda
            \E\left[Z_{\tph} - 
             X_{I+\tph}\right]}{1 - \lambda}.\label{eq:ost}
  \end{align}
  All that remains is to determine $\E[Z_{\tph} - X_{I+\tph}]$.
  The expression $Z_{\tph} - X_{I+\tph}$ is zero unless
  the Markov chains $Z_{\omitted}$ and $X_{I+\omitted}$ never match up.
  So,
  \begin{align*}
    \E[Z_{\tph} - X_{I+\tph}] &= \E\left[ \1_{\{Z_0\neq X_I,\ldots,
      Z_{n-I}\neq X_n\}}(Z_{n-I}-X_n)\right].
  \end{align*}
  Averaging over the possible choices of $I$ gives
  \begin{align*}
    \E[Z_{\tph} - X_{I+\tph}] 
      &= \frac{1}{n}\sum_{i=1}^{n}\E\left[ \1_{\{Z_0\neq X_i,\ldots,
      Z_{n-i}\neq X_n\}}(Z_{n-i}-X_n)\right]\\
      &= \frac{1}{n}\sum_{i=0}^{n-1}\E\left[ \1_{\{Z_0\neq X_0,\ldots,
      Z_{i}\neq X_{i}\}}(Z_{i}-X_{i})\right].
  \end{align*}
  In the last step, we shifted $X_{\omitted}$ and replaced $i$ with $n-i$
  to make the indices
  easier to deal with; the step is justified because $X_{\omitted}$ is 
  stationary and independent
  of $Z_{\omitted}$.
  Next, we take an inductive approach. By the Markov property,
  \begin{align*}
    &\E\left[ \1_{\{Z_0\neq X_0,\ldots,
      Z_{i}\neq X_{i}\}}(Z_{i}-X_{i}) \,\middle\vert\,
      X_0,\ldots,X_{i-1},Z_0,\ldots,Z_{i-1}\right]\\
      &\qquad\qquad= \1_{\{Z_0\neq X_0,\,\ldots,\,
       Z_{i-1}\neq X_{i-1}\}}\E[ Z_i - X_i \mid X_{i-1}, Z_{i-1}]\\
      &\qquad\qquad= \1_{\{Z_0\neq X_0,\,\ldots,\,
       Z_{i-1}\neq X_{i-1}\}}\lambda(Z_{i-1} - X_{i-1}).
  \end{align*}
  Taking expectations,
  \begin{align*}
    &\E\left[ \1_{\{Z_0\neq X_0,\ldots,
      Z_{i}\neq X_{i}\}}(Z_{i}-X_{i})\right]\\
       &\qquad\qquad= \lambda\E\left[ \1_{\{Z_0\neq X_0,\ldots,
      Z_{i-1}\neq X_{i-1}\}}(Z_{i-1}-X_{i-1})\right]\\
      &\qquad\qquad= \lambda^i(1-p).
  \end{align*}
  Thus
  \begin{align*}
    \E[Z_{\tph} - X_{I+\tph}] 
      &= \frac{1-p}{n}\sum_{i=0}^{n-1}\lambda^i\\
      &= \frac{(1-p)(1-\lambda^n)}{n(1-\lambda)}.
  \end{align*}
  Substituting this into \eqref{eq:symreduce} and \eqref{eq:ost}
  completes the proof.
  \end{proof}

Finally, we prove the quantitative Markov central limit theorem by combining these facts, which we emphasize is not needed for the main results of this paper.

\begin{proof}[Proof of Proposition~\ref{prop:markovclt}]
Let $W = (V_n - \Exp V_n)/\sqrt{\Var{V_n}}$ and let $Z$ be a standard normal.
Our starting point is the standard Stein's method through size-bias coupling lemma (see Theorem~3.20 of~\cite{Ross}), which states that
\[
\dw( W, Z) \leq \frac{\Exp V_n}{\Var V_n} \sqrt{ \frac{2}{\pi}}
\sqrt{ \Var \left(
\Exp \left[
V_n^s - V_n \mid V_n
\right]
\right)} + 
\frac{\Exp V_n}{(\Var V_n)^{3/2}}
\Exp\left[
\left(V_n^s - V_n\right)^2
\right].
\]
We will need to know that the standard deviation of $V_n$ is $\Theta( \sqrt{n}).$  (In this proof,
the implicit constants in asymptotic expressions should be understood to depend on the law of the Markov
chain.) On the one hand, we know that $\Exp V_n = n p = \Theta( n )$ by stationarity.  On the other hand, from the definition of the size-bias distribution that
\[
\Exp \left[
V_n^s - V_n 
\right] = \frac{\Var V_n}{\Exp V_n}.
\]
By hypothesis on $\mu$ and $\beta,$ it follows from Proposition~\ref{prop:sbc_expectation} that $\frac{\Var V_n}{\Exp V_n} = \Theta(1),$ and hence $\Var V_n = \Theta(n)$.  
Meanwhile, by Jensen's inequality and Corollary~\ref{cor:sbc_tail2},
\begin{align*}
  \Var (
\Exp [
V_n^s - V_n\mid V_n
]
)&\leq \Var(\Exp [
V_n^s - V_n\mid \X
])=O(1/n),
\end{align*}
and by Proposition~\ref{prop:sbc_tail1},
\[
\Exp\left[
\left(V_n^s - V_n\right)^2
\right] = O(1).\qedhere
\] 
\end{proof}

\bibliographystyle{alpha}
\bibliography{QH}

\end{document}